\documentclass[11pt,reqno]{amsart}
\usepackage{amsmath,amsthm,amssymb}
\usepackage{mathrsfs}
\usepackage{enumitem}
\usepackage{pgf,tikz,pgfplots}
\usetikzlibrary{arrows}
\usepackage[unicode]{hyperref}
\hypersetup{colorlinks=true, linkcolor=blue, citecolor=blue, filecolor=blue, urlcolor=blue}
\usepackage{comment}
\usepackage{chngcntr}
\usepackage{mathabx}
\usepackage{bm}
\usepackage{etoolbox}
\usepackage{algpseudocode}
\usepackage{enumitem}
\usepackage{algorithm}
\usepackage{subcaption}
\usepackage{float}
\usepackage{pkgfile}
\usepackage{bm}
\usepackage{extarrows}
\usepackage{xcolor}
\usepackage{algorithm}
\usepackage{algpseudocode}
\usepackage{enumitem}
\usepackage{booktabs}
\usepackage{multirow}
\usepackage{threeparttable}
\usepackage{caption}
\usepackage{mathtools}
\usepackage{algorithm}
\usepackage{lipsum} 

\usepackage[comma,sort&compress, numbers]{natbib}

\allowdisplaybreaks

\newtheorem{theorem}{Theorem}[section]
 
\newtheorem{lemma}[theorem]{Lemma} 
\newtheorem{proposition}[theorem]{Proposition} 
  
\theoremstyle{definition}
\newtheorem{definition}[theorem]{Definition} 
\newtheorem{remark}[theorem]{Remark} 
 
\newtheorem{assumption}{Assumption}

\newcommand{\indep}{\perp \!\!\! \perp}

\def\dCov{\mathrm{dCov}}
\def\JdCov{\mathrm{JdCov}}
\def\RJdCov{\mathrm{RJdCov}}
\def\RdCov{\mathrm{RdCov}}

\def\i{\iota}
\def\d{\mathrm{d}}

\def\SO{\mathcal{SO}}

\def\R{\hat{R}}

\def\o{\mathrm{oracle}}

\title[Distribution-Free Joint Independence Testing and ICA]{Distribution-Free Joint Independence Testing and Robust Independent Component Analysis Using Optimal Transport}

\author[Niu]{Ziang Niu}
\address{Applied Mathematics and Computational Science, University of Pennsylvania, Philadelphia, USA}
\email{ziangniu@sas.upenn.edu}

\author[Bhattacharya]{Bhaswar B. Bhattacharya}
\address{Department of Statistics and Data Science, University of Pennsylvania, Philadelphia, USA}
\email{bhaswar@wharton.upenn.edu}


\begin{document}
	
	\maketitle
	
	\begin{abstract} In this paper, we study the problem of measuring and testing joint independence for a collection of multivariate random variables. Using the emerging theory of optimal transport (OT)-based multivariate ranks, we propose a distribution-free test for multivariate joint independence. Towards this we introduce the notion of {\it rank joint distance covariance} $(\RJdCov)$, the higher-order rank analogue of the celebrated distance covariance measure, which can capture the dependencies among all the subsets of the variables. The $\RJdCov$ can be easily estimated from the data without any moment assumptions and the associated test for joint independence is universally consistent. We derive the asymptotic null distribution of the $\RJdCov$ estimate, using which we can readily calibrate the test without any knowledge of the (unknown) marginal distributions (due to the distribution-free property). We also provide an efficient data-agnostic resampling-based implementation of the test which controls Type-I error in finite samples and is consistent with only a fixed number of resamples. In addition to being distribution-free and universally consistent, the proposed test is also statistically efficient, that is, it has non-trivial asymptotic (Pitman) efficiency (power against $1/\sqrt n$ alternatives). We demonstrate this by computing the limiting local power of the test for both mixture alternatives and joint Konijn alternatives. We then use the $\RJdCov$ measure to develop a new method for independent component analysis (ICA) that is easy to implement and robust to outliers and contamination. Extensive simulations are performed to illustrate the efficacy of the proposed test in comparison to other existing methods. Finally, we apply the proposed method to learn the higher-order dependence structure among different US industries based on  stock prices. As a byproduct of our theoretical analysis, we develop a version of Hoeffding's classical combinatorial central theorem for mutiple independent permutations and a multivariate H\'ajek representation result for joint rank statistics, which might be of independent interest. 
	\end{abstract}

	\section{Introduction}
	
	Given $r$ probability distributions on $\mathbb R^{d_1}, \mathbb R^{d_2}, \ldots, \mathbb R^{d_r}$ with marginal laws $\mu_1, \mu_2, \ldots, \mu_r$, respectively, and joint law $\mu$ on $\mathbb  R^{d}$, where $d=\sum_{s=1}^r d_s$, the {\it mutual independence testing} problem is given by
	\begin{align}\label{eq:independence}
	H_0:\mu=\mu_1\otimes\mu_2\otimes\cdots\otimes\mu_r\qquad\text{versus}\qquad H_1:\mu\neq\mu_1\otimes\mu_2\otimes\cdots\otimes\mu_r , 
	\end{align}
	where $\mu_1\otimes\mu_2\otimes\cdots\otimes\mu_r$ is the product of the marginal distributions $\mu_1, \mu_2, \ldots,\mu_r$. This has been extensively studied in the case $r=2$, which is the classical {\it pairwise independence testing} problem for a collection of random variables. For univariate distributions, that is, $d_1=d_2=1$, nonparametric tests for pairwise independence begin with the celebrated results of~\citet{hoeffding} and~\citet{blum1961distribution}. These tests are {\it distribution-free}, that is, the distributions of the test statistics under $H_0$ do not depend on the (unknown) marginal distributions, and consistent against a general class of alternatives. 	 Since then a slew of methods for measuring and testing univariate pairwise independence have been proposed. These include, among several others, results of~\citet{yanagimoto1970measures},~\citet{feuerverger1993consistent},~\citet{bergsma2014consistent},~\citet{signcovariance},~\citet{heller2016consistent} and the recent breakthrough of Chatterjee~\cite{chatterjeecorrelation,chatterjee2022survey}.

	Pairwise independence testing is more challenging for dimensions greater than 1, that is, either $d_1 > 1$ or $d_2 > 1$, because of the lack of a natural ordering in multivariate data. One of the most popular methods for pairwise multivariate independence testing is the {\it celebrated distance covariance} ($\dCov$) of   \citet{szekely2007measuring} and \citet{szekely2009brownian}, which measures dependence between 2 random vectors based on the difference of their joint and marginal characteristic functions and can be estimated using the correlation among the pairwise distances. The distance covariance is zero if and only if the null hypothesis of independence holds, provided $\mu_1$ and $\mu_2$ have finite first moments. Another approach is based on the Hilbert-Schmidt independence criterion (HSIC), which also provides a consistent kernel-based test for pairwise independence (see \citet{gretton2005kernel,gretton2005measuring,gretton2007kernel}). \citet{Sejdinovic2013} showed that  distance covariance and the kernel-based independence criterion are, in fact, equivalent if the kernel is chosen based on the relevant distance function. Other popular methods include mutual information-based tests \cite{ai2022testing,berrett2019nonparametric,kinney2014equitability,wu2009smoothed}, graph-based methods \cite{deb2020measuring,heller2012class,friedman1983graph}, the maximal information coefficient \cite{novelassociation}, ranking of interpoint distances \cite{heller2013consistent,moon2022interpoint}, ball covariance \cite{pan2019ball}, and binning approaches based on partitions of the sample space \cite{heller2013consistent,scanning2019fisher,zhang2019bet} (see \cite{JH16,independence2022evaluating} for reviews of the various methods). However, none of the aforementioned multivariate tests simultaneously inherit the distribution-free and universal consistency properties of the rank-based univariate tests. A breakthrough in this direction was made recently by \citet{deb2021multivariate}  and \citet{shi2020distribution} using optimal transport (OT) based multivariate ranks \cite{chernozhukov2017monge,delbarrio2019,Hallin2017}. The OT-based  tests for pairwise independence are distribution-free in finite samples, computationally feasible, universally consistent (that is, the power of the tests converge to 1 as the sample size increases), and enjoy attractive efficiency properties \cite{shi2020consistent} (see also \cite{deb2021efficiency,scanning2019fisher}). 

In this paper we consider the mutual independence testing problem \eqref{eq:independence} for more than 2 variables. Measuring and testing higher-order dependence\footnote{We say a collection of random variables $X_1,\ldots,X_r$ has higher-order dependency if they are pairwise independent but not jointly independent.} and understanding the dependence structure between the variables are much more involved tasks than detecting pairwise dependence. Examples of such higher-order dependencies abound in the literature. This includes, among others, applications in diagnostic checking for directed acyclic graphs (DAGs), where the noise variables are assumed to be jointly independent, hence inferring pairwise independence is not enough, and independent component analysis (ICA), which entails finding a suitable transformation of multivariate data with mutually independent components (see Section \ref{sec:ICA} for details). In fact, \citet{bottcher2020dependence} collects over 350 datasets featuring statistically significant higher-order dependencies. While there are a few results for mutual independence testing in the univariate case, where $d_1=d_2=\cdots=d_r = 1$ (see \cite{distributionnonparametric,independenceempiricalcharacteristic,hoeffding2010multivariate,kankainen1995consistent,poczos2012copula}  and the references therein), in higher dimensions the problem is much more challenging. Towards this, \citet{pfister2018kernel} generalized the HSIC to multiple variables and obtained kernel-based tests joint independence (see also \citet{sejdinovic2013kernel} for a kernel-based 3-variable interaction test). Recently, \citet{bottcher2019distance} and \citet{chakraborty2019distance} proposed higher-order generalizations of the distance covariance, which provide consistent tests for multivariate joint independence. This is referred to as the {\it total distance multivariance} by \citet{bottcher2019distance} or the {\it joint distance covariance}  ($\JdCov$) by \citet{chakraborty2019distance}. One of the attractive properties of the distance multivariance/$\JdCov$ is that it has a hierarchical structure that can capture the dependence among any subset of the variables. Recently, \citet{consistentindependence} proposed tests for joint independence based on pairwise distances and linear projections (see also \citet{banerjeetest} for generalizations to functional data). However, none of the aforementioned methods possesses the distribution-free property  of the univariate methods. In fact, these tests usually have intractable (non-Gaussian) asymptotic distributions under $H_0$ that depend on the (unknown) marginal distributions and are generally calibrated using permutation methods.

	In this paper we propose a distribution-free test for joint independence using the framework of OT-based multivariate ranks. Towards this, we introduce the  notion of {\it rank joint distance covariance} $(\RJdCov)$, the rank analogue of the  $\JdCov$, which is obtained by aggregating the (higher-order) rank distance covariances ($\RdCov$) over all the subsets of variables (Section \ref{subsec:rank-based jdcov}). The higher-order $\RdCov$ characterizes the mutual independence of a subset of variables given their lower-order independence and the $\RJdCov$ characterizes the total mutual independence of all the $r$ variables (Proposition \ref{ppn:RJdCovindependence}). The $\RdCov$ and, hence, the $\RJdCov$ can be consistently estimated from data without any moment assumptions (Theorem \ref{thm:consistency}).  Consequently, since OT-based multivariate ranks are themselves distribution-free (see Section~\ref{subsec:multivariate_rank}), we can construct a distribution-free test for multivariate joint independence based on the estimated $\RJdCov$. The proposed test has the following properties:

\begin{itemize}	 
	 

\item {\it Distribution-free, universally consistent tests}: Since $\RJdCov$ is well-defined whenever the marginal distributions $\mu_1,  \mu_2, \ldots, \mu_r$ are  absolutely continuous and is zero if and only if the null hypothesis of joint independence holds, our proposed test is consistent whenever the joint distribution has absolutely continuous marginals (Proposition \ref{ppn:phiconsistency}). Moreover, we can use the $\RdCov$ estimates to obtain distribution-free, universally consistent tests for higher-order independence of a subset of variables given their lower-independence.  For example, if a collection variables are pairwise independent, the $\RJdCov$ measure can be used to construct distribution-free, consistent tests for 3-way or higher-order dependence (see Remark \ref{remark:RdCov} for a discussion and Section \ref{sec:realdata} for an application in real data).

\item {\it Asymptotic null distribution}: In Section \ref{sec:theory} we derive the limiting distribution of the proposed test under the null hypothesis of joint independence. In fact, we derive the limiting joint distribution of the $\RdCov$ estimates for all the subsets of $r$ variables, which can be expressed as a squared integral of a certain Gaussian process or, equivalently, as an infinite weighted sum of independent chi-squared distributions (Theorem \ref{thm:H0}). One remarkable property of the limiting distribution is that the distributions of the $\RdCov$ estimates across the various subsets are asymptotically independent, thereby facilitating the hierarchical testing for higher-order dependence mentioned above.  Since the OT-based multivariate ranks are distributed as uniform random permutations, the $\RdCov$ estimates can be expressed as a combinatorial sum indexed by multiple (more than 1) independent permutations. For this we use Stein's method based on exchangeable pairs to develop  a version of Hoeffiding's classical combinatorial central limit theorem for multi-dimensional arrays (tensors) indexed by independent permutations (see Theorem \ref{thm:clt} in Appendix \ref{sec:clt}), a result which might be of independent interest. The distribution-free property implies that the asymptotic null distribution of the test statistic does not depend on the distribution of the data generating mechanism, hence, can be readily used to calibrate the test statistic.

	\item {\it Finite sample properties}: The distribution-free property also allows us to approximate the quantiles of the null distribution in finite samples  using a data-agnostic resampling technique. Consequently, we can calibrate the test statistic to have finite sample level $\alpha$ (Section \ref{sec:finitesample}). Interestingly, this finite sample implementation is consistent with only a finite number of resamples (Theorem \ref{thm:consistency_permutation}). Hence, we can implement the proposed test both asymptotically and in finite samples efficiently, without having to estimate any nuisance parameters of the marginal distributions or use computationally intensive data-dependent permutation techniques to approximate the rejection thresholds.

	\item {\it Asymptotic efficiency}: 	In Section \ref{sec:local_power} we establish the asymptotic (Pitman) efficiency of the proposed test by computing its limiting power for local contiguous alternatives (alternatives shrinking towards $H_0$ at rate $O(1/\sqrt n)$) \cite{vdv2000asymptotic}. Specifically, we derive the asymptotic distribution of the proposed test for two kinds of local alternatives: mixture alternatives (Theorem \ref{thm:mixture}) and joint Konijn alternatives (Theorem \ref{thm:powerK}). To the best of our knowledge, these are the first results on the efficiency properties of  nonparametric joint independence tests. This implies that the proposed test, in addition to being distribution-free, universally consistent, and computationally feasible, is also statistically efficient, making it particularly attractive for modern data applications.  The proof of the local power analysis is based on a Haj\'{e}k representation result for the $\RdCov$ estimates (Proposition \ref{ppn:ZnS}), which allows us to replace the empirical rank maps with their population counterparts without altering the limiting distribution of the test statistic under $H_0$.

\end{itemize} 

Going beyond hypothesis testing, the $\RJdCov$ can be used more generally to quantify deviations from joint independence. Specifically, in Section \ref{sec:ICA} using the $\RJdCov$ measure as an objective function we develop a new method for the independent component analysis (ICA) problem that is robust to outliers and contamination. ICA is a method for extracting independent components from multivariate data that emerged from research in artificial neural networks and has found applications in blind source separation, feature extraction, computational biology, and time series analysis (see \cite{hyvarinen2000independent} for a review). Our estimator, which is obtained by minimizing the $\RJdCov$ measure, can be computed efficiently and is consistent for the independent components (Theorem \ref{thm:independentcomponent}). We illustrate the effectiveness of the proposed ICA estimator with the approach of \citet{matteson2017independent} based on the $\dCov$ in various simulation settings (Section \ref{sec:ICAsimulation}).  

In Section \ref{sec:sim}  we compare the performance of the proposed test with other popular tests for joint independence, such the dHSIC \cite{pfister2018kernel}, $\JdCov$ \cite{chakraborty2019distance}, and the $\dCov$ based measure in \cite{matteson2017independent}. Our test performs well across a variety of data distributions and is especially powerful compared to the existing tests in contamination models and heavy-tailed distributions. 

Finally, in Section \ref{sec:realdata} we apply our method to a dataset consisting of stock prices of different companies in USA. Our approach sheds more insights into the structure of the higher-order dependencies in this data, producing results that are more interpretable than existing methods. All the implementation of tests can be found in \href{https://github.com/ZiangNiu6/Distribution-free-mutual-independence-test}{GitHub}.

	\section{Rank-Based Joint Independence Measures}

	In this section, we define a rank-based measure of joint independence by combining higher-order distance covariances with optimal transportation of measures. We review the relevant concepts regarding higher-order distance covariance and joint dependence measures in Section \ref{subsec:RjdCov} and OT-based multivariate ranks in Section \ref{subsec:multivariate_rank}. The rank-based joint dependence measures are introduced in Section \ref{subsec:rank-based jdcov} and their estimation is discussed in Section \ref{subsec:rank-based jdcov est}. We begin by introducing some notations.

	\subsection{Notations} 
	
	Fix $r \geq 2$ and suppose $\bm X= (X_1,\ldots,X_r)$ is a random vector, where each $X_i$ is a random variable taking values in $\mathbb{R}^{d_i}$, for $1 \leq i \leq r$ and $d_0 = \sum_{i=1}^r d_i$. The characteristic function of $X_i$ will be denoted by 
	$$\phi_{X_i}(t)=\E\left[e^{\i\langle t,X_i\rangle}\right] , $$ 
	for $t\in\mathbb{R}^{d_i}$. 
	
	\begin{definition}\label{defn:r} The random vector $\bm X=(X_1, \ldots, X_r)$ is said to be $q$-independent (for some $2 \leq q \leq r$) if for any sub-family $\{i_1, i_2, \ldots, i_q\} \subset \{1, 2, \ldots, r\}$ the random variables $X_{i_1}, X_{i_2}, \ldots, X_{i_q}$ are mutually independent. 
	\end{definition}
	
	Throughout we will denote 
	$$w_{d_i}(t) := \frac{1}{c_{d_i}\|t\|^{1+d_i}} \quad \text{where } c_{d_i} := \frac{\pi^{\frac{1+d_i}{2}}}{\Gamma\left( \frac{1+d_i}{2} \right) },$$ 
	for $1 \leq i \leq r$. Moreover, for $t_i \in \mathbb R^{d_i}$ define 
	\begin{align}\label{eq:w}
		\d w_i := w_{d_i}(t_i) \d t_i = \frac{\d t_i}{c_{d_i}\|t_i\|^{1+d_i}} \quad \text {and } \quad \d w= \prod_{i=1}^{r}\d w_i . 
	\end{align} 
	Finally, for $p \geq 1$, let $\mathcal{P}(\mathbb{R}^{p})$ and $\mathcal{P}_{ac}(\mathbb{R}^{p})$ denote the collection of all probability distributions and Lebesgue absolutely continuous probability distributions on $\mathbb{R}^{p}$, respectively.
	
	\subsection{Higher-Order Distance Covariance and Joint Dependence Measures}\label{subsec:RjdCov}
	
	The celebrated distance covariance $(\dCov)$ of Sz\'ekely and Rizzo \cite{szekely2007measuring} is a powerful measure for independence between two random vectors. It has been widely used for testing of independence \cite{szekely2007measuring,sejdinovic2013equivalence}, feature screening \cite{li2012feature} and general association analysis, including canonical component analysis \cite{zhu2020nonlinear} and independent component analysis \cite{matteson2017independent}. Recently, there has been efforts to generalize the notion of $\dCov$ beyond pairwise independence to joint independence of a collection of $r > 2$ random vectors. Towards this, independently and concurrently, 
	B\"ottcher et al. \cite{bottcher2019distance} and Chakraborty and Zhang \cite{chakraborty2019distance} proposed the following higher-order generalization of $\dCov$. This is referred to as the {\it distance multivariance} (by B\"ottcher et al. \cite{bottcher2019distance}) or the {\it $r$-th order $\dCov$} (by Chakraborty and Zhang \cite{chakraborty2019distance}).

	\begin{definition}[{\cite[Definition 2.1]{bottcher2019distance} and \cite[Definition 1]{chakraborty2019distance}}] 
		The $r$-th order $\dCov$ of $(X_1, X_2, \ldots, X_r)$ is defined as the positive square root of 
		\begin{align}\label{eq:dCov}
			\dCov^{2}(X_1,\ldots,X_r)=\int_{\mathbb{R}^{d_0}}\left|\E\left[\prod_{i=1}^{r}(\phi_{X_i}(t_i)-e^{\i\langle t_i,X_i \rangle}) \right] \right|^{2}\d w.
		\end{align}
		where $\phi_{X_i}(t_i)=\E\left[e^{\i\langle t_i,X_i\rangle}\right]$ for $t_i\in\mathbb{R}^{d_i}$, and $\d w$ is as defined in \eqref{eq:w}. 
	\end{definition}

	Clearly, $\dCov(X_1, X_2, \ldots, X_r) = 0$ whenever the collection $(X_1, X_2, \ldots, X_r)$ are mutually independent. However, the converse is only true for $r=2$, in which case \eqref{eq:dCov} reduces to the classical $\dCov$ between two random vectors $X_1$ and $X_2$ \cite{szekely2007measuring}. In other words, when  $r \geq 3$ the joint independence of $(X_1, X_2, \ldots, X_r)$ is not a necessary
	condition for $\dCov(X_1, X_2, \ldots, X_r)$ to be zero. (For example, $\dCov(X_1, X_2, X_3) = 0$ whenever $X_1 \indep (X_2, X_3)$.) One can conclude $(X_1, X_2, \ldots, X_r)$ are jointly independent when $\dCov(X_1, X_2, \ldots, X_r) = 0$ under the additional assumption that $(X_1, X_2, \ldots, X_r)$ are $(r-1)$-independent (see Definition \ref{defn:r} and \cite[Theorem 3.4]{bottcher2019distance}), but not in general.  To characterize joint independence one has to consider  all possible interactions between the $r$ variables $X_1, X_2, \ldots, X_r$. This leads to the notion of the {\it total distance multivariance} \cite{bottcher2019distance} or the {\it joint distance correlation} ($\JdCov$) \cite{chakraborty2019distance}. Towards this, we need the following notation: For $S \subseteq \{1, 2, \ldots, r\}$ with $|S| \geq 2$ denote by $\bm X_{S} = (X_{i})_{i \in S}$ and
	\begin{align}\label{eq:dCovS}
		\dCov^{2}(\bm X_{S})=\int_{\mathbb{R}^{d_S}}\left|\E\left[\prod_{i \in S} (\phi_{X_i}(t_i)-e^{\i\langle t_i,X_i \rangle}) \right] \right|^{2} \prod_{i \in S} \d w_i , 
	\end{align}
	where $d_S : = \sum_{i \in S} d_i$. In other words, $\dCov^{2}(\bm X_{S})$ is the $|S|$-th order $\dCov$ for the variables in $\bm X_{S}$.

	\begin{definition} [{\cite[Definition 2]{chakraborty2019distance}}]
		Given a vector non-negative weights $\bm C = (C_2, C_3, \ldots, C_r)$, the joint distance covariance $(\JdCov)$ of the random vector $\bm X= (X_1, X_2, \ldots, X_r)$ is defined as: 
		\begin{align}\label{eq:expressionJdCov}
			\JdCov^{2}(\bm X; \bm C) = \sum_{s=2}^r C_s \sum_{\substack{S \subseteq \{1, 2, \ldots, r\} \\ |S| = s }} \dCov^2(\bm X_S) . 
		\end{align}
	\end{definition}

	$\JdCov$ completely characterizes the joint independence of 
	$(X_1, X_2, \ldots, X_r)$, that is, 
	\begin{align}\label{eq:JdCovindependent}
		\JdCov^{2}(\bm X; \bm C) = 0 \text{ if and only if } (X_1, X_2, \ldots, X_r) \text{ are mutually independent } 
	\end{align}
	(see \cite[Proposition 3]{chakraborty2019distance}). Moreover, by choosing all the weights to be equal, that is, $C_s=1$, for $2 \leq s \leq r$, one gets the {\it total distance multivariance} as in \cite[Definition 2.1]{bottcher2019distance}.  Chakraborty and Zhang \cite{chakraborty2019distance} suggests choosing $C_s= c^{r-s}$, for some constant $c>0$, which allows the following simplier expression that does not require evaluating all the $(2^r-r -1)$ dCov terms in \eqref{eq:expressionJdCov}. 
	
	\begin{lemma}[Proposition 4 in \cite{chakraborty2019distance}]\label{lm:U}
		For any $c\geq0$, 
		\begin{align*}
			\JdCov^{2}(\bm X; c) := \JdCov^2(\bm X; (c^{r-s})_{2 \leq s \leq r}) = \E\left[ \prod_{i=1}^{r}(U_i(X_i,X_i')+c)\right] - c^{r} , 
		\end{align*}
		where $(X_1',\ldots,X_r')$ be an independent copy of $\bm X = (X_1, X_2, \ldots, X_r)$ and, for $1\leq i\leq r$, 
		\begin{align*}
			U_i(x,x') & =\E\|x-X_i'\|+\E\|X_i-x'\|-\|x-x'\|-\E\|X_i-X_i'\| \nonumber \\ 
			& = \int_{\mathbb{R}^{d_i}}\bigg\{(\phi_{X_i}(t)-e^{\i\langle t,x\rangle})(\phi_{X_i}(-t)-e^{-\i\langle t,x'\rangle})\bigg\}w_{d_i}(t)\d t .
		\end{align*}
	\end{lemma} 
	
	Depending on whether $c > 1$ or $c < 1$, $\JdCov^{2}(\bm X; c)$ puts more or less weights on the lower-order dependence terms, respectively. For example, if the joint distribution of $\{X_1,\ldots,X_r\}$ is known to be Gaussian,  where mutual independence is equivalent to pairwise independence, larger values of $c$ should be considered. Otherwise, a smaller $c<1$ makes more sense.

	\subsection{Multivariate Ranks Based on Optimal Transport}\label{subsec:multivariate_rank}
	
	To motivate the notion of multivariate ranks based on optimal transport, let us recall some fundamental facts about univariate ranks. Suppose that $X \sim \nu \in \cP_{ac}(\mathbb R)$  is a random variable with distribution function $F$. Here, the distribution function itself serves as the one-dimensional population rank function, which has the property that 
	$F(X) \sim \mathrm{Unif}([0, 1])$, that is, {\it $F$ transports $\mu$ to $\mathrm{Unif}([0, 1])$}, the uniform distribution on $[0, 1]$. In fact, when $\mu$ has finite second moment, it can be shown that $F$ is the almost everywhere unique map that transports $\mu$ to $\mathrm{Unif}([0, 1])$ and minimizes the expected squared-error cost, that is, $$F=\mathrm{arg}\min_{T: T(X) \sim \mathrm{Unif}([0,1])} \E_{X \sim \mu}[(X-T(X))^2],$$ where the minimization is over all functions $T$ that transport the distribution of $X$ to the uniform distribution on $[0, 1]$. This shows that in dimension 1, ranks can be thought of as the univariate analogue of the celebrated {\it Monge transportation problem}  \cite{monge1781memoire,Villani2003}.   	Hallin et al. \cite{delbarrio2019,Hallin2017} used this interpretation of ranks and proposed the following multivariate generalization. 
	
	\begin{definition}[\cite{deb2021multivariate,delbarrio2019,Hallin2017}] \label{defn:multivariate_rank}
		Given a random variable $X \sim \mu$ and a pre-specified reference distribution $\nu$, with $\mu, \nu \in \cP_{ac}(\mathbb R^d)$, the {\it multivariate population rank map} $R_{\mu} : \mathbb R^d \rightarrow \mathbb R^d$ is defined as:
		\begin{align}\label{eq:population_rank}
			R_{\mu} := \mathrm{arg}\min_{T: T(X) \sim \nu} \E_{X \sim \mu}[||X-T(X)||^2], 
		\end{align}
		where the minimum is over all functions $T : \mathbb R^d \rightarrow \mathbb R^d$ that transports the distribution of $X$ to $\mathbb R^d$ and $|| \cdot ||$ denotes the usual Euclidean norm in $\mathbb R^d$.\footnote{Note that, even though the  definition in \eqref{eq:population_rank} is not meaningful when $X \sim \mu$ does not have finite second moment, OT-based multivariate ranks can still be defined using the Brenier-McCann theorem \cite{Mccann1995,brenier1991polar}, which says that there exists an almost everywhere unique map $R_{\mu}: \mathbb R^d \rightarrow \mathbb R^d$, which is the gradient of a convex function and satisfies $R_{\mu, \nu}(X) \sim \nu$. This notion coincides with that in Definition \ref{defn:multivariate_rank}, whenever $\mu$ and $\nu$ have finite second moment.}
	\end{definition} 
	
	To define the empirical analogue of the population multivariate rank we begin with a `natural' discretization $\cH_N^d:=\{ h_1^d,\ldots , h_N^d\}$ of the pre-specified reference distribution $\nu \in \cP_{ac}(\mathbb R^d)$, that is, the empirical measure  $\nu_N:= \frac{1}{N} \sum_{i=1}^N \delta_{h_i^d}$ converges converges weakly to $\nu$. The natural choice for $d=1$ is $\cH_N^1=\{1/N, 2/N, \ldots, 1\}$, whose empirical distribution converges to $\mathrm{Unif}([0, 1])$. For higher dimensions one can choose $\cH_N^d$ as $N$ i.i.d. points from $\nu$ or, more commonly, a deterministic quasi-Monte  Carlo sequence such as the Halton sequence \cite{deb2021multivariate,Hofer2009}. Then, given i.i.d. samples $X_1, X_2, \ldots, X_N$ from a distribution $\mu \in \cP_{ac}(\mathbb R^d)$, the  {\it multivariate empirical rank map} $\hat{R}(X_i)$ is the optimal transport map from the empirical distribution of the data $\hat \mu_N:= \frac{1}{N} \sum_{i=1}^N \delta_{X_i}$ to $\hat \mu_N:= \frac{1}{N} \sum_{i=1}^N \delta_{h_i^d}$. In other words, 
	\begin{align}\label{eq:RN}
		\hat{R}(X_i)=h_{\hat{\sigma}_N(i)}^d,
	\end{align} 
	where 
	\begin{align*} 
		\mathbf{\hat{\sigma}}_N:= \mathrm{arg}\min_{\sigma \in S_N} \sum_{i=1}^N ||  X_i-h_{{\sigma(i)}}^d||^2 , 
	\end{align*} 
	and $S_N$ denotes the set of all $N!$ permutations of $\{1, 2, \ldots, N \}$. This is a {\it minimum bipartite matching problem} (also known as the {\it assignment problem}), which is a fundamental problem in combinatorial optimization that can be solved exactly $O(N^3)$ time \cite{jonker1987shortest} and approximately in near-linear time (see \cite{bipartite} and the references therein).

	A remarkable feature of the optimal transport based multivariate empirical ranks defined in \eqref{eq:RN} is that 	they mimic the distribution-free property of univariate ranks  \cite{deb2021multivariate,delbarrio2019}.

	\begin{proposition}[{\cite[Proposition 1.6.1]{delbarrio2019} and  \cite[Proposition 2.2]{deb2021multivariate}}]\label{ppn:H0distributionfree_OT}
		Suppose that $X_1,\ldots,X_N$ are i.i.d. samples from a distribution $\mu \in \cP_{ac}(\mathbb R^d)$. Then the vector of empirical ranks $$( \hat R(X_1), \hat R(X_2), \ldots, \hat R(X_N))$$ 
		is uniformly distributed over the $N!$ permutations of the fixed grid $\mathcal{H}_N^{d}$. 
	\end{proposition}

	The result above suggests a general strategy for constructing non-parametric distribution-free tests, by replacing the data points with their empirical ranks (appropriately defined depending on the testing problem). This strategy was recently adopted in \cite{deb2021multivariate,shi2020distribution} to construct distribution-free, computationally efficient, and universally consistent multivariate two-sample and  independence tests. For other interesting properties and applications of optimal transport based ranks see  \cite{chernozhukov2017monge,ghosal2022multivariate,hallin2020fully,hallin2020center} and the references therein.

\begin{remark} Common choices of the reference distribution $\nu$ are $\mathrm{Unif}([0, 1]^d)$, the uniform distribution on the $d$-dimensional cube $[0, 1]^d$ \cite{deb2021multivariate}, the spherical uniform distribution $U_d$, which is the product of the uniform distribution on $[0, 1)$ and the uniform distribution on the unit sphere $S^{d-1}$ \cite{Hallin2017,shi2020distribution}, and $N(\bm 0, \bm I_d)$, the $d$-dimensional standard Gaussian \cite{deb2021efficiency}. Hereafter, for concreteness we will choose $\nu = \mathrm{Unif}([0, 1]^d)$ and  $\cH_N^d:=\{ h_1^d,\ldots , h_N^d\}$ will denote the Halton sequence corresponding to  this distribution. However, our results continue to hold for any reference distribution $\nu$ with finite moments, which, in particular, includes the spherical uniform and the standard multivariate Gaussian. 
\end{remark}

\subsection{Rank-Based Joint Dependence Measures}\label{subsec:rank-based jdcov}

We are now ready to define the rank-based counterparts of the $d$-th order $\dCov$ and  $\JdCov$. Towards this, suppose, as before, $(X_1,\ldots,X_r)$ is a $r$-dimensional random vector, where each $X_i$ is a random variable taking in $\mathbb{R}^{d_i}$ with distribution $\mu_i$, for $1 \leq i \leq r$ and $d_0 = \sum_{i=1}^r d_i$. Throughout we will assume $\mu_i \in\mathcal{P}_{ac}(\mathbb{R}^{d_i})$ and denote the population rank map of $\mu_i$ by $R_{\mu_i}$ (recall Definition \ref{defn:multivariate_rank}), for $1 \leq i \leq r$.

\begin{definition}\label{def:RdCov}
	The $r$-th order rank $\dCov$ $(\RdCov)$ of $(X_1, X_2, \ldots, X_r)$ is defined as the positive square root of 
	\begin{align}\label{eq:RdCov}
		\RdCov^{2}(X_1,\ldots,X_r)=\int_{\mathbb{R}^{d_0}} \left|\E\left[\prod_{i=1}^{r} \left\{ \E\left(e^{\i\langle t_i,R_{\mu_i}(X_i)\rangle}\right) - e^{\i\langle t_i, R_{\mu_i}(X_i)\rangle}\right\} \right] \right |\d w
	\end{align}
	where $t_i\in\mathbb{R}^{d_i}$, for $1 \leq i \leq r$, and $\d w$ is as defined in \eqref{eq:w}. 
\end{definition} 

Note that the $r$-th order $\RdCov$ is obtained from the $r$-th order $\dCov$ by replacing $(X_1, X_2, \ldots, X_r)$ with their population rank maps $(R_{\mu_1}(X_1), R_{\mu_2}(X_2), \ldots, R_{\mu_r}(X_r) )$. For $r=2$, \eqref{def:RdCov} reduces to the rank distance covariance of $(X_1, X_2)$ defined in 
\cite{deb2021multivariate,shi2020distribution}, which, like the classical $\dCov$, is zero if and only if  $X_1 \indep X_2$. For the $r \geq 3$, as in $\JdCov$, we need to consider all the interactions between the $r$ variables $(X_1, X_2, \ldots, X_r)$ to characterize the joint indepedence. This leads to our the notion of {\it rank joint distance covariance} ($\RJdCov$):

\begin{definition} Given a vector non-negative weights $\bm C = (C_2, C_2, \ldots, C_r)$, the rank joint distance covariance $(\RJdCov)$ of the random vector $\bm X= (X_1, X_2, \ldots, X_r)$ is defined as: 
	\begin{align}\label{eq:RJdCov}
		\RJdCov^{2}(\bm X; \bm C) = \sum_{s=2}^r C_s \sum_{\substack{S \subseteq \{1, 2, \ldots, r\} \\ |S| = s }} \RdCov^2(\bm X_S) , 
	\end{align} 
	where 
	\begin{align}\label{eq:RdCovdS}
	\RdCov^{2}(\bm X_S)=\int_{\mathbb{R}^{d_S}} \left|\E\left[\prod_{ i \in S } \left\{ \E\left(e^{\i\langle t_i,R_{\mu_i}(X_i)\rangle}\right) - e^{\i\langle t_i, R_{\mu_i}(X_i)\rangle}\right\} \right] \right | \prod_{i \in S} \d w_i . 
	\end{align}
	Moreover, for any $c\geq0$, 
	\begin{align}\label{eq:RJDCovS}
		\RJdCov^{2}(\bm X; c) := \RJdCov^2(\bm X; (c^{r-s})_{2 \leq s \leq r}) = \sum_{s=2}^r c^{r-s} \sum_{\substack{S \subseteq \{1, 2, \ldots, r\} \\ |S| = s }} \RdCov^2(\bm X_S) . 
	\end{align} 
	$($Note that $ \RdCov^{2}(X_1,\ldots,X_r) = \RJdCov^{2}(\bm X; 0)$.$)$
\end{definition}

As in Lemma \ref{lm:U}, $\RJdCov^{2}(\bm X; c)$ has the following compact representation: 
\begin{align}\label{eq:RJdCovW}
	\RJdCov^{2}(\bm X; c) := \E\left[ \prod_{i=1}^{r}(W_i(R_{\mu_i}(X_i), R_{\mu_i}(X_i') )+c)\right] - c^{r} , 
\end{align}
where $(X_1',\ldots,X_r')$ be an independent copy of $\bm X = (X_1, X_2, \ldots, X_r)$ and, for $1\leq i\leq r$, 
\begin{align}\label{eq:W}
	W_i(x,x') & =\E\|x-R_{\mu_i}(X_i')\|+\E\|R_{\mu_i}(X_i)-x'\|-\|x-x'\| - \E\|R_{\mu_i}(X_i)-R_{\mu_i}(X_i')\| . 
\end{align}
Note that $R_{\mu_1}(X_1), R_{\mu_2}(X_2), \ldots, R_{\mu_r}(X_r)$ are i.i.d. $\mathrm{Unif}([0, 1]^d)$ when the variables $X_1, X_2, \ldots, X_r$ are mutually independent (by the definition of the optimal transport maps). Therefore, the distribution of $\RJdCov^{2}(\bm X; \bm C)$ does not depend on the marginal distributions of $X_1, X_2, \ldots, X_r$ under mutual independence. Moreover,  $\RJdCov^{2}(\bm X; \bm C)$ characterizes the joint independence of $X_1, X_2, \ldots, X_r$. 

\begin{proposition}\label{ppn:RJdCovindependence} 
Let $\RdCov$ and $\RJdCov$ be as defined in \eqref{eq:RdCovdS} and \eqref{eq:RJdCov}, respectively. Then the following hold: 
\begin{itemize} 
\item[$(1)$] For any $S \subseteq \{1, 2, \ldots, r\}$ with $|S| \geq 2$, $\RdCov(\bm X_S) = 0$ if the variables $(X_i)_{i \in S}$ are independent. Moreover, the converse holds whenever the variables $(X_i)_{i \in S}$ are $|S| -1$ independent. 

\item[$(2)$] For any vector positive weights $\bm C$, $\RJdCov^{2}(\bm X; \bm C) = 0$ if and only if  $(X_1, X_2, \ldots, X_r)$ are mutually independent. 
\end{itemize} 
\end{proposition} 

\begin{proof} Note that $\{R_{\mu_i}(X_i)\}_{i \in S}$ are independent when $(X_i)_{i \in S}$ are mutually independent, which implies $\RdCov^2(\bm X_S) = 0$, by definition. For the converse, note that $\RJdCov^{2}(\bm X_S) = 0$ and $\{R_{\mu_i}(X_i)\}_{i \in S}$ are $|S|-1$ independent implies $\{R_{\mu_i}(X_i)\}_{i \in S}$ are mutually independent by \cite[Theorem 3.4]{bottcher2019distance}. Then, since $\mu_i \in \cP_{ac}(\mathbb R^{d_i})$, by  McCann's theorem \cite{Mccann1995} (see also \cite[Proposition 2.1]{deb2021multivariate}) there exists a measurable function $Q_{\mu_i}$ such that $Q_{\mu_i}(R_{\mu_i}(X_i))=X_i$ almost everywhere $\mu_i$, for $1 \leq i \leq r$. This implies, $X_1, X_2, \ldots, X_r$ are mutually independent. This completes the proof of (1).

For (2) note that $\RJdCov^{2}(\bm X; \bm C) = 0$ if and only if  $\RdCov^{2}(\bm X_S) = 0$ for all $S \subseteq \{1, 2, \ldots, r\}$ with $|S| \geq 2$. The result then follows from (1). 
\end{proof} 

\subsection{Consistent Estimation of $\RJdCov$}\label{subsec:rank-based jdcov est}

In this section we discuss how $\RdCov$ and $\RJdCov$ can be consistently estimated given $n$ samples i.i.d. $\{\bm X_a\}_{a=1}^{n}$, with $\bm X_a=(X_{1}^{(a)},\ldots,X_{r}^{(a)})$, from a distribution $\mu \in \cP_{ac}(\mathbb R^{d_0})$.  The natural plug-in estimator of $\RdCov^2(\bm X_S)$, for $S \subseteq \{1, 2, \ldots, r\}$ with $|S| \geq 2$, is: 
\begin{align}\label{eq:RdCovSn}
	\RdCov^{2}_{n}(\bm X_S):=\int_{\mathbb{R}^{d_S}}\left|\frac{1}{n}\sum_{b=1}^{n}\left\{\prod_{i \in S} \left (\frac{1}{n}\sum_{a=1}^{n}e^{\i\langle t_i,\R_i(X_{i}^{(a)})\rangle}-e^{\i\langle t_i,\R_i(X_{i}^{(b)})\rangle} \right) \right\} \right|^{2} \prod_{i \in S} \d w_i, 
\end{align} 
where $\R_{i}$ is the empirical rank map for the $i$-th marginal distribution, that is, the optimal transport map from $\frac{1}{n}\sum_{a=1}^n \delta_{X_i^{(a)}}$ to $\frac{1}{n}\sum_{a=1}^n\delta_{h_a^{d_i}}$, with $\cH_n^{d_i} = \{ h_1^{d_i},  h_2^{d_i}, \ldots,  h_n^{d_i}\} $ the Halton sequence in $[0, 1]^{d_i}$, for $1 \leq i \leq r$. Using the representation in \eqref{eq:RJdCovW} and \eqref{eq:W} the estimate in \eqref{eq:RdCovSn} can be written as: 
\begin{align}\label{eq:RdCovWXS}
	\RdCov^2_n(\bm X_S) &=\frac{1}{n^2}\sum_{1 \leq a, b \leq n} \prod_{i \in S} \hat{\cE}_i(a, b) . 
\end{align} 
where 
\begin{align}\label{eq:estimateW}
	\hat{\cE}_i(a, b) & := \frac{1}{n}\sum_{v=1}^n\|\hat{R}_{i}(X_{i}^{(a)})-\hat{R}_{i}(X_{i}^{(v)})\|+\frac{1}{n}\sum_{u=1}^n\|\hat{R}_{i}(X_{i}^{(u)})-\hat{R}_{i}(X_{i}^{(b)})\| \nonumber \\ 
	& \hspace{1.25in} -\|\hat{R}_{i}(X_{i}^{(a)})-\hat{R}_{i}(X_{i}^{(b)})\| -\frac{1}{n^2}\sum_{1 \leq u,v \leq n} \|\hat{R}_{i}(X_{i}^{(u)})-\hat{R}_{i}(X_{i}^{(v)})\| , 
\end{align} 
for $1 \leq a, b \leq n$. Consequently, the natural plug-in estimate of \eqref{eq:RJdCov} is 
\begin{align}\label{eq:RJdCovn}
	\RJdCov^{2}_n(\bm X; \bm C) = \sum_{s=2}^r C_s \sum_{\substack{S \subseteq \{1, 2, \ldots, r\} \\ |S| = s }} \RdCov^2_n(\bm X_S) , 
\end{align} 
and that of \eqref{eq:RJDCovS} (using the representation in \eqref{eq:RJdCovW} is 
\begin{align*}
	\RJdCov^{2}_n(\bm X; c)  
	& :=  \frac{1}{n^2}\sum_{1 \leq a, b \leq n} \prod_{i = 1}^r \left( \hat{\cE}_i(a, b)  + c \right) - c^r . 
\end{align*}

	\begin{remark}\label{remark:functionS} 
		For $S \subseteq \{1,2, \ldots, r\}$ with $|S| \geq 2$, consider the function $\theta_S: (\mathbb R^{d_S})^n \rightarrow \mathbb R_{\geq 0}$
		\begin{align}\label{eq:thetaS}
			\theta_{S}((x_i^{(1)})_{i \in S}, \ldots, (x_i^{(n)})_{i \in S}):=\int_{\mathbb{R}^{d_S}}\left|\frac{1}{n}\sum_{b=1}^{n}\left\{\prod_{i \in S} \left (\frac{1}{n}\sum_{a=1}^{n}e^{\i\langle t_i, x_i^{(a)})\rangle}-e^{\i\langle t_i, x_i^{(b)}\rangle} \right) \right\} \right|^{2} \prod_{i \in S} \d w_i, 
		\end{align} 
		where $x_i^{(a)} \in \mathbb R^{d_i}$, for $1 \leq i \leq r$ and, hence, $(x_i^{(a)})_{i \in S} \in \mathbb R^{d_S}$, for $1 \leq a \leq n$. Evaluating $\theta_{S}$ at the data points gives us the natural plug-in estimate of $\dCov(\bm X_S)$ (recall \eqref{eq:dCovS}) which is denoted by: 
		$$\dCov_n(\bm X_S) :=  \theta_{S}(\bm X_S^{(1)}, \ldots, \bm X_S^{(n)}) , $$
		where $\bm X_S^{(a)} := (X_i^{(a)})_{i \in S} \in \mathbb R^{d_S}$, for $1 \leq a \leq n$. On the other hand, evaluating $\theta_{S}$ on the the rank transformed data gives us the estimate of $\RdCov$ in \eqref{eq:RdCovSn}. Specifically, recalling \eqref{eq:RdCovSn} note that 
		\begin{align}\label{eq:RdCovHalton}
			\RdCov_n(\bm X_S)=  \theta_{S}(\hat {\bm R}_S^{(1)}, \ldots, \hat {\bm R}_S^{(n)} ), 
		\end{align}
		where $\hat {\bm R}_S^{(a)} := (\R_i(X_i^{(a)}))_{i \in S} \in \mathbb R^{d_S}$, for $1 \leq a \leq n$. The function $\bm \theta_S$ will play an important role in the calibration of the independence tests on $\RJdCov_n$ discussed in Section \ref{sec:independencetesting}. 
	\end{remark}

	To establish the consistency of $\RJdCov_n$ we will assume the following:  
	
	\begin{assumption}\label{assumption:U}
		For every $1 \leq i \leq p$, the empirical distribution of $\cH_n^{d_i} = \{ h_1^{d_i},  h_2^{d_i}, \ldots,  h_n^{d_i}\} $ converges weakly to $\mathrm{Unif}([0, 1]^{d_i})$. 
	\end{assumption}
	
	As mentioned before, Assumption \ref{assumption:U} holds whenever 
	$\cH_n^{d_i}$ is the Halton sequence in $[0, 1]^{d_i}$ (which is our default choice), for a random sample $n$ i.i.d. $\mathrm{Unif}([0, 1]^{d_i})$ points, or other quasi-Monte Carlo sequences as well (see \cite[Appendix D]{deb2021multivariate} for a discussion on the various of choices $\cH_n^{d_i}$). Under this assumption we have the following  result: 
	
	\begin{theorem}\label{thm:consistency} Suppose Assumption \ref{assumption:U} holds. Then for any $S \subseteq \{1, 2, \ldots, d\}$ with $|S| \geq 2$, 
		$$\RdCov^{2}_n(\bm X_S)\stackrel{a.s.}{\rightarrow}\RdCov^{2}(\bm X_S).$$
		Consequently, $\RJdCov^{2}_n(\bm X; \bm C)\stackrel{a.s.}{\rightarrow}\RJdCov^{2}(\bm X; \bm C)$. 
	\end{theorem}
	
	The proof of Theorem \ref{thm:consistency} is given in Appendix \ref{sec:consistencypf}. One of the  highlights of the above result is that it holds without any moment assumptions on the data generating distribution. This is in contrast to consistency results for  $\JdCov$ which require the distribution to satisfy certain moment conditions. Consequently, the rank-based joint dependence measures are more robust to heavy-tail distributions, as will been seen from the simulations in Section \ref{sec:sim}.

	\section{Distribution-Free Joint Independence Testing}
	\label{sec:independencetesting}
	
	In section we discuss new distribution-free tests for the multivariate joint independence testing. Given i.i.d.samples $\{\bm X_a\}_{a=1}^{n}$, with $\bm X_a=(X_{1}^{(a)},\ldots,X_{r}^{(a)})$, from a distribution $\mu \in \cP_{ac}(\mathbb R^{d_0})$ with marginal distributions $\mu_1, \mu_2, \ldots, \mu_r$ such that $\mu_i \in \cP_{ac}(\mathbb R^{d_i})$, for $1 \leq i \leq r$, consider the joint independence testing problem:  
	\begin{align}\label{eq:H0H1}
		H_0:\mu=\mu_1\otimes\mu_2\otimes\cdots\otimes\mu_r \quad \text{ versus } \quad H_1:\mu \neq \mu_1\otimes\mu_2\otimes\cdots\otimes\mu_r.
	\end{align}

	Our test for \eqref{eq:H0H1} will be based on rejecting $H_0$ for `large' values of $\RJdCov^{2}(\bm X; \bm C)$, for a given choice non-negative weights $\bm C = (C_2, C_3, \ldots, C_r)$. Note that the vector of ranks $\{\R_i(X_i^{(a)})\}_{1 \leq a \leq n}$ are distributed uniformly over the $n!$ permutations of $\cH_n^{d_i}$, for $1 \leq i \leq r$ (by Proposition \ref{ppn:H0distributionfree_OT}) and under $H_0$ independent over $1 \leq i \leq r$. Hence, the distribution of $\RJdCov^{2}(\bm X; \bm C)$ under $H_0$ does not depend on the marginal distributions $\mu_1, \mu_2, \ldots,\mu_r$.

	\begin{proposition}[Distribution-freeness under $H_0$]\label{ppn:H0distributionfree}
		Under $H_0$, the distribution of $\RJdCov^{2}_n(\bm X; \bm C)$ is universal, that is, it is free of $\mu_1,\ldots,\mu_r$. 
	\end{proposition}

	The above result automatically gives a finite sample distribution-free independence test that uniformly controls the Type-I error for \eqref{eq:H0H1}. To this end, fix a level $\alpha\in (0,1)$ and let $c_{\alpha, n}$ denote the upper $\alpha$ quantile of the universal distribution in Proposition \ref{ppn:H0distributionfree}. Consider the test function:
	\begin{equation}\label{eq:testrank}
		\phi_n(\bm C):=\bm 1\left\{\RJdCov^{2}_n(\bm X; \bm C) \geq c_{\alpha, n}\right\}.
	\end{equation}
	This test is exactly distribution-free for all $n\geq 1$ and uniformly of level $\alpha$ under $H_0$,\footnote{Strictly speaking, to guarantee exact level $\alpha$, we have to randomize $\phi_n(\bm C)$, as the exact distribution of $\RJdCov^{2}_n(\bm X; \bm C)$ is discrete. However, this makes no practical difference unless $n$ is very small.} that is, 
	\begin{equation}\label{eq:uniflevelt}
		\sup_{\mu=\mu_1\otimes\mu_2\otimes\cdots\otimes\mu_r} \E_\mu\left[\phi_n(\bm C) \right]=\alpha.
	\end{equation}
	Moreover, since $\RJdCov^{2}_n(\bm X; \bm C)$ is a consistent estimate of $\RJdCov^{2}(\bm X; \bm C)$, which characterizes joint independence, the test $\phi_n(\bm C)$ is consistent against all fixed alternatives. This is summarized in the following proposition (see Appendix \ref{sec:consistencypf} for the proof). 
	
	\begin{proposition}[Universal consistency]\label{ppn:phiconsistency}
		Suppose Assumption \ref{assumption:U} holds. Then the test $\phi_n(\bm C)$ is  universally consistent, that is, for any $\mu \ne \mu_1\otimes\mu_2\otimes\cdots\otimes\mu_r$, 
		\begin{equation}\label{eq:power}
			\lim_{n \rightarrow \infty} \E_{\mu}\left[\phi_n(\bm C) \right] = 1 . 
		\end{equation} 
	\end{proposition}

	\begin{remark}\label{remark:RdCov} 
		Note that, while $\RJdCov_n(\bm X; \bm C)$ provides a distribution-free, universally consistent test for joint independence, the individual $\RdCov(\bm X_S)$ can be used for testing the joint independence of the subset of variables $\bm X_{S}$ given that they are $|S|-1$ independent. This is because $\RdCov_n(\bm X_S) = 0$ if and if only if the variables $\bm X_{S}$ are mutually independent, provided they are $|S|-1$ independent. 
		For example, if we are interested in testing pairwise independence: 
		$$H_0: X_1, X_2, \ldots, X_r \text{ are pairwise independent } \quad \text{ versus } \quad H_1: \text{ not } H_0,$$
		then the test which rejects $H_0$ for `large' values of 
		$$\sum_{1 \leq i < j \leq r} \RdCov^2_n(X_i, X_j),$$
		will be distribution-free and consistent. This is, in fact, the rank-analogue of the test based on pairwise distance covariances studied in \cite{yao2018testing}. Similarly, if we know that 3 variables $(X_1, X_2, X_3)$ are pairwise independent, then $\RdCov^2_n(X_1,X_2,X_3)$ can be used to obtain a distribution-free, consistent test for the mutual independence of $(X_1, X_2, X_3)$. Another attractive property of $\RdCov$ is that the collection $\RdCov_n(\bm X_S)$ is asymptotically independent over $S \subseteq \{ 1, 2, \ldots, d \}$ under the null hypothesis of joint independence (see Theorem \ref{thm:H0}), which can be leveraged to learn the higher-order dependencies among the variables $X_1, X_2, \ldots, X_r$. 
	\end{remark}
	
	\begin{remark} As mentioned in the Introduction, another related measure for joint independence based on the second-order $\dCov$ was proposed by \citet{matteson2017independent}, and its rank version was discussed in \cite{deb2021multivariate}. Although these measures characteristic joint independence, unlike $\JdCov/\RJdCov$ it is unable to capture the dependence structure among the variables. In Section \ref{sec:sim} we compare our method with test proposed in \cite{matteson2017independent} in simulations. 
		\end{remark}

	\subsection{Asymptotic Null Distribution}
	\label{sec:theory}

	In this section we will derive the asymptotic null distribution of the collection of $2^{r} - r -1$ random variables: 
	\begin{align}\label{eq:RS}
		\{ \RdCov_n(\bm X_S): S \in \cT \} , 
	\end{align} 
	where $\cT:=\{S \subseteq \{ 1, 2, \ldots, r \} \text{ such that } |S| \geq  2\}$, 
	and consequently that of $\RJdCov_n(\bm X; \bm C)$. To describe the limit we need the following definition: 
	
	\begin{definition}\label{defn:ZSprocess}
		Denote by $\{Z_S\}_{S \in \cT}$ a collection of mutually independent complex-valued Gaussian processes such that, for each $S \in \cT$,
		\begin{align}\label{eq:ZS}
			\{Z_S(\bm t): \bm t \in \mathbb R^{d_S} \}
		\end{align}
		a complex-valued Gaussian process indexed by $\mathbb R^{d_S}$ (recall $d_S = \sum_{i \in S} d_i$) with zero mean and covariance function: 
		\begin{align}\label{eq:CS}
			C_S(\bm t, \bm t'): = \Cov[Z_S(\bm t), Z_S(\bm t')] := \prod_{i \in S} \left( \E\left[e^{\i\langle t_{i}-t_{i}',U_{i}\rangle}\right]-\E\left[e^{\i\langle t_{i},U_{i}\rangle} \right] \E\left[e^{\i\langle -t_{i}',U_{i}\rangle} \right] \right) , 
		\end{align}
		where 
		\begin{itemize} 
			
			\item $\bm t= (t_i)_{i \in S} \in \mathbb R^{d_S}$ and $\bm t'= (t_i')_{i \in S} \in \mathbb R^{d_S}$, with $t_i, t_i' \in \mathbb R^{d_i}$, for $i \in S$, and 
			
			\item $U_1, U_2, \ldots, U_r$ are independent with $U_i \sim \mathrm{Unif}([0, 1]^{d_i})$, for $1 \leq i \leq r$. 
			
		\end{itemize} 
	\end{definition}
	
	\begin{remark}
		The covariance function \eqref{eq:CS} can be expressed in closed form  using the fact, 
		$$\E\left[e^{\i\langle t_{i},U_{i}\rangle} \right]  = \prod_{j = 1}^{d_i} \frac{e^{\i t_{ij}} - 1}{\i t_{ij}},$$ 
		for  $U_i \sim \mathrm{Unif}([0, 1]^{d_i})$ and $t_i = (t_{ij})_{1 \leq j \leq d_i} \in \mathbb R^{d_i}$. Hence,  
		$$ C_S(\bm t, \bm t') = \prod_{i \in S}  \left( \prod_{j = 1}^{d_i} \frac{e^{\i ( t_{ij} - t_{ij}' ) } - 1}{\i (t_{ij} - t_{ij}') }   - \prod_{j = 1}^{d_i} \frac{e^{\i t_{ij}} - 1}{ t_{ij}}  \prod_{j = 1}^{d_i} \frac{e^{-\i t_{ij}'} - 1}{t_{ij}'} \right),$$
		where $t_i = (t_{ij})_{1 \leq j \leq d_i} \in \mathbb R^{d_i}$ and $t_i' = (t_{ij}')_{1 \leq j \leq d_i} \in \mathbb R^{d_i}$.
	\end{remark}

	With the above definition we can now describe the asymptotic distribution of the collection \eqref{eq:RS} under the null hypothesis of joint independence. 

	\begin{theorem}\label{thm:H0}
		Suppose Assumption \ref{assumption:U} holds. Then under $H_0$ as in \eqref{eq:H0H1}, as $n \rightarrow \infty$, 
		\begin{align} \label{eq:H0RdCovS}
			\{ n \RdCov_n(\bm X_S): S \in \cT \} \dto \left\{ \int_{\mathbb R^{d_S}} |Z_S(\bm t)|^2 \prod_{i\in S} \d w_i \right\}_{S \in \cT} , 
		\end{align} 
		for $Z_S(\bm t)$ as defined in \eqref{eq:ZS}. Consequently,  for any vector of non-negative weights $\bm C= (C_1, C_2, \ldots, C_r)$, under $H_0$ as in \eqref{eq:H0H1}, 
		\begin{align}\label{eq:H0RJdCov}
			n\RJdCov^{2}_n(\bm X; \bm C) \dto \sum_{s=2}^r C_s \sum_{\substack{S \subseteq \{1, 2, \ldots, r\} \\ |S| = s }} \int_{\mathbb R^{d_S}} |Z_S(\bm t)|^2 \prod_{i\in S} \d w_i . 
		\end{align} 
	\end{theorem}

	The proof of Theorem \ref{thm:H0} is given in Appendix \ref{sec:H0pf}. One of the challenges in dealing with $\RdCov_n$ and $\RJdCov_n$ is that they are functions of dependent multivariate ranks. 
	To solve the problem, we develop a version of the classical Hoeffiding's combinatorial central limit theorem for $r$-tensors indexed by multiple independent permutations (see Theorem \ref{thm:clt} in Appendix \ref{sec:clt}), a result which might be more broadly useful in the analysis of other nonparametric tests. Our proof of the combinatorial CLT uses Stein's method based on exchangeable pairs, combined with techniques from empirical process theory gives us the result in Theorem \ref{thm:H0}.

	\begin{remark}(Equivalent representation of the limiting distribution) By considering the process  $Z_S(\bm t)$ as a random element in $L^2(\mathbb R^{d_S})$ and an eigen-expansion of the integral operator on associated with the covariance kernel \eqref{eq:CS} the limit in \eqref{eq:H0RdCovS} can be expressed as: 
		$$\int_{\mathbb R^{d_S}} |Z_S(\bm t)|^2 \prod_{i\in S} \d w_i \stackrel{D} = \sum_{j=1}^\infty \lambda_j Z_j^2, $$
		where $\{\lambda_j\}_{j \geq 1}$ are positive constants and $Z_1, Z_2, \ldots, $ are i.i.d. $N(0, 1)$ (see \cite[Chapter 1, Section 2]{kuo1975gaussian}). Hence, by the independence of the processes $\{Z_S\}_{S \in \cT}$, the limit in \eqref{eq:H0RJdCov} can be represented as an infinite sum weighted chi-squares: 
		$$n\RJdCov^{2}_n(\bm X; \bm C) \dto \sum_{s=2}^r C_s \sum_{\substack{S \subseteq \{1, 2, \ldots, r\} \\ |S| = s }} \int_{\mathbb R^{d_S}} |Z_S(\bm t)|^2 \prod_{i\in S} \d w_i \stackrel{D}{=} \sum_{j=1}^\infty \lambda_j' Z_j^2, $$
		for a sequence of positive constants $\{\lambda_j'\}_{j \geq 1}$ and $Z_1, Z_2, \ldots, $ are i.i.d. $N(0, 1)$. 
	\end{remark}

	\begin{remark} Another interesting consequence of Theorem \ref{thm:H0} is that collection $\{ \RdCov_n(\bm X_S): S \in \cT \}$ is asymptotically independent. This can be attributed to the representation of $\RdCov_n(\bm X_S)$ as a product of zero mean (under $H_0$) random variables (recall \eqref{eq:RdCovSn}). This property is shared by the $\JdCov$ (equivalently, the distance multivariance) which has a similar product structure (see \cite[Theorem 4.10]{bottcher2019distance}). 
	\end{remark}

	\subsection{Approximating the Cut-off and Finite Sample Properties} 
	\label{sec:finitesample}
	
	Note that the limit in \eqref{eq:H0RdCovS}, as expected, does not depend on the distribution of the data or the specific choices of $\cH_n^{d_1}, \cH_n^{d_2}, \ldots, \cH_n^{d_r}$. This is a manifestation of the distribution-free property of the multivariate ranks. Hence, if $c_\alpha$ denotes the $(1-\alpha)$-th quantile of the distribution in the RHS of \eqref{eq:H0RJdCov}, the test which rejects $H_0$ when 
	\begin{align}\label{eq:testasymptotic}
		\phi^{\mathrm{asymp}}(\bm C) = \bm 1\{ n\RJdCov^{2}_n(\bm X; \bm C) > c_\alpha\} , 
	\end{align}
	will have asymptotic level $\alpha$  and universally consistent for the hypothesis \eqref{eq:H0H1}. Although this test is asymptotically valid, there is, in general, no tractable form of the quantile $c_\alpha$ for the distribution \eqref{eq:H0RJdCov}. Nevertheless, using the distribution-free property of the multivariate ranks we can still approximate the quantiles of \eqref{eq:H0RdCovS} (and consequently \eqref{eq:H0RJdCov}) in a data-agnostic manner. This is because $\RdCov_n(\bm X_S)$ is a function of the multivariate ranks (recall \eqref{eq:RdCovSn} and \eqref{eq:RdCovHalton}), and the ranks are independent and the uniformly distributed over the $n!$ permutations of $\cH_n^{d_i}$, for $1 \leq i \leq r$ (by  Proposition \ref{ppn:H0distributionfree_OT}). Therefore, we can compute an approximate quantiles of $\RdCov_n(\bm X; \bm C)$ as follows:  For $1 \leq b \leq B$ repeat the following two steps: 
	
	\begin{itemize}
		
		\item[$(1)$] Generate $r$ i.i.d. uniform random permutations $\bm \pi:= \{ \pi_1, \pi_2, \ldots, \pi_r \}$ from $S_n$. 
		
		\item[$(2)$] Compute the value of the statistic $\RdCov_n(\bm X; \bm C)$ on the permuted Halton sequences 
		$\pi_1\cH_{n}^{d_1}, \pi_2 \cH_{n}^{d_2}, \ldots, \pi_r \cH_{n}^{d_r}$, where 
		$$\pi_i \cH_{n}^{d_i} := \{h_{\pi_i(1)}^{d_i}, h_{\pi_i(2)}^{d_i}, \ldots, h_{\pi_i(n)}^{d_i}\},$$
		for $1 \leq i \leq r$. More precisely, we evaluate 
		\begin{align}\label{eq:pi}
			\RdCov_n^{(b)}(\bm X; \bm C)= \sum_{s=2}^r C_s \sum_{\substack{S \subseteq \{1, 2, \ldots, r\} \\ |S| = s }}  \theta_{S}(\bm h_{S, \bm \pi }^{(1)}, \ldots, \bm h_{S, \bm \pi}^{(n)} ), 
		\end{align}
		where $\theta_S$ is the function in \eqref{eq:thetaS} and $\bm h_{S, \bm \pi}^{(a)} := (h_{\pi_i(a)}^{d_i})_{i \in S} \in \mathbb R^{d_S}$, for $1 \leq a \leq n$. 
		
	\end{itemize} 
	Then the permutation $(1-\alpha)$-th quantile of $\RdCov_n(\bm X; \bm C)$ is obtained as: 
	\begin{align}\label{eq:cn}
	c_{n, M, \alpha} : = \min \left\{\RdCov_n^{(s)}(\bm X; \bm C): \frac{1}{B} \sum_{b=1}^B \bm 1 \left\{\RdCov_n^{(b)}(\bm X; \bm C) > \RdCov_n^{(s)}(\bm X; \bm C) \right \} \leq \alpha \right \} .  
	\end{align} 
	To compute the permutation $p$-value, let $R$ be the rank of $\RdCov_n(\bm X; \bm C)$ (the original value computed from the data) in the sequence 
	\begin{align*}
		(\RdCov_n^{(1)}(\bm X; \bm C), \RdCov_n^{(2)}(\bm X; \bm C), \ldots, \RdCov_n^{(B)}(\bm X; \bm C)) ,
	\end{align*} 
	by breaking ties at random and where $R=1$ denotes the rank of the largest element. Then the permutation $p$-value as $p_B:=R/(B+1)$.  
	Note that if $H_0$ is true, $\P(p_B\leq\alpha)=\lfloor \alpha(B+1)\rfloor/(B+1) \leq \alpha$. In other words, 
	\begin{align}\label{eq:phitest}
		\tilde \phi_n(\bm C) : = \bm 1 \{ p_B \leq \alpha \}. 
	\end{align}
	is a finite sample level $\alpha$ test.

	\begin{remark}
		Note that the resampling method described above is different from the (data-dependent) bootstrap/permutation method used to calibrate $\JdCov(\bm X, \bm C)$, or other independence tests which are not distribution-free), where the asymptotic distribution depends on the (unknown) marginal distributions of $X_1, X_2, \ldots, X_r$ (see \cite[Proposition 9]{chakraborty2019distance}). On the other hand, our method can be used to compute the cut-offs by only permuting the Halton sequences, without any knowledge of the data, 
		and hence is much more efficient when one has to do multiple tests.  
	\end{remark}

	Although \eqref{eq:phitest} produces a test which has level $\alpha$ in finite samples, its power properties with a finite number of permutations is  a-priori unclear.  
	This is due to fact that for any collection of $r$ random permutations $\bm \pi= \{\pi_1, \pi_2, \ldots, \pi_r\}$, the RHS of \eqref{eq:pi} tends to zero with high probability (see Proposition \ref{ppn:consistencypermutation} in Appendix \ref{sec:consistencypermutationpf}). This combined with Theorem \ref{thm:consistency} leads to  the following result:
	
	\begin{theorem}[Consistency with finite number of resamples]\label{thm:consistency_permutation}
		Suppose Assumption \ref{assumption:U} holds. Fix $\alpha\in(0,1)$ and suppose $B\geq1/\alpha-1$. Then, for any fixed 
		$\mu \ne \mu_1\otimes\mu_2\otimes\cdots\otimes\mu_r$, 
		$$\lim_{n\rightarrow\infty}\E_{\mu}[\tilde \phi(\bm C)]=1 . $$
	\end{theorem} 
	
	\begin{proof}
		Define $T:=\RdCov^{2}(\bm X, \bm C)$, $T_n:=\RdCov_n(\bm X; \bm C)$, and 
		$T_n^{(b)}:=\RdCov_n^{(b)}(\bm X; \bm C)$, for $1 \leq b \leq B$ (as defined in \eqref{eq:pi}). Then, since $B\geq1/\alpha-1$, 
		\begin{align*}
			\E_\mu[\tilde \phi_n(\bm C)] & = \P(p_B \leq \alpha)\\
			&
			\geq\P\left(T_n > T/2 \text{ and } T_n^{(b)} < T/2\text{ for } 1 \leq b \leq B \right)\\
			&
			\geq 1-\left\{ \P(T_n <T/2)+B \cdot \P(T_n^{(1)}>T/2) \right\} \tag*{(by the union bound)} \\
			&
			\rightarrow 1 , 
		\end{align*}
		where the last convergence is due to the consistency result in Theorem \ref{thm:consistency} and Proposition \ref{ppn:consistencypermutation} in Appendix \ref{sec:consistencypermutationpf}.
	\end{proof}

	\section{Local Power Analysis}\label{sec:local_power}

	In this section, we derive the asymptotic local power of the statistic 
	$\RJdCov^2_n(\bm X; \bm C)$ under the contiguous alternatives. In the context of testing pairwise independence two types of local alternatives have been considered in the literature: (1) mixture alternatives \cite{deb2021efficiency,shi2020consistent}  and (2) Konijn alternatives which was originally proposed in \cite{konijn1956power} and has been later investigated in \cite{gieser1997nonparametric} and used more recently in \cite{deb2021efficiency,shi2020consistent}. To the best of our knowledge, local power analysis has not been carried out for joint independence testing of more than 2 variables. The first step towards this is to define contiguous versions of the mixture and Konijn alternatives for multiple random vectors. For this suppose the joint distribution $\mu$ has density $f$ with respect to the Lebesgue measure in $\mathbb R^{d_0}$ and the marginal distributions $\mu_1, \mu_2, \ldots, \mu_r$ have densities $f_{X_1}, f_{X_2}, \ldots, f_{X_r}$ with respect to the Lebesgue measures in $\mathbb R^{d_1}, \mathbb R^{d_2}, \ldots, \mathbb R^{d_r}$, respectively.

	\subsection{Mixture Alternatives}
	
	The mixture alternatives constructed as follows: Given a $\delta > 0$ consider the  following mixture density in $\mathbb R^{d_0}$:  
	\begin{align*} 
		f (\bm x)= (1-\delta) \prod_{i=1}^r f_{X_i} (x_i) + \delta g (\bm x) , 
	\end{align*}
	where $x_i \in \mathbb R^{d_i}$ for $1 \leq i \leq r$, $\bm x = (x_1, x_2, \ldots, x_r) \in \mathbb R^{d_0}$, and $g \ne  \prod_{i=1}^r f_i$ is a probability density function with respect to the Lebsegue measure in $\mathbb R^{d_0}$ such that the following holds: 
	
	\begin{assumption}\label{assumption:mixture} The support of $g$ is contained in that of $\prod_{i=1}^{r}f_i$ and 
		$$0<\E\left[  \left(\frac{g(\bm X)}{ \prod_{i=1}^{r}f_{X_i}(X_i)} - 1 			\right)^2  \right]<\infty , $$ where 
		the expectation is taken under $H_0$, that is, $\bm X = (X_1, X_2, \ldots X_r) \sim \prod_{i=1}^{r}f_{X_i}$.  
			\end{assumption}
	
	Under this assumption contiguous local alternatives are obtained by considering local perturbations of the mixing proportion $\delta$ as follows: 
	\begin{equation}\label{eq:H0mixture}
		H_0:\delta = 0 \qquad \mathrm{versus} \qquad H_1:\delta = h/\sqrt{n},
	\end{equation}  
	for some $h > 0$.  This type of alternative captures all local additive perturbations from $H_0$ by compressing potential lower-order dependence among the variables in the functions $g$. The following theorem derives the distribution of $\RJdCov^2_n(\bm X; \bm C)$ under $H_1$ as in \eqref{eq:H0mixture}. 
	
	
	\begin{theorem}\label{thm:mixture} 
		Suppose Assumptions \ref{assumption:U} and  \ref{assumption:mixture} hold. Then under $H_1$ as in \eqref{eq:H0mixture}, as $n \rightarrow \infty$,  
		\begin{align} \label{eq:H0NRdCovS}
			\{ n \RdCov_n(\bm X_S): S \in \cT \} \dto \left\{ \int_{\mathbb R^{d_S}} |Z_S(\bm t) + \mu_S(\bm t) |^2 \prod_{i\in S} \d w_i \right\}_{S \in \cT} , 
		\end{align} 
		with $Z_S(\bm t)$ as defined in \eqref{eq:ZS} and 
		\begin{align}\label{eq:meanH1}
			\mu_S(\bm t)= h \cdot \E_{\bm X \sim H_0}\left[ \left(\frac{g(\bm X)}{ \prod_{i=1}^{r}f_{X_i}(X_i)} - 1\right) \prod_{i \in S} \left \{\E\left[e^{\i \langle t_{i}, R_{\mu_i}(X_i)\rangle} \right]-e^{\i\langle t_{i},R_{\mu_i}(X_i)\rangle} \right\}\right] , 
		\end{align} 
		where $\bm t = (t_i)_{i \in S}$ and $t_i \in \mathbb R^{d_i}$ for $i \in S$. 
		Consequently,  for any vector of non-negative weights $\bm C= (C_1, C_2, \ldots, C_r)$, under $H_1$ as in \eqref{eq:H0mixture}, 
		\begin{align}\label{eq:H0NRJdCov}
			n\RJdCov^{2}_n(\bm X; \bm C) \dto \sum_{s=2}^r C_s \sum_{\substack{S \subseteq \{1, 2, \ldots, r\} \\ |S| = s }} \int_{\mathbb R^{d_S}} |Z_S(\bm t) + \mu_S(\bm t)|^2 \prod_{i\in S} \d w_i . 
		\end{align} 
	\end{theorem}

	The proof of this theorem is given in Appendix \ref{sec:local_power_proof}. 
	Using this result we can derive the limiting local power of the test based on $\RdCov^2(\bm X; \bm C)$. Specifically, if $\cH_{\bm C}$ denotes the CDF of the limiting distribution in \eqref{eq:H0NRJdCov} then local asymptotic power of the test $\phi_n^{\mathrm{asymp}}(\bm C)$ (recall \eqref{eq:testasymptotic}) is given by 
	$$\lim_{n \rightarrow \infty} \E_{H_1}[\phi_n^{\mathrm{asymp}}(\bm C)]= 1- \cH_{\bm C}^{(h)}(c_\alpha).$$
	This implies, $\phi_n^{\mathrm{asymp}}(\bm C)$ (and similarly $\phi_n(\bm C)$ and $\tilde \phi_n(\bm C)$ in \eqref{eq:testrank} and \eqref{eq:phitest}, respectively) have non-trivial Pitman efficiency and is rate-optimal, in the sense that, 
	$$\lim_{|h| \rightarrow \infty }\lim_{n \rightarrow \infty} \E_{H_1}[\phi_n^{\mathrm{asymp}}(\bm C)] = 1.$$

	\subsection{Konijn Alternative} 
	
	The Konijn family of alternatives \cite{konijn1956power} for 2 random vectors can be defined as follows: Given two independent random vectors $X_1' \sim \mu_1 \in \cP_{ac}(\mathbb R^{d_1})$, $X_2' \sim \mu_1 \in \cP_{ac}(\mathbb R^{d_2})$, and $\delta > 0$, consider the law of  
	\begin{align}\label{eq:X1X2}
		\begin{pmatrix}
			X_1\\
			X_2
		\end{pmatrix}
		=
		\begin{pmatrix}
			(1- \delta) \bm I_{d_1} & \delta \bm M_1\\
			\delta \bm M_2 & (1-\delta) \bm I_{d_2}
		\end{pmatrix}
		\begin{pmatrix}
			X_1'\\
			X_2'
		\end{pmatrix} 
		\end{align}
		where $\bm M_1$ and $\bm M_2$ are $d_1\times d_2$ and $d_2\times d_1$ dimensional deterministic matrices, respectively. Note that when $\delta=0$, the random vectors $X_1, X_2$ are independent and as $\delta$ increases introduces dependence/correlation through the matrix $\bm M$. Gieser \cite{gieser1993new} showed that if one choses $\delta= h/\sqrt n$, for some $h > 0$, then the distributions of $(X_1, X_2)$ and $(X_1', X_2')$ are contiguous when $X_1'$ and $X_2'$ are elliptically symmetric distributions centered at $\theta_1 \in \mathbb R^{d_1}$ and $\theta_2 \in \mathbb R^{d_2}$ with covariance matrices $\Sigma_1$ and $\Sigma_2$, respectively. Here, we extend the Konijn family to multiple random vectors as follows:

		\begin{definition}\label{defn:XA} Given $r$ independent random vectors $X_1', X_2', \ldots, X_r'$, where $X_i' \sim \mu_i \in \cP_{ac}(\mathbb R^{d_i})$, for $1 \leq i \leq r$, and $\delta > 0$, consider the law of  
			\begin{align}\label{eq:matrixA}
				\bm X = 
				\begin{pmatrix}
					X_1\\
					X_2 \\ 
					\vdots \\ 
					X_r
				\end{pmatrix}
				=
				\begin{pmatrix}
					(1- \delta) \bm I_{d_1} & \delta \bm M_{1, 2} & \cdots & \delta \bm M_{1, r} \\
					\delta \bm M_{2, 1} & (1 - \delta) \bm I_{d_2} & \cdots & \delta \bm M_{2, r} \\ 
					\vdots & \vdots & \ddots & \vdots \\ 
					\delta \bm M_{r, 1}^\top & \delta  \bm M_{r, 2} & \cdots & (1-\delta) \bm I_{d_r} \\ 
				\end{pmatrix}
				\begin{pmatrix}
					X_1'\\
					X_2' \\ 
					\vdots \\ 
					X_r'
				\end{pmatrix} := \bm  A_{\delta} \begin{pmatrix}
					X_1'\\
					X_2' \\ 
					\vdots \\ 
					X_r' 
				\end{pmatrix} , 
			\end{align} 
			where $\bm M_{i, j}$, for $1 \leq i \ne j \leq r$, is a $d_i\times d_j$ a dimensional deterministic matrix.  
		\end{definition}

		Note that the matrix $\bm A_\delta$ in \eqref{eq:matrixA} is a $d_0 \times d_0$ times block matrix with $\bm I_{d_i}$ in the $i$-th diagonal block for $1 \leq i \leq r$ and $\bm M_{i, j}$ in the $(i, j)$-th block for $1 \leq i \ne j \leq r$. When $\delta=0$, the matrix $\bm A_\delta$ is identity and, hence, by continuity there exists a neighborhood $\Theta$ of  zero such that $\bm A_\delta$ is invertible for  $\delta \in \Theta$. Therefore, by a change of variable formula, for $\delta \in \Theta$, the density of $\bm X$ can be written as: 
		\begin{align}\label{eq:densitydelta}
		f_{\bm X}( \bm x | \delta)=\frac{1}{|\mathrm{det}(\bm  A_{\delta})|} f_{\bm X'}(\bm A_\delta \bm x) , 
		\end{align} 
		where $f_{\bm X'}(\cdot) = \prod_{i=1}^r f_{X_i'}(\cdot)$ is the density of $(X_1', X_2', \ldots, X_r')$. Hereafter, we will denote $f'_{\bm X}( \bm x | \delta) = \frac{\mathrm d}{\mathrm d \theta} f_{\bm X}(\bm x | \bm \theta )|_{\theta = \delta}$ and assume the following:

		\begin{assumption}\label{assumption:K} The family of distributions $\{f_{\bm X}( \cdot | \delta)\}_{\delta \in \Theta}$ satisfies the following:  
			
			\begin{itemize}
				
				\item The support of $f_{\bm X}( \cdot | \delta)$ does not depend on $\delta \in \Theta$. In other words, the set $\cS:= \{ \bm x \in \mathbb R^{d_0}: f_{\bm X}( \bm x | \delta) > 0 \}$ does not depend on $\delta$. 
				
				\item The map $\bm x \mapsto\sqrt{ f_{\bm X} (\bm x|0)}$ is continuously differentiable.  
								
				\item The Fisher information 
				$$I(\delta)=\int_{\mathbb R^{d_0}} \left(\frac{ f'_{\bm X}( \bm x | \delta) }{ f_{\bm X}( \bm x | \delta) } \right)^2 f_{\bm X'} (x)  \d x $$ 
				is well-define, finite, strictly positive, and continuous at $\delta=0$. 
				
			\end{itemize}		
			
		\end{assumption}

		Under the above assumption, we consider the following sequence of local hypotheses: 
		\begin{equation}\label{eq:KH0}
			H_0:\delta = 0 \qquad \mathrm{versus} \qquad H_1:\delta = h/\sqrt{n},
		\end{equation}  
		for some $h > 0$. We show in Lemma \ref{lm:KH0} that $H_0$ and $H_1$ as in \eqref{eq:KH0} are contiguous under Assumption \ref{assumption:K}. This assumption also ensures that the family $\{f_{\bm X}( \cdot | \delta)\}_{\delta \in \Theta}$ is quadratic mean-differentiable (QMD) (see \cite[Definition 12.2.1]{lr}) and, hence, Le Cam's theory of contiguity applies (see \cite[Chapter 12]{lr} and \cite[Chapter 6]{vdv2000asymptotic} ) .

		\begin{remark} Assumption \ref{assumption:K} is satisfied, for example, when $X_1, X_2, \ldots, X_r$ are centered elliptical distributions, under certain non-degeneracy conditions on the respective scale matrices (as discussed  \cite[Example 5.1]{shi2020consistent} for the 2 variables case). 
		\end{remark} 
		
		To describe the limiting distribution of $\RJdCov(\bm X; \bm C)$ under $H_1$, denote the likelihood ratio and its derivative as: 
		$$L(\bm x|\delta) := \frac{f_{\bm X}(\bm x| \delta)}{f_{\bm X}(\bm x|0)} = \frac{f_{\bm X}(\bm x| \delta)}{f_{\bm X'}(\bm x)} \quad \text{and} \quad L'(\bm x|\delta) := \frac{\grad f_{\bm X}(\bm x| \delta)}{f_{\bm X'}(\bm x)} , $$
		where $\grad f_{\bm X}(\bm x| \delta) = \frac{\partial}{\partial \delta} f_{\bm X}(\bm x| \delta)$.

		\begin{theorem}[Local power for Konijn alternatives]\label{thm:powerK} 
			Suppose Assumptions \ref{assumption:U} and  \ref{assumption:K} hold. Then under $H_1$ as in \eqref{eq:KH0}, as $n \rightarrow \infty$,  
			\begin{align} \label{eq:H0KRdCovS}
				\{ n \RdCov_n(\bm X_S): S \in \cT \} \dto \left\{ \int_{\mathbb R^{d_S}} |Z_S(\bm t) + \nu_S(\bm t) |^2 \prod_{i\in S} \d w_i \right\}_{S \in \cT} , 
			\end{align} 
			with $Z_S(\bm t)$ as defined in \eqref{eq:ZS} and 
			\begin{align}\label{eq:meanKH1}
				\nu_S(\bm t)= h \cdot \E_{\bm X \sim H_0}\left[ L'(\bm X| 0 ) \prod_{i \in S} \left \{\E\left[e^{\i \langle t_{i}, R_{\mu_i}(X_i)\rangle} \right]-e^{\i\langle t_{i},R_{\mu_i}(X_i)\rangle} \right\}\right] , 
			\end{align} 
			where $\bm t = (t_i)_{i \in S}$ and $t_i \in \mathbb R^{d_i}$ for $i \in S$. 
			Consequently,  for any vector of non-negative weights $\bm C= (C_1, C_2, \ldots, C_r)$, under $H_1$ as in \eqref{eq:H0mixture}, 
			\begin{align}\label{eq:H0KRJdCov}
				n\RJdCov^{2}_n(\bm X; \bm C) \dto \sum_{s=2}^r C_s \sum_{\substack{S \subseteq \{1, 2, \ldots, r\} \\ |S| = s }} \int_{\mathbb R^{d_S}} |Z_S(\bm t) + \nu_S(\bm t)|^2 \prod_{i\in S} \d w_i . 
			\end{align} 
		\end{theorem}

		\subsection{H\'ajek Representation and Proof Outline}

		The proofs of Theorems \ref{thm:mixture} and \ref{thm:powerK} are given in Appendix \ref{sec:local_power_proof}. One of the main ingredients of the proof is a {\it H\'ajek representation} for the empirical process associated with $\RdCov_n(\bm X_S)$, which shows that we can replace the empirical rank maps in the definition of $\RdCov_n(\bm X_S)$  with their population counterparts without altering its limiting distribution under $H_0$. Since this result could be of independent interest we summarized this as a proposition below. To describe the result, consider the following empirical process: 
		\begin{align}\label{eq:ZnS}
			Z_{S, n}(\bm t)
			& =  \frac{1}{n}\sum_{b=1}^{n}\left\{\prod_{i \in S} \left (\frac{1}{n}\sum_{a=1}^{n}e^{\i\langle t_i,\R_i(X_{i}^{(a)})\rangle}-e^{\i\langle t_i,\R_i(X_{i}^{(b)})\rangle} \right) \right\} , 
		\end{align} 
		where $S \in \cT$, $\bm t= (t_i)_{i \in S} \in \mathbb R^{d_S}$ and $t_i \in \mathbb R^{d_i}$, for $i \in S$. Recalling \eqref{eq:RdCovSn}, note that 
		\begin{align}\label{eq:RdCovZnS}
			\RdCov^2_n(\bm X_S) =  \int |Z_{S, n}( \bm t )|^2 \prod_{i \in S} \mathrm d w_i. \end{align} 
		One of the challenges of dealing with the process $Z_{S, n}(\bm t)$ is the dependence across the indices $i \in S$ arising from the empirical rank maps $\{\R_i(\cdot)\}_{i \in S}$. The H\'ajek representation result shows that this dependence is asymptotically negligible under $H_0$. Towards this, define the oracle version of the empirical process $Z_{S, n}$ where the empirical rank maps $\{\R_i(\cdot)\}_{i \in S}$ in replaced with their population analogues empirical rank maps $\{R_{\mu_i}(\cdot)\}_{i \in S}$
		\begin{align}\label{eq:ZS_oracle}
			Z_{S, n}^{\textrm{oracle}}(\bm t) 
			&
			= \frac{1}{n}\sum_{b=1}^{n}\left\{\prod_{i \in S} \left (\frac{1}{n} \sum_{a=1}^{n}e^{\i\langle t_i, R_{\mu_i} (X_{i}^{(a)})\rangle}-e^{\i\langle t_i R_{\mu_i}(X_{i}^{(b)})\rangle} \right) \right\} , 
		\end{align} 
		for $\bm t$ as in \eqref{eq:ZS}. We are now ready the state the  H\'ajek representation result which shows that the difference between $Z_{S,n}$ and its oracle version is $o_P(1/\sqrt n)$ under $H_0$.

		\begin{proposition}\label{ppn:ZnS}
			Suppose Assumption \ref{assumption:U} holds. Then under $H_0$ as $n \rightarrow \infty$, 
			$$\sqrt n \left(Z_{S, n}(\bm t) - Z_{S, n}^{\mathrm{oracle}}(\bm t) \right) = o_P(1),$$
			for all $S \in \cT$ and $\bm t \in \mathbb R^{d_S}$. 
		\end{proposition}
		
		Once the H\'ajek representation is established the limiting distribution of $\{\RdCov(\bm X_S)\}_{S \in \cT}$ under contiguous local alternatives can be obtained as follows: 
		
		\begin{itemize} 
			
			\item The first step is to show the alternatives considered in \eqref{eq:H0mixture} and \eqref{eq:KH0} are mutually contiguous with $H_0$. While this is well-known for the case of 2 variables and relatively straightforward to extend to multiple variables for the mixture alternative \eqref{eq:H0mixture}, for the Konjin alternative \eqref{eq:KH0} establishing contiguity is much more involved. Invoking Le Cam's second lemma \cite{vdv2000asymptotic}, we show in Lemma \ref{lm:KH0} that Konjin alternatives as in \eqref{eq:matrixA} are contiguous under Assumption \ref{assumption:K}.

			\item Next, we derive the joint distribution of the collection of empirical processes $\{\sqrt n  Z_{S, n}^{\mathrm{oracle}}(\bm t) \}_{S \in \cT}$ and the log-likelihood under $H_0$. Proposition \ref{ppn:ZnS} together with contiguity and Le Cam's third lemma \cite{vdv2000asymptotic} gives the joint distribution of $\{\sqrt n  Z_{S, n}^{\mathrm{oracle}}(\bm t) \}_{S \in \cT}$ under $H_1$ as in \eqref{eq:H0mixture} and \eqref{eq:KH0}. The results in \eqref{eq:H0NRdCovS} and \eqref{eq:H0KRdCovS} then follow by the representation in \eqref{eq:RdCovZnS} and the continuous mapping theorem. 
		\end{itemize}

		\begin{remark} 
			Note that, as expected, one obtains the limiting null distribution of $n \RJdCov_n(\bm X; \bm C)$ by substituting $h =0 $ in Theorems \ref{thm:mixture} and \ref{thm:powerK}. In fact the H\'ajek representation result provides an alternative strategy to proving of the asymptotic null distribution, which does require the combinatorial CLT for multiple permutations. Nevertheless, we present both approaches because the technical by-products emerging from them, namely the H\'ajek representation and the combinatorial CLT, are results of independent interest. In particular, the scope of the combinatorial CLT with multiple permutations goes beyond rank-based methods to other nonparametric independence tests, such as methods based on geometric graphs \cite{friedman1983graph}. 
		\end{remark}

		To validate the asymptotic results discussed above we present a finite sample  simulation comparing the empirical power of the $\RJdCov$ based test with test based on $\JdCov$ studied in \cite{chakraborty2019distance}. Towards this we consider the following 3 distributions: (1) the multivariate normal (MVN), (2) the third-order multivariate normal copula (MVN copula),\footnote{A random variable $X$ is said to follow a $p$-dimensional $r$-th order, for $r > 0$, multivariate normal copula if $X \stackrel{D} = Z^r$, where $Z \sim N_p(\bm \mu, \bm \Sigma)$} as in \cite[Example 1]{chakraborty2019distance},  and (3) the multivariate student distribution (MVT) with 5 degrees of freedom, with dimensions $d_1=d_2=d_3=2$. In all the 3 cases the mean vectors are set to zero and the covariance matrix 
		$\Sigma=((\sigma_{ij}))_{1 \leq i,j \leq 2}$ is $\sigma_{11} = \sigma_{22}=1$ and $\sigma_{12}=\sigma_{21} = 0.5$.  		We set the number of variables $r=3$, the matrices $\bm M_{i, j}= \bm I_2$, for $1 \leq i \ne j \leq 3$ and compute the matrix $\bm  A_{\delta}$ as in \eqref{eq:matrixA} by varying $\delta$ as follows: $ \delta \in \{0,0.4,0.8,1.2,1.6\}$ for the Gaussian case, $\delta \in \{0,0.1,0.2,0.3,0.4\}$ for the  Gaussian copula case, and $\delta \in \{0,0.25,0.5,0.75,1\}$ for the multivariate $t$-distribution. 
		Table \ref{fig:local_power} shows the empirical power (over 200 repetitions) with sample size $n=300$ as a function of the signal strength ($\delta$ varying as above), for tests based on the $\JdCov$ and the $\RJdCov$. The plots show that our method has comparable power with $\JdCov$  for the Gaussian and $t$-distributions, while it is significantly better than the $\JdCov$ in the more heavy-tailed Gaussian copula model, illustrating the attractive efficiency properties of  rank-based distribution-free methods.

		\begin{figure*}
			\centering
			\begin{subfigure}[h]{0.3\textwidth}
				\centering
				\includegraphics[width=\textwidth]{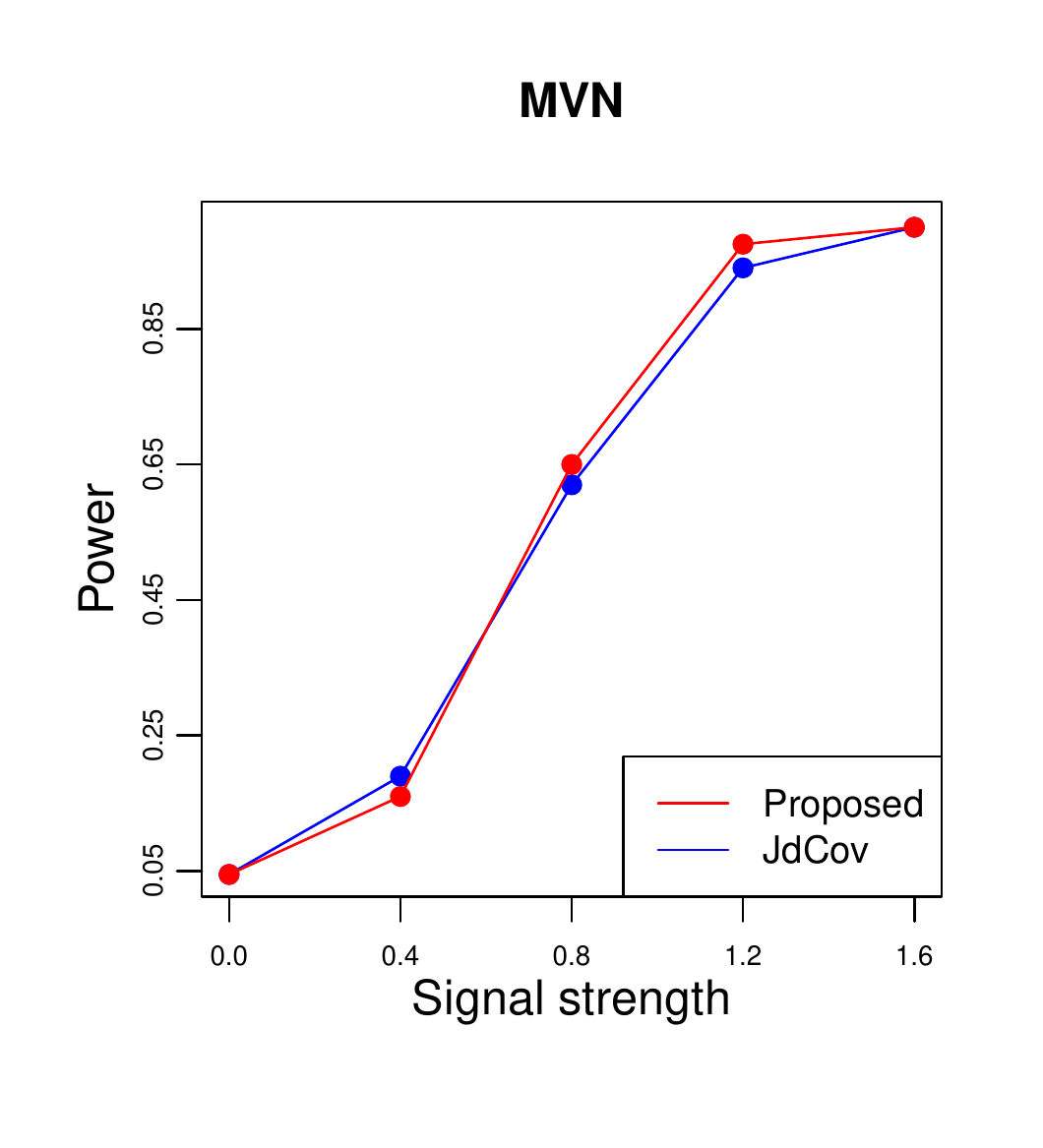} \\
				\small{(a)}
			\end{subfigure}
			\begin{subfigure}[h]{0.3\textwidth}  
				\centering
				\includegraphics[width=\textwidth]{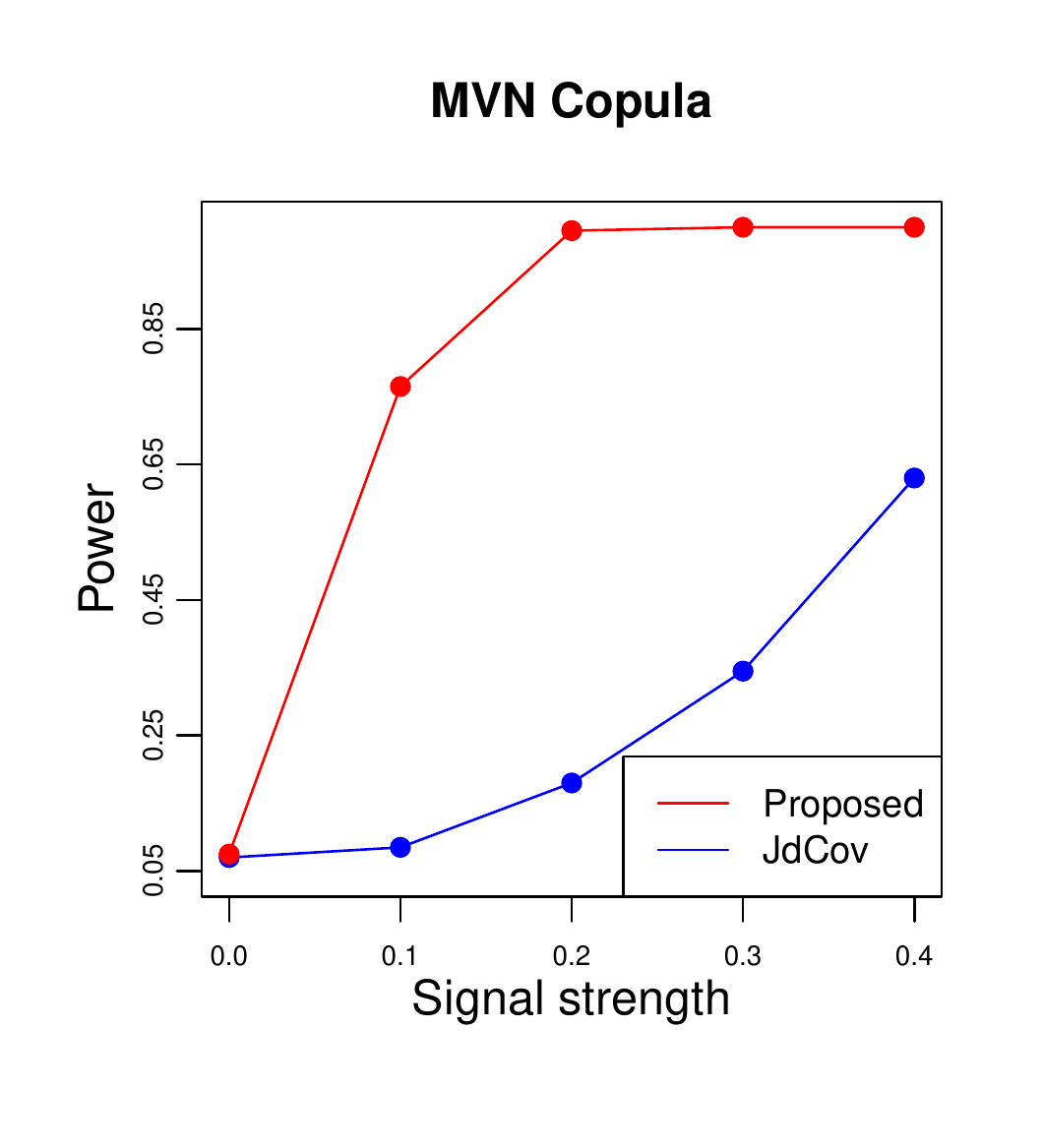} \\ 
				\small{(b)}
			\end{subfigure}
			\begin{subfigure}[h]{0.3\textwidth}  
				\centering
				\includegraphics[width=\textwidth]{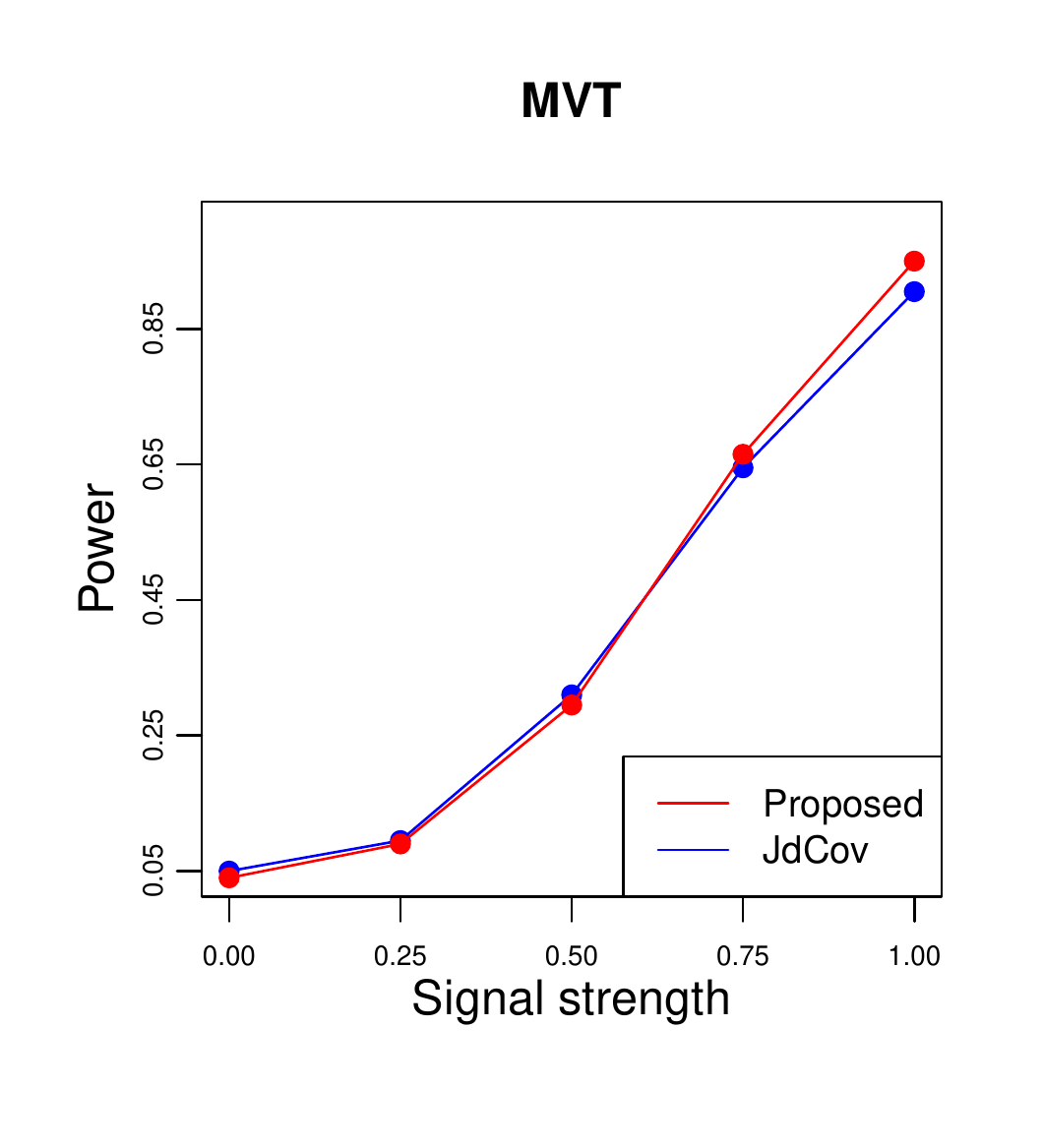} \\ 
				\small{(c)}
			\end{subfigure}
			\caption{\small{Empirical power of against Konijn alternatives with 3 variables: (a) bivariate normal, (b) bivariate normal copula model of order $3$, and (c) bivariate $t$-distribution with $5$ degrees of freedom. } }
			\label{fig:local_power} 
		\end{figure*}

		\section{Robust Independent Component Analysis}\label{sec:ICA}

		In this section, we apply the $\RJdCov$ measure to develop a robust method for independent component analysis (ICA).  
		Towards this, suppose $X \in \mathbb R^r$ be a random vector satisfying the following the moment conditions: 
			\begin{assumption}\label{assumption:Zmoment} The random vector $X$ has a continuous distribution with $\E[ X]=0$ and $\E[\|X\|^{2}]<\infty$. 
		\end{assumption}

The ICA model assume that there is a random vector $S\in\mathbb{R}^{r}$ with jointly independent coordinates such that $X = \bm M S$, where $\bm M$ is a nonsingular mixing matrix. Without loss of generality, we assume $S$ is standardized such that $\E[S] = 0$ and $\Var[S] = \bm I_{r}$. Then by the spectral decomposition:  
$$\Sigma_{X} = \Var[X] = \bm P \Lambda \bm P^\top ,$$ 
where $\bm P$ and $\Lambda$  are the $r \times r$ matrix of eigenvectors and the diagonal matrix of the corresponding eigenvalues of $\Sigma_{\bm X}$, respectively.  
Define $\bm O:=\Lambda^{-\frac{1}{2}} \bm P^{\top}$ and $Z = \bm O X$. Note that $\E[Z]=0$ and $\Var[Z] = \bm I_r$. Therefore, when the ICA model holds we can write 
\begin{align}\label{eq:WZ}
S=\bm M^{-1}X=\bm M^{-1} \bm O^{-1}Z= \bm W Z, 
\end{align} 
		where $\bm W= \bm M^{-1} \bm O^{-1}$. This implies, $\bm I=\Var[S]= \bm W \Var[Z] \bm W^{\top}=\bm W \bm W^{\top}$, that is, $\bm W$ is an orthogonal matrix.

		Following \citet{matteson2017independent} we restrict $\bm W\in\SO(r)$, where $\SO(r)$ denotes the set $r \times r$ orthogonal matrices with determinant 1 (the rotation group) and parameterize $\bm W$ by a vector $\theta$, where $\theta$  is a vectorized triangular array of rotation angles of length $r(r-1)/2$ indexed by $\{i, j: 1\leq i<j\leq r\}$, 
		as follows: 
					\begin{align}\label{eq:rotationrepresentation}
					\bm W = \bm W(\theta)= \bm Q^{(r-1)}\cdots \bm Q^{(1)}, \quad \text{with } \bm Q^{(i)}= \bm Q_{i,r}(\theta_{i,r})\cdots \bm Q_{i,i+1}(\theta_{i,i+1}) , 
		\end{align}
		where $\bm Q_{i,j}(\theta)$ is the identity matrix $\bm I_r$ with the $(i,i)$ and $(j,j)$ elements replaced by $\cos(\theta)$, the $(i,j)$ element replaced by $-\sin(\theta)$ and the $(j,i)$ element replaced by $\sin(\theta)$. To ensure the existence and uniqueness of the representation \eqref{eq:rotationrepresentation} define 
		\begin{align}\label{eq:rotationtheta}
			\Theta=\Bigg\{
			\theta_{i,j}:
			\begin{cases}
				0\leq \theta_{1,j}<2\pi, \\
				0\leq \theta_{i,j}<\pi, & i\neq 1.
			\end{cases}
			\Bigg\}.
		\end{align}
 Then by \cite[Theorem 2.3.2]{matteson2008statistical} there exists a unique inverse map of $\bm W\in \mathcal{SO}(r)$ into $\theta\in \Theta$ such that the mapping is assured to be continuous if either all elements on the main-diagonal of $W$ are positive, or all elements of $W$ are nonzero  (see \cite{matteson2017independent,matteson2008statistical} for more details on such decompositions).

Another identification issue with the ICA model is the sign and order of the independent components. Suppose $\bm P_{\pm}$ is a signed permutation matrix and then note that $X=\bm M S=\bm M \bm P^{\top}_{\pm} \bm P_{\pm}S$ in which case $\bm P_{\pm}S$ are the new independent components and $\bm M \bm P^{\top}_{\pm}$ is the new mixing matrix. This necessitates considering metrics that are invariant to the scale, sign, and order of the independent components, when comparing different estimates (see Section \ref{sec:ICAsimulation} for details).

		In general the ICA model \eqref{eq:WZ} is misspecified, in which case 
		$\theta$ is chosen such that the components of $\bm W(\theta) Z$ are as close to mutually independent as possible. To this end, the $\RJdCov$ provides a robust  measure for quantifying deviations from joint independence. Hence, one choose $\theta$ by minimizing the following objective function:  	
		 \begin{align}\label{eq:minimizerICA}
			\theta_0={\arg\min}_{\theta}\RJdCov^{2}(\bm W(\theta)Z; c).
		\end{align} 
		Recall that the (population) rank map in the 1-dimensional case is the distribution function, hence, $\RJdCov^{2}(\bm W(\theta)Z; c)$ can be computed  by replacing $R_{\mu_i}$ in \eqref{eq:RdCov} with $F_{W_i(\theta)^{\top}Z}$,  the distribution function of $F_{W_i(\theta)^{\top}Z}$, for $1 \leq i \leq r$, where  $W_i(\theta)$ denotes the transpose of $i$-th row of $\bm W(\theta)$.

		\subsection{$\RJdCov$-Based ICA Estimator} 
		
		We now discuss how one can estimate $\theta$ based on samples by minimizing the empirical version of \eqref{eq:minimizerICA}. Towards this, suppose $\bm X_n = \{X_1, X_2, \ldots X_n\}$ be i.i.d. samples distributed as $X$ satisfying Assumption \ref{assumption:Zmoment}. We define approximately uncorrelated observations $\hat{\bm Z}_n = \bm X_n \hat{\bm O}^\top$, where $\hat {\bm O}_n$ is the sample covariance matrix  of $\bm X_n$.  Assumption \ref{assumption:Zmoment} implies that $\Var[\hat{\bm Z}_n] \stackrel{a.s.} \rightarrow \bm I_r$. To simplify notation, we omit the above steps described and assume $\bm Z_n = \{Z_1, Z_2, \ldots Z_n\}$ are i.i.d. samples distributed as $Z$, which has zero mean vector and identity covariance matrix. Also, we arrange  $ \{Z_1, Z_2, \ldots Z_n\}$ such that the $i$-th row of $\bm Z_n$ is $Z_i^\top$, that is, $\bm Z_n$ has dimension $n \times r$.


		 To estimate $\theta$ from $\bm Z_n$ we replace $\RJdCov^{2}$ in \eqref{eq:minimizerICA} with its empirical version
		\begin{align}\label{eq:thetaestimation}
			\hat{\theta}_n={\arg\min}_{\theta}\RJdCov^{2}_n(\bm Z_n \bm W(\theta)^\top; c) 
		\end{align} 
			Note that for any value of $\theta$ the objective function $\RJdCov^{2}_n(\bm Z_n \bm W(\theta)^\top; c)$ can be computed easily, since in dimension 1 the empirical rank map is the empirical distribution function. Specifically, the empirical rank map for the $i$-th component is  $\hat F_{i, n}(t|\theta) = \frac{1}{n} \sum_{a=1}^n \bm 1\{ W_{i}^{\top}(\theta)Z_a \leq t \}$. To establish the consistency of $\hat{\theta}_n$ we need the following definition:

		\begin{definition}
			We say $\hat{\theta}_n$ {\it is consistent for $\theta$ over $\mathcal{SO}(r)$}  if the function $\bm W(\theta)$ is continuous in $\theta$ and $\mathcal{D}(\bm W(\hat{\theta}_n), \bm W(\theta))\stackrel{a.s.}{\rightarrow}0$, as $n\rightarrow\infty$, for any metric $\mathcal{D}$ on $\mathcal{SO}(r)$ satisfying $\mathcal{D}(\bm W, \bm A)=0$ if and only if there exists a signed permutation matrix $\bm P_{\pm}$ such that $\bm W= \bm P_{\pm} \bm A$.
		\end{definition}

		With this definition we can now state the consistency result of $\hat{\theta}_n$.  
		Towards this, recall the definition of $\Theta$ from \eqref{eq:rotationtheta} and 
		let $\bar{\Theta}$ denote a sufficiently large compact subset of the space $\Theta$.

		\begin{theorem}\label{thm:independentcomponent}
			Suppose Assumption \ref{assumption:Zmoment} holds and the density function of $Z$ is uniformly bounded. Moreover, assume there exists a unique minimizer $\theta_0\in\bar{\Theta}$ for \eqref{eq:minimizerICA} and $\bm W(\theta)$ satisfies the conditions for a unique continuous inverse to exist, then $\hat{\theta}_n$ is consistent for $\theta_0$ over $\mathcal{SO}(d)$. 
		\end{theorem}

		The proof of Theorem \ref{thm:independentcomponent} is given in Appendix \ref{sec:icapf}. This shows that the proposed estimator $\hat \theta_n$ is consistent for $\theta_0$ for metrics which are invariant to the non-identifiability issues of the ICA problem.  The assumption of uniform boundedness is required to guarantee the distribution function is Lipschitz, which ensures the equicontinuity of the objective function.  
		
		\subsection{Practical Implementation}\label{sec:implementation}
		
		One of the challenges in directly implementing gradient based methods to minimize the objective in \eqref{eq:thetaestimation} is that, the empirical distribution functions are not differentiable with respect to $\theta$. \citet{matteson2017independent} suggested a practical way to circumvent this issue by estimating the univariate distribution functions $F_{W_i(\theta)^{\top}Z}$, for $1 \leq i \leq r$, using kernel-based  methods. Specifically, we consider the following problem: 
		\begin{align}\label{eq:kernelICA}
			\tilde{\theta}_n=\arg\min_{\theta}\JdCov^2_n( (\tilde{F}_1(\bm Z_n W_1(\theta)),\ldots,\tilde{F}_r(\bm Z_n W_r(\theta)) ;c), 
		\end{align}
		with  $\tilde{F}_1(\bm Z_nW_1(\theta)) = (\check{F}_1(W_1(\theta)^{\top} Z_1),\ldots,\tilde{F}_1(W_1(\theta)^{\top} Z_n))^\top $ and 
		\begin{align}\label{eq:cdfG}
		     \tilde{F}_i(s)=\frac{1}{n}\sum_{a=1}^nG\bigg(\frac{s-W_i(\theta)^{\top}Z_a}{h_{n}(i)}\bigg) , 
		\end{align}
		where $G$ is the integral of a density kernel and $h_{n}(i)$ is a data dependent bandwidth for the $i$-th component, for $1 \leq i \leq r$ (see Section 2 in \cite{matteson2017independent} for practical choices of $h_{n}(i)$). The optimization problem \eqref{eq:kernelICA} can now be solved using a gradient-based method (see Section \ref{sec:gradICA} for the computation of the gradient).  
		
		The advantages of using kernel-based estimates is that they have nice theoretical properties, which can be used to establish the consistency of the estimate $\tilde{\theta}_n$. In particular, if we assume that $h_{n}(i)$ is a positive measurable function of the $i$-th coordinates of $Z_1, Z_2, \ldots, Z_n$ such that $h_{n}(i)\stackrel{a.s.}{\rightarrow}0$ as $n\rightarrow\infty$, then 		
			\begin{align*}
				\sup_{s\in\mathbb{R}}\big|\tilde{F}_i(s)-F_i(s)\big|\stackrel{a.s.}{\rightarrow}0, 
			\end{align*}
	as $n \rightarrow \infty$, for $1 \leq i \leq r$ (see Corollary 1 in \cite{chacon2010note}). 
%
%
			Then whenever $G$ is Lipschitz and Assumption \ref{assumption:Zmoment} and the conditions in Theorem \ref{thm:independentcomponent} holds, similar to the proof Theorem 2.1 in \cite{matteson2017independent}, by replacing their objective function with $\RJdCov_n^2$, it can be shown that $\tilde{\theta}_n$ is consistent for $\theta_0$ over 
$\mathcal{SO}(r)$.

		\section{Simulation Study}\label{sec:sim}

		In this section, we showcase the efficacy of the proposed test, both in terms of Type-I error control and power, across a range of simulation settings. For our test we implement the finite sample version discussed in Section \ref{sec:finitesample} with cut-off chosen as in \eqref{eq:cn} and all the weights set to 1. 
		We compare our method (hereafter referred to as $\RJdCov$) with the JdCov based test in \cite{chakraborty2019distance}, the dCov based test studied by Matteson and Tsay \cite{matteson2017independent} (hereafter denoted as  MT), and the dHSIC based test \cite{pfister2018kernel}. The section is organized as follows: In Section \ref{subsec:type_I_error} we present the results on Type-I error control. The empirical power for testing joint independence is given in Section \ref{sec:powerindependence} and results for detecting higher-order dependence are in Section \ref{sec:higher}. The performance of the ICA estimator is studied in Section \ref{sec:ICAsimulation}.

		\subsection{Type-I Error Control}\label{subsec:type_I_error}
		
		To illustrate the finite sample Type-I error control of the proposed test  we consider $\bm X = (X_1^\top, X_2^\top, X_3^\top)^\top$ in the following three settings: 
		\begin{itemize} 
		
		\item[$(1)$] {\it Multivariate Gaussian}: $X_1, X_2, X_3$ are i.i.d. $N_3(\bm 0, \bm I_3)$. 
		
				\item[$(2)$] {\it Third-order Gaussian copula}: $X_1, X_2, X_3$ are independent random vectors in $\mathbb R^3$ with i.i.d. coordinates distributed as $N(0, 1)^3$. 
				
				\item[$(3)$] {\it Cauchy}: $X_1, X_2, X_3$ are independent random vectors in $\mathbb R^3$ with i.i.d. coordinates distributed as $\mathcal C(0, 1)$, the Cauchy distribution with location parameter 0 and scale to parameter 1. 
		
		\end{itemize} 
		
		Table \ref{tab:type_I} shows the empirical Type-I error probability (over 500 repetitions) with sample size $n=500$ for the different tests in the above 3 situations. Our test satisfactorily controls Type-I error in finite samples, validating the theoretical results. The other tests also have good Type-I error control.

	\small 	
		\begin{table}[h]
			\centering
			\caption{\small{Empirical probability of Type-I error of the different tests. } }
			\label{tab:type_I}
			\begin{tabular}{cccc}
				\toprule  
				Distribution
				& Multivariate Gaussian & Gaussian Copula  & Cauchy \\
				\midrule  
				Proposed & 0.038  & 0.042  & 0.042\\
				\addlinespace[0.7mm]
				JdCov & 0.036 & 0.030& 0.046 \\
				\addlinespace[0.7mm]
				Matteson-Tsay (MT) & 0.056 & 0.046& 0.056 \\
				\addlinespace[0.7mm]
				dHSIC & 0.032 & 0.074& 0.050 \\
				\bottomrule 
			\end{tabular}
		\end{table}
		
	\normalsize	
		
		\subsection{Power Results} \label{sec:powerindependence}
		
		In this section we compare the empirical power of $\RJdCov$ with the existing methods in the 4 settings described below. In the first 2 cases we have (light-tailed) multivariate Gaussian distributions with covariance structure as in \cite[Example 2]{chakraborty2019distance}, the third setting is a (heavy-tailed) Cauchy regression model, and the fourth is a Cauchy sine dependence model (similar to \cite{zhang2018large}). Throughout we set the sample size $n=500$ and compute the empirical power over 500 repititions.

		\begin{figure*}
			\centering
			\begin{subfigure}[h]{0.49\textwidth}
				\centering
				\includegraphics[width=\textwidth]{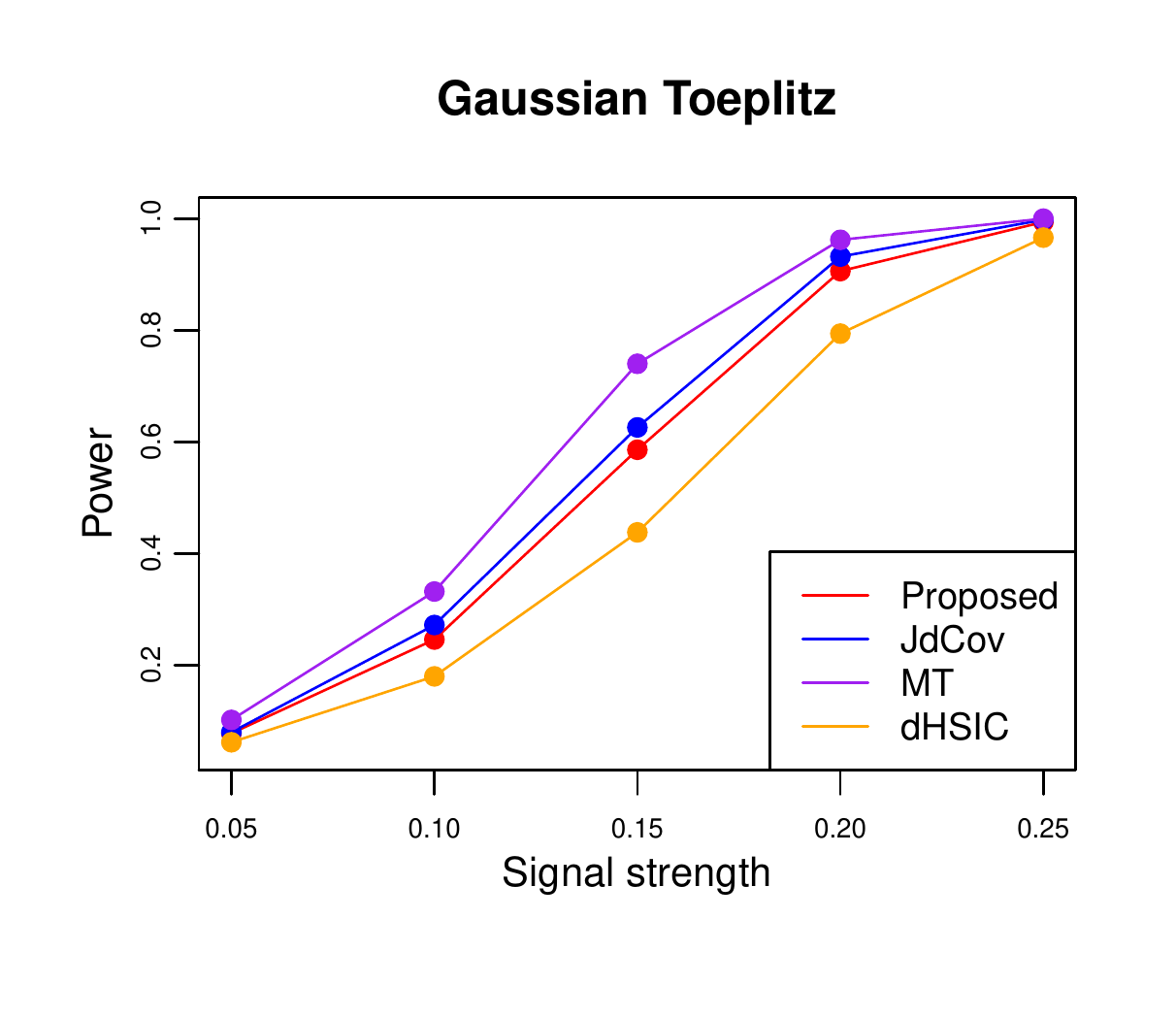} \\
				\small{(a)}
			\end{subfigure}
			\begin{subfigure}[h]{0.49\textwidth}  
				\centering
				\includegraphics[width=\textwidth]{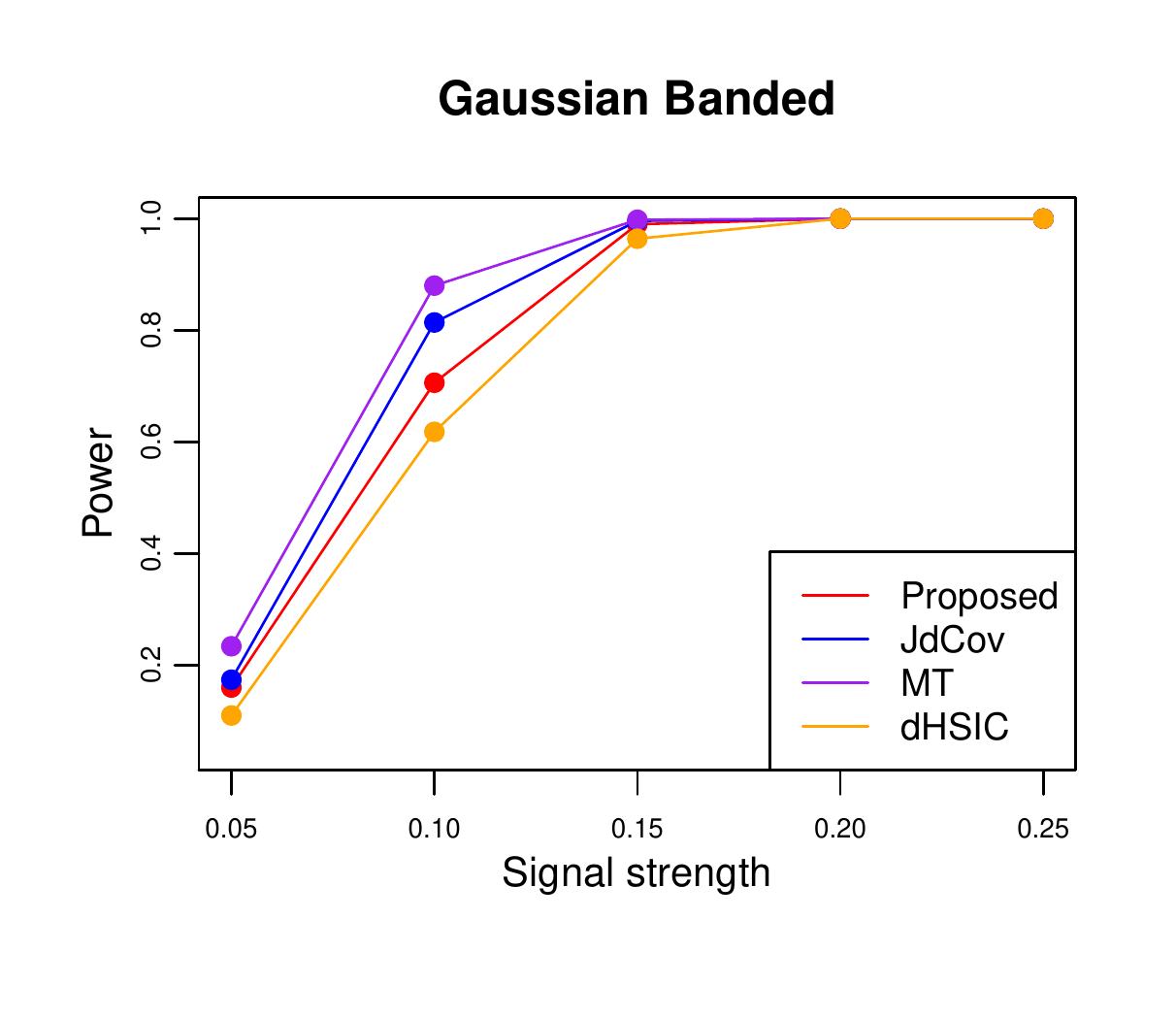} 
				\small{(b)}
			\end{subfigure}
			\caption{\small{Power of the different tests in (a) the multivariate Gaussian model with autoregressive correlation structure and (b) the multivariate Gaussian model with banded covariance matrix. } } 
			\label{fig:powernormal}
		\end{figure*}

		\begin{itemize} 
		
		\item[$(1)$] {\it Multivariate Gaussian with autoregressive covariance structure}: In this case the data is generated as follows: $\bm X \sim N_9(\bm 0, \Sigma)$ where $\Sigma = ((\sigma_{ij}))_{1 \leq i, j \leq 9}$ has an autocorrelation structure, that is, $\sigma_{ij} = \rho^{|i-j|}$. Then the first 3 coordinates of $\bm X$ are set as $X_1$, second 3 coordinates as $X_2$, and the last 3 coordinates as $X_3$. Figure \ref{fig:powernormal} (a) shows the empirical power of the different tests as 
		$\rho$ varies between $0.05$ to $0.25$ over a grid of 5 equally spaced values. 
		
				\item[$(2)$] {\it Multivariate Gaussian with banded covariance matrix}: Here, $\bm X \sim N_9(\bm 0, \Sigma)$ where $\Sigma = ((\sigma_{ij}))_{1 \leq i, j \leq 9}$ has a banded structure, that is, $\sigma_{ii}=1$ for $1 \leq i \leq 9$, $\sigma_{ij}=\rho$ for $1\leq |i-j|\leq 2$, and $\sigma_{ij}=0$ otherwise. As before, the first 3 coordinates of $\bm X$ are set as $X_1$, second 3 coordinates as $X_2$, and the last 3 coordinates as $X_3$. Figure \ref{fig:powernormal} (b) shows the empirical power of the different tests as 
		$\rho$ varies between $0.05$ to $0.25$ over a grid of 5 equally spaced values.

				\item[$(3)$] {\it Cauchy regression}: In this case $\bm X = (X_1^{\top},X_2^{\top},X_3^{\top})^{\top}$ is generated as follows: 
		\begin{align}\label{eq:hvy_tail_noise}
			X_i=Z_i+ a\cdot V , 
		\end{align} 
		for $1 \leq i \leq 3$,  where $Z_i \in \mathbb R^3$ have i.i.d. coordinates distributed according to $\mathcal{C}(0,1)$, and $V=(V_1,V_2,V_3)^\top$ with $V_1=V_2=V_3\sim\mathcal{C}(0,1)$. Note that the variables $X_1, X_2, X_3$ are mutually dependent because of the shared noise vector. Figure \ref{fig:powercauchysine} (a) shows the empirical power of the different tests as $a$ varies between $0.2$ to 1 over a grid of 5 equally spaced values. 
	
	\item[$(4)$] {\it Sine dependence}: In this case $\bm X = (X_1^{\top},X_2^{\top},X_3^{\top})^{\top}$ is generated as in \eqref{eq:hvy_tail_noise}, but with a different noise distribution. Specifically, we fix  $a=1$, generate $Z_i \in \mathbb R^3$ with i.i.d. coordinates distributed according to $\mathcal{C}(0,1)$ as before, but change $V=(V_1,V_2,V_3)^\top$ as  $V_1=V_2=V_3=\sin(b\cdot W)$ where $W\sim\mathcal{C}(0,1)$. Figure \ref{fig:powercauchysine} (b) shows the empirical power of the different tests as $b$ varies between $0.1$ to $0.5$ over a grid of 5 equally spaced values.  	
		\end{itemize}

We observe from the plots in Figure \ref{fig:powernormal} that the $\RJdCov$ along with the other methods have very similar performance in the Gaussian settings. On the other hand, Figure \ref{fig:powercauchysine} shows in the heavy-tailed Cauchy and sine dependence settings the $\RJdCov$ significantly outperforms the other tests. Overall the $\RJdCov$ emerges as the preferred method by exhibiting good power over a range of data distributions.

		\begin{figure*}
			\centering
			\begin{subfigure}[h]{0.49\textwidth}  
				\centering
				\includegraphics[width=\textwidth]{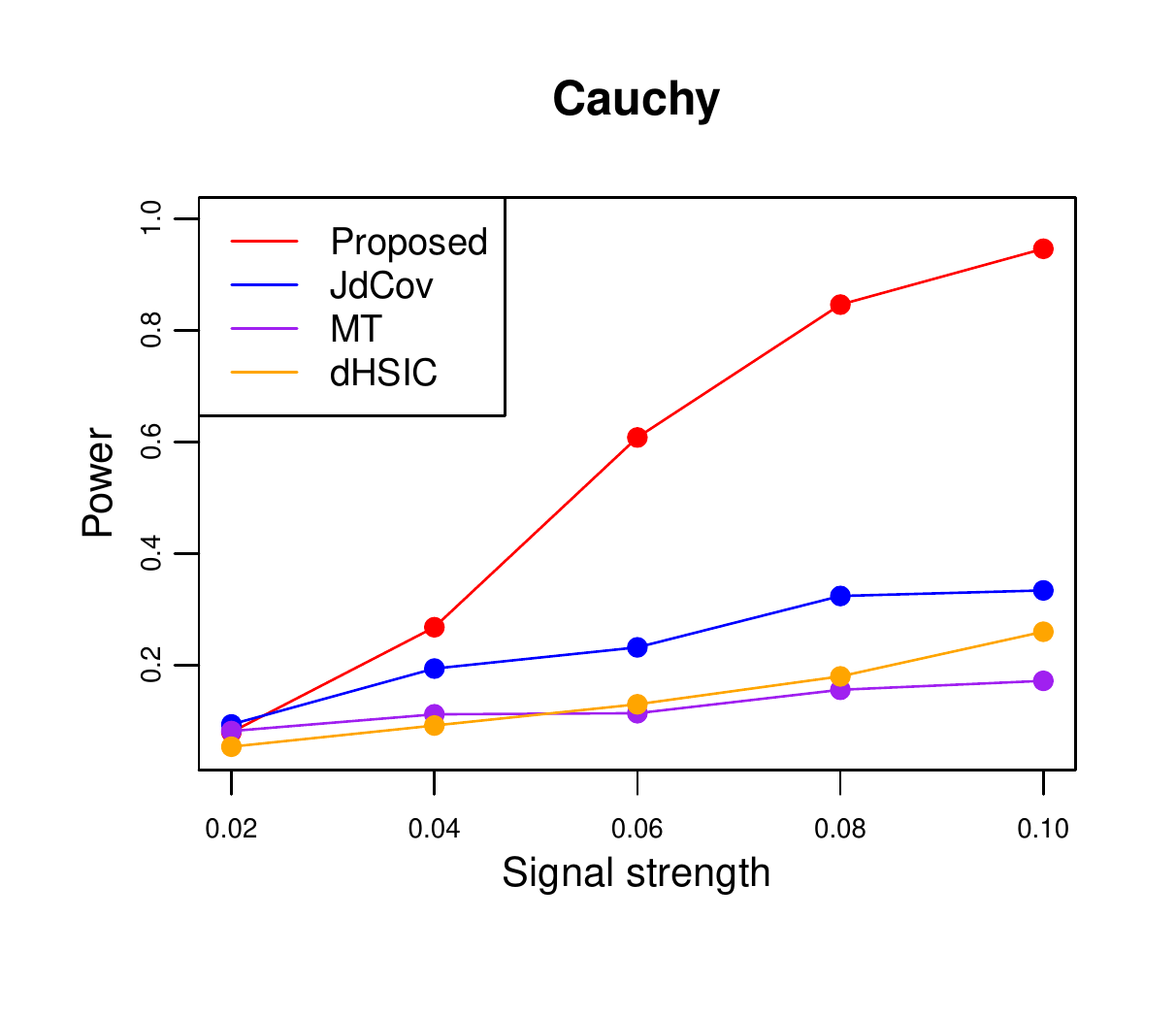} \\ 
				\small{(a)}
			\end{subfigure}
			\begin{subfigure}[h]{0.49\textwidth}  
				\centering
				\includegraphics[width=\textwidth]{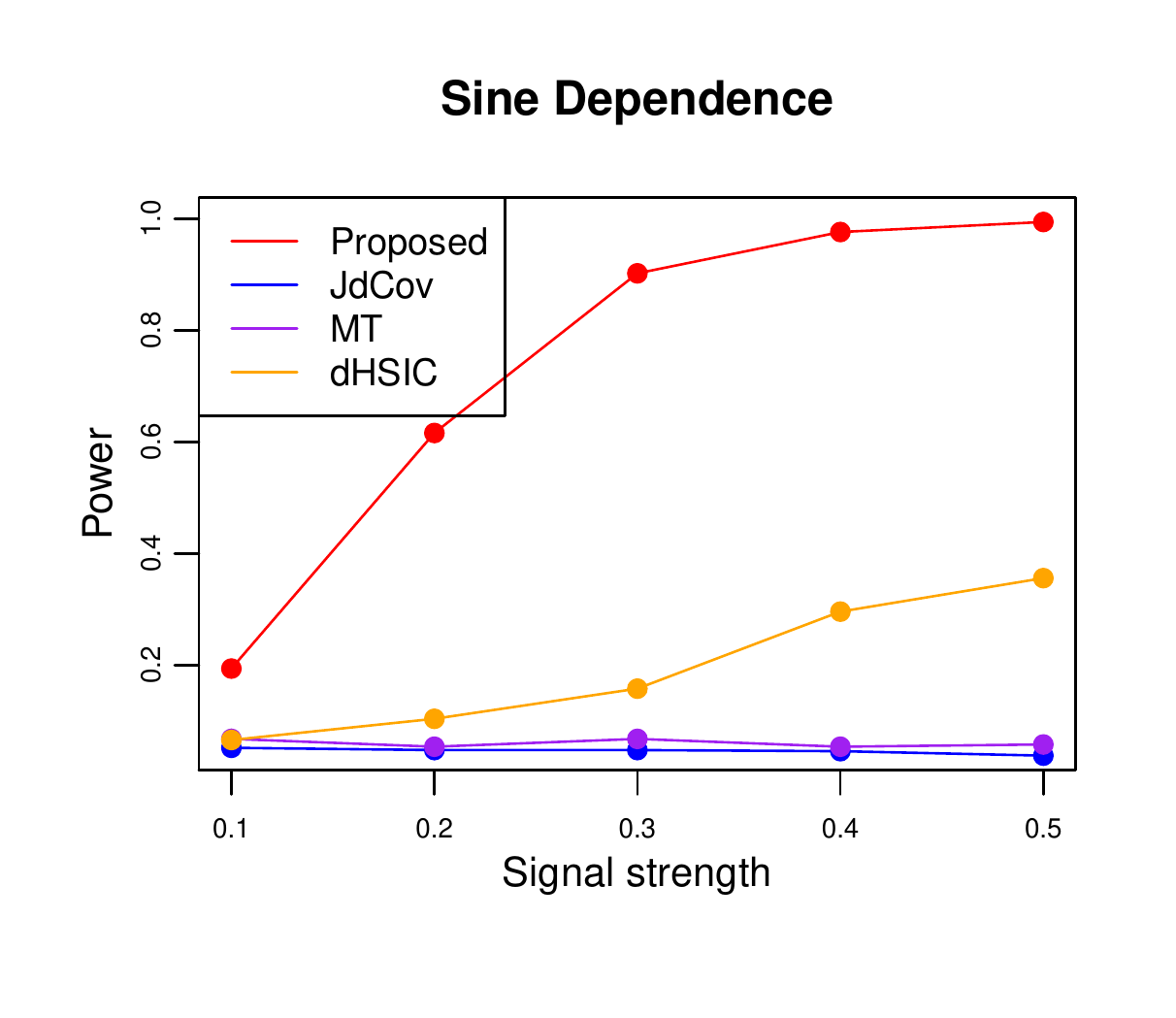} \\ 
				\small{(b)}
			\end{subfigure}
			\caption{\small{Power of the different tests in (a) the Cauchy regression model and (b) the sine dependence model. } } 
			\label{fig:powercauchysine}
		\end{figure*}

		\subsection{Higher-Order Dependence}\label{sec:higher}
		
		In this section we consider situations where the variables are pairwise independent but jointly dependent. Specifically, generate $(X, Y, Z')$ as follows: 		\begin{itemize}
			\item Suppose $X$, $Y$ and $Z$ are independent random vectors with i.i.d. coordinates distributed according to $F$, where $F$ is a distribution in $\mathbb R$ that is symmetric around $0$. 
			\item For $1 \leq s \leq d-1$, set $Z_s'=Z_s$ and set $Z'_d=-Z_d$ if $X_d Y_d Z_d$ is positive, otherwise set $Z'_d=Z_d$. 
					\end{itemize}
		It can easily be checked that $X,Y,Z'$ are pairwise independent. However, they are mutually dependent, since, for example, $Z_d'X_dY_d$ is always non-negative. 
			 In the simulations we choose $F$ to  the following 4 distributions: (1) $N(0, 1)$, (2) the $t$-distribution with 2 degrees of freedom $(t_{(2)})$, (3) the $t$-distribution with 3 degrees of freedom $(t_{(3)})$, and  (4) the Cauchy distribution $\mathcal C(0, 1)$, equivalently, the the $t$-distribution with 1 degree of freedom.  
			 
			 		 In Table \ref{table:higherorder} we show the empirical power (over 500 iterations) with sample size $n=500$ of the following 3 tests (and their corresponding versions based on $\dCov$/$\JdCov$):  
		\begin{itemize}
		\item The test for pairwise independence  that rejects for large values of 
		$$\RdCov^2_n(X, Y) + \RdCov^2_n(Y, Z') + \RdCov^2_n(X, Z').$$ 
		\item The test for higher-order independence that rejects for large values of $\RdCov^2_n(X, Y, Z')$. 
		\item The test for joint independence based on $\RJdCov_n^2$.  
		\end{itemize}  
The results in Table \ref{table:higherorder} show that for all the 4 distributions considered both the $\RJdCov$ and $\JdCov$ based tests successfully detect the null hypothesis of pairwise independence (probability of Type-I error). However, the  $\JdCov$ based test fails to detect higher-order dependence for the Cauchy distribution and joint dependence for the $t_{(2)}$ and the Cauchy distribution. This is not unexpected because the consistency of the $\dCov$ based rely on certain moment assumptions which are not satisfied by heavy-tailed distributions like the $t_{(2)}$ and the Cauchy (recall that $t_{(2)}$ does not have finite variance and $\mathcal C(0, 1)$ does not have finite mean). On the other hand, the $\RJdCov$ based tests easily detects both the higher-order and joint dependences in the all 4 cases, which once again highlights the robustness and broad application of the proposed method.  

\small 
\begin{table}[h]
	\centering
	\caption{\small{ Empirical power for detecting higher-order dependence. } }
	\label{table:higherorder}
	\begin{center}
		\begin{tabular}{ c|c|c|c|c|c|c } 
			\hline
			Dependence& \multicolumn{2}{c|}{Pariwise} &  \multicolumn{2}{c|}{Higher-Order } & \multicolumn{2}{c}{Joint}  \\
			\hline
			Tests& JdCov & Proposed & JdCov & Proposed & JdCov & Proposed\\ 
			\hline
			Gaussian & 0.058&0.054 & 1&1 &1.000&0.942 \\ 
			$t_{(3)}$ & 0.052 & 0.064 &1&1&0.868& 0.904 \\ 
			$t_{(2)}$ &0.056&0.040&1&1&0.098&0.888\\ 
			Cauchy&0.064&0.044&0.04&1&0.050&0.788\\
			\hline
		\end{tabular}
	\end{center}
\end{table} 

\normalsize
		
		\subsection{Performance of the ICA Estimator}\label{sec:ICAsimulation}
		
		To evaluate the performance of the distribution-free ICA estimator proposed in Section \ref{sec:ICA}, we consider the following model:  
		\begin{align*}
			X = \bm M Z , 
		\end{align*}
		where $Z=(Z_{1}, Z_2, \ldots, Z_d)$, $Z_1, Z_2,  \ldots, Z_d$ i.i.d. for some distribution $F$, and $M$ is a $d \times d$ random mixing matrix with condition number between 1 and 2, using the R package \texttt{ProDenICA}.\footnote{\url{https://cran.r-project.org/web/packages/ProDenICA/index.html}} We consider 12 candidate distributions for $F$ as described Table \ref{table:distributionICA}. For each candidate distribution we apply the gradient descent algorithm to solve the optimization program \eqref{eq:kernelICA} (where the gradient is derived in Section \ref{sec:implementation}). We compare our results with the method proposed in \citet{matteson2017independent}. As in \cite{matteson2017independent} the estimation error is measured in the following metric, which is designed to resolve the non-identifiability issues: 
		\begin{align*}
			D(\hat{M},M)=\frac{1}{\sqrt{d-1}}\inf_{C\in\mathcal{C}}\|C\hat{M}^{-1}M- \bm I_d\|_{F}
		\end{align*} 
		where 
		\begin{itemize} 
		
		\item $\hat{M}$ is the estimated mixing matrix, 
		
		\item $\|\cdot\|_{F}$ denotes the Frobenius norm,  
		
		\item $\mathcal{C}=\{C\in\mathcal{M}:C=P_{\pm}B\text{ for some }P_{\pm} \text{ and }B\}$, where $\mathcal{M}$ is the the set of $d\times d$ nonsingular matrices, $P_{\pm}$ is a $d\times d$ signed permutation matrix, and $B$ is a $d\times d$ diagonal matrix with positive diagonal elements. 
		\end{itemize}		
		The computation for metric $D$ is available in the R package \texttt{JADE} \cite{nordhausen2014jade}. We set dimension $d=3$ and the sample size $n=500$ and present the results in \ref{fig:ica}, where the indices in $x$-axis denote different distributions as described in Table \ref{table:distributionICA}.  The results in Table \ref{table:distributionICA} show that our estimator have smaller estimation error and variance except for distribution (f).
		
		\small 
					
						\begin{table}[h]
						\centering
						\caption{\small{Description of the distributions for the ICA simulations and the means and standard deviations of the estimation errors. } } 
						\label{table:distributionICA}
						\begin{tabular}{c|c|c|c}
							& & \multicolumn{2}{c}{Estimation Error (Standard Deviation)} \\
                            \hline
                            Index & Distribution & Proposed & MT \\
							\hline 
							$\mathsf{(a)}$ & $\mathsf{N(0,1)^3}$ & 0.129 (0.077) & 0.192 (0.094)\\
							$\mathsf{(b)}$ & $\mathsf{N(0,1)^5}$ & 0.170 (0.109) & 0.209 (0.120)\\
							$\mathsf{(c)}$ & $\mathsf{Gamma(5,1)}$ & 0.067 (0.028) & 0.073 (0.049)\\
							$\mathsf{(d)}$ & $\mathsf{Gamma(10,1)}$ & 0.127 (0.069) & 0.148 (0.088)\\
							$\mathsf{(e)}$ & $\mathsf{\frac{3}{10}Exp(1) + \frac{7}{19} Exp(5)}$ & 0.034 (0.046) & 0.038 (0.055)\\
							$\mathsf{(f)}$ & $\mathsf{\frac{3}{10}N(-2,1) + \frac{7}{10}N(2,1)}$ & 0.157 (0.121) & 0.181 (0.120)\\
                            $\mathsf{(g)}$ & $\mathsf{Uniform(0,1)}$ & 0.189 (0.110) & 0.310 (0.132)\\
							$\mathsf{(h)}$ & $\mathsf{\frac{7}{10}N(-2,3) + \frac{3}{10}N(2,1)}$ & 0.115 (0.117) & 0.174 (0.134)\\
							$\mathsf{(i)}$ & $\mathsf{\frac{5}{10}N(-2,2) + \frac{5}{10}N(2,2)}$ & 0.390 (0.105) & 0.435 (0.104)\\
							$\mathsf{(j)}$ & $\mathsf{(\frac{5}{10}N(-2,2) + \frac{5}{10}N(2,2))^3}$ & 0.121 (0.085) & 0.205 (0.093)\\
							$\mathsf{(k)}$ & $\mathsf{(\frac{5}{10}N(-2,2) + \frac{5}{10}N(2,2))^5}$ & 0.191 (0.097) & 0.236 (0.107)\\
							$\mathsf{(l)}$ & $\mathsf{(\frac{5}{10}N(-2,2) + \frac{5}{10}N(2,2))^7}$ & 0.131 (0.084) & 0.157 (0.101)\\
						\end{tabular}
					\end{table}

				\normalsize

		\begin{figure*}
			\centering
			\begin{subfigure}[h]{0.49\textwidth}
				\centering
				\includegraphics[width=\textwidth]{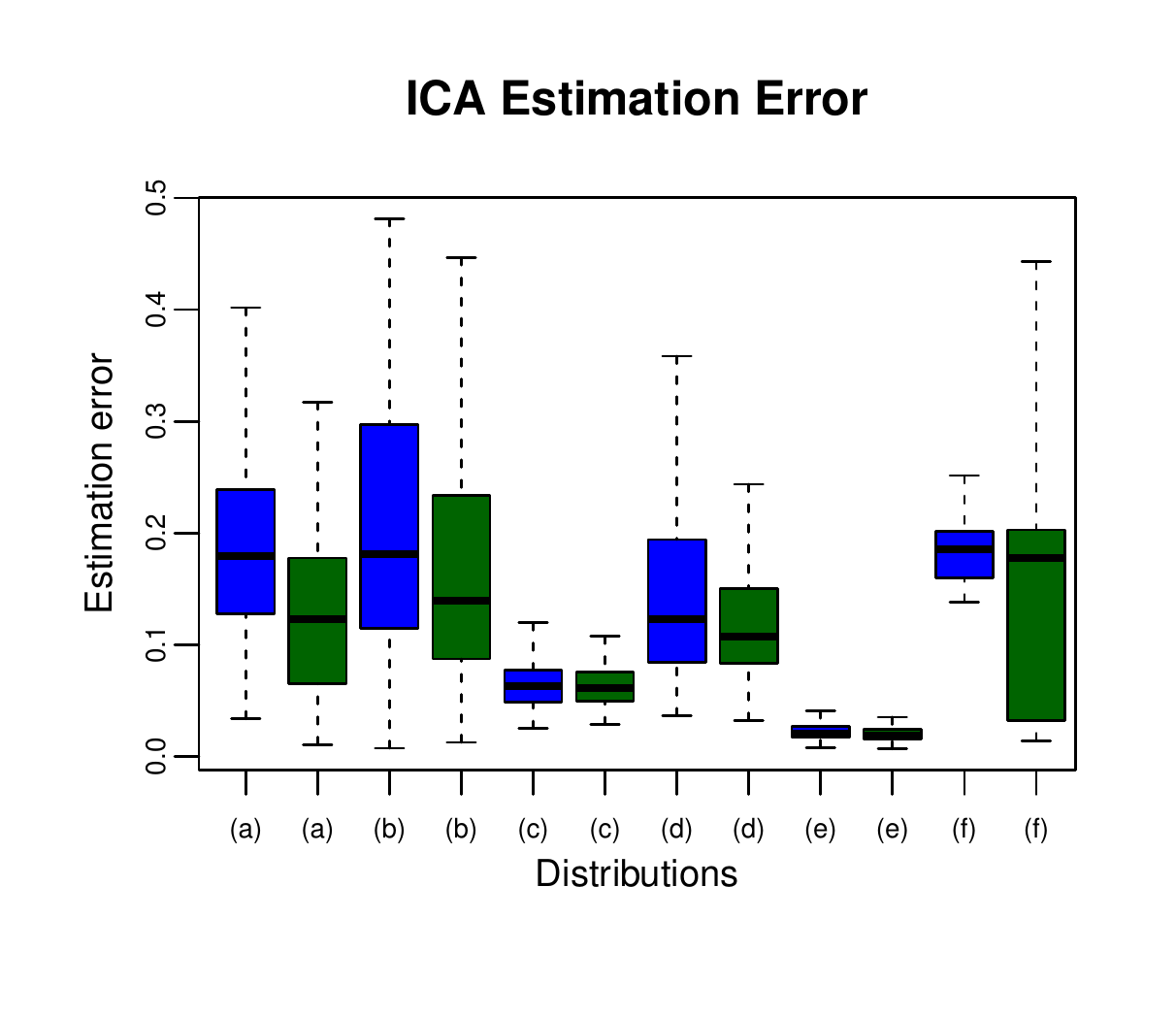}
			\end{subfigure}
			\begin{subfigure}[h]{0.49\textwidth}  
				\centering 
				\includegraphics[width=\textwidth]{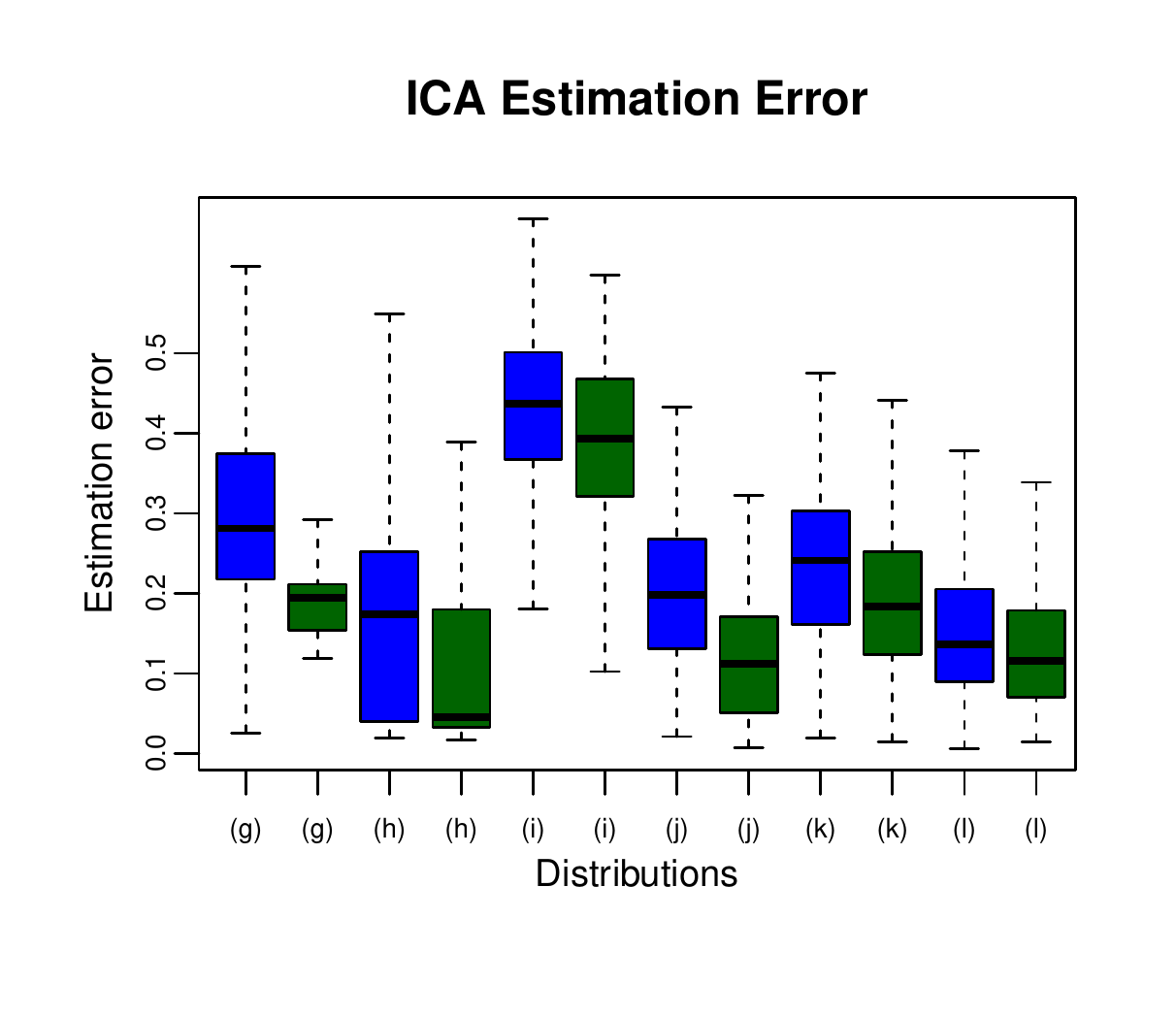}
			\end{subfigure}
			\caption{\small{ ICA estimation error across 12 distributions: The blue boxplots correspond to the Matteson-Tsay estimator and the green boxplots corresponds to the proposed estimator based on $\RJdCov$. } } 
			\label{fig:ica}
		\end{figure*}

		\section{Real data application}\label{sec:realdata}

		In this section, we apply our method to the stock price data available in R package \texttt{gmm}. The data can be accessed using the command  \texttt{data(finance)}. The dataset has prices for 24 stocks corresponding to 20 companies in US starting from January 5th, 1993 to January 30th, 2009. Since the dataset has several missing entries before 2003, we use data from the 2003-2009 period for 9 of these companies (described below) for our analysis. To account for the temporal dependence, we average the stock prices over the week, after which we obtain $n=300$ data points. We categorize the different companies into 6 different industries according to North American Industry Classification System (NAICS) as follows: (1) Real Estate (RE) category which contains the company NHP, 
		(2) the Food and Retail (FR) category which includes the companies WMK and GCO; (3) Manufacture (M) category including ROG, MAT, and VOXX; (4) the Finance (F) category with the company FNM; (5) Communications Systems (JCS); and (6) Zoom Technologies (ZOOM). We created 2 separate categories for the companies JCS and ZOOM because their  stock prices seem to vary independently of the other companies, possibly due to their relatively small scales, even though they both belong to the information industry.

		\subsection{Results and Analysis}

		To understand the dependency structure between the 6 categories (industries) described above, we first test the pairwise dependence of the stock prices between different industries. Then we test for higher-order dependencies between triples of categories for which the null hypothesis of pairwise independence is accepted. We report the results for the tests based on $\RdCov/\RJdCov$ and $\dCov/\JdCov$ for comparison. Throughout we set the significance level to $0.05$.
	

	The $p$-values for the pairwise independence tests are summarized in Table \ref{table:pairwisedependencestockprice}.  Note that this entails performing ${6 \choose 2} = 15$, so the $p$-values are adjusted using the BH-procedure.  To better visualize the results, we plot the dependency graph between 6 categories  using function \texttt{dependence.structure} in R package \texttt{multivariance} \cite{bottcher2020dependence} and the results are shown in Figure \ref{fig:finance}. From these we can see that hypothesis of pairwise independence is accepted by $\RdCov$ for the following 3 triples: (I) Real Estate-JCS-ZOOM, (II) Finance-JCS-ZOOM, and (III) Manufacture-JCS-ZOOM. On the other hand, $\dCov$ only accepts the pairwise independence hypothesis for the triple: Real Estate-JCS-ZOOM when using $\dCov$. This might be due to the small scale of companies JCS and ZOOM, which tends to make them independent of each other and among others. Moreover, observe that the Manufacture and Food and Retail categories are prone to be dependent with all the others, which is also reasonable because their services are essential to the other industries.
		
		\small 

                   \begin{table}[h]
						\centering
						\caption{\small{$p$-values for pairwise independence testing among the 6 categories. } }
						\label{table:pairwisedependencestockprice}
						\begin{tabular}{l|r|r}
							& \multicolumn{2}{c}{$p$-values}\\
                            \hline 
                            Industry-Industry & JdCov & Proposed\\
							\hline 
                            RE-JCS & 0.096 & 0.136 \\
							RE-ZOOM & 0.188 & 0.107\\
							RE-FR & 0.002 & 0.000\\
							RE-M & 0.002 & 0.000\\
							RE-F & 0.002 & 0.000\\
							JCS-ZOOM & 0.152 & 0.136\\
                            JCS-FR & 0.002 & 0.004\\
                            JCS-M & 0.002 & 0.258\\
                            JCS-F & 0.102 & 0.408\\
                            ZOOM-FR & 0.008 & 0.010\\
                            ZOOM-M & 0.012 & 0.058\\
                            ZOOM-F & 0.032 & 0.968\\
                            FR-M & 0.002 & 0.000\\
                            FR-F & 0.002 & 0.000\\
                            M-F & 0.002 & 0.000\\
						\end{tabular}
					\end{table}

 \normalsize

		The $p$-values for testing third-order independence for the 3 triples for which the hypothesis of pairwise independence are accepted are given in Table \ref{table:higherorderstockprice}. We observe from Table \ref{table:higherorderstockprice} that the third-order $\dCov$ based $p$-value, which tests for the null hypothesis of joint independence of RE, JCS, and ZOOM given their pairwise independence, accepts the null hypothesis at level $0.05$. However,  the $\JdCov$ based $p$-value, which tests for the null hypothesis of joint independence of RE, JCS, and ZOOM (without any assumptions on pairwise independence), rejects the null hypothesis of joint independence at $0.05$. Given that the pairwise independence hypothesis has been accepted, for consistent hierarchical interpretability one would have ideally expected the third-order $\dCov$ and the $\JdCov$ to have the same conclusions. 
		On the other hand, the conclusions from the $\RdCov$ and $\RJdCov$ $p$-values are hierarchically consistent. Moreover, Figure \ref{fig:finance}(b) shows that there is higher-order dependence between JCS, ZOOM, and Manufacture industries when using $\RdCov$ and $\RJdCov$. This may be because JCS and ZOOM, and Manufacturing will typically serve as  downstream industries but direct links between them are not obvious due to the scales of the companies.

		\begin{figure*}
			\centering
			\begin{subfigure}[h]{0.49\textwidth}
				\centering
				\includegraphics[width=\textwidth]{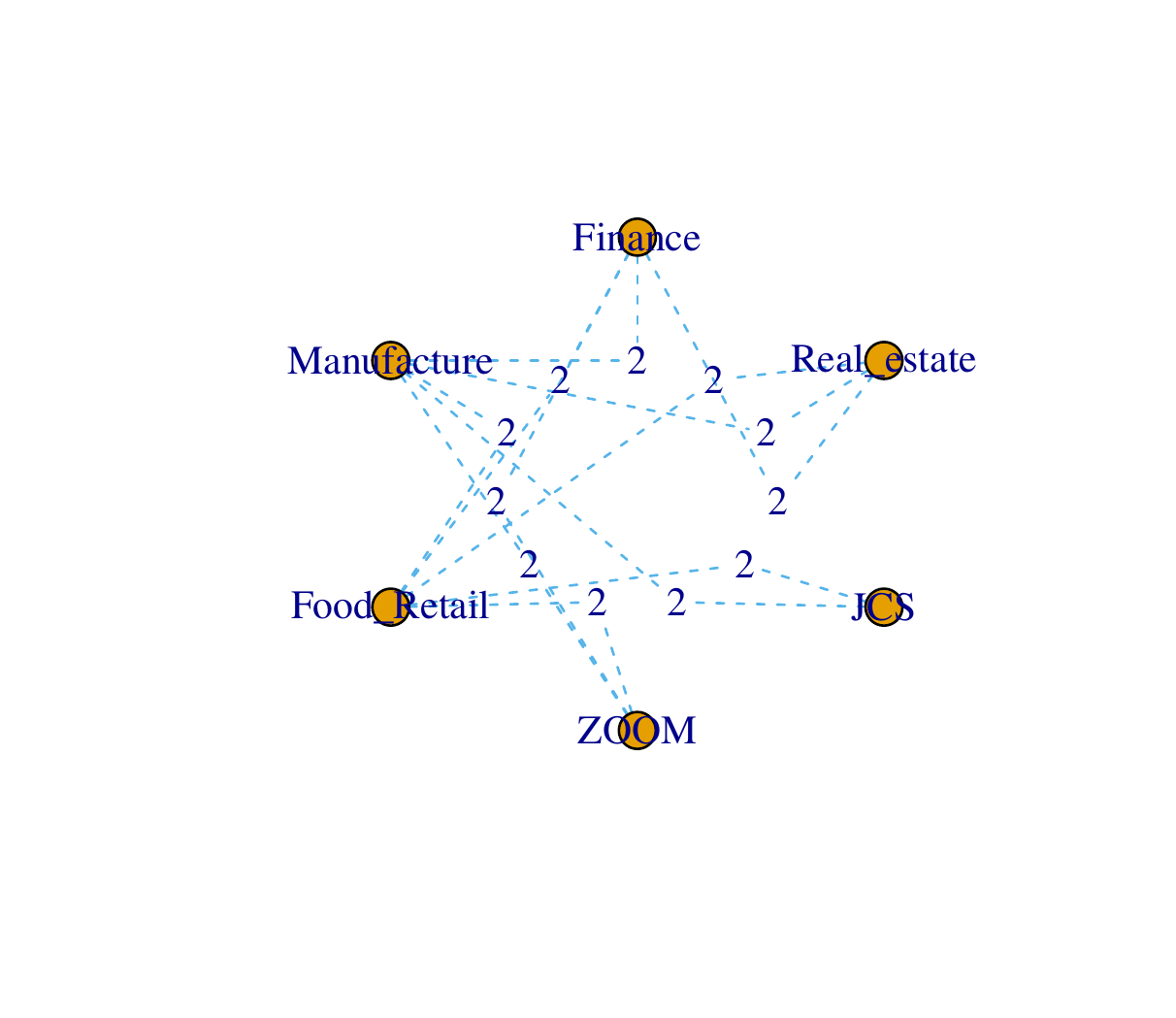} 
				(a) 
			\end{subfigure}
			\begin{subfigure}[h]{0.49\textwidth}  
				\centering 
				\includegraphics[width=\textwidth]{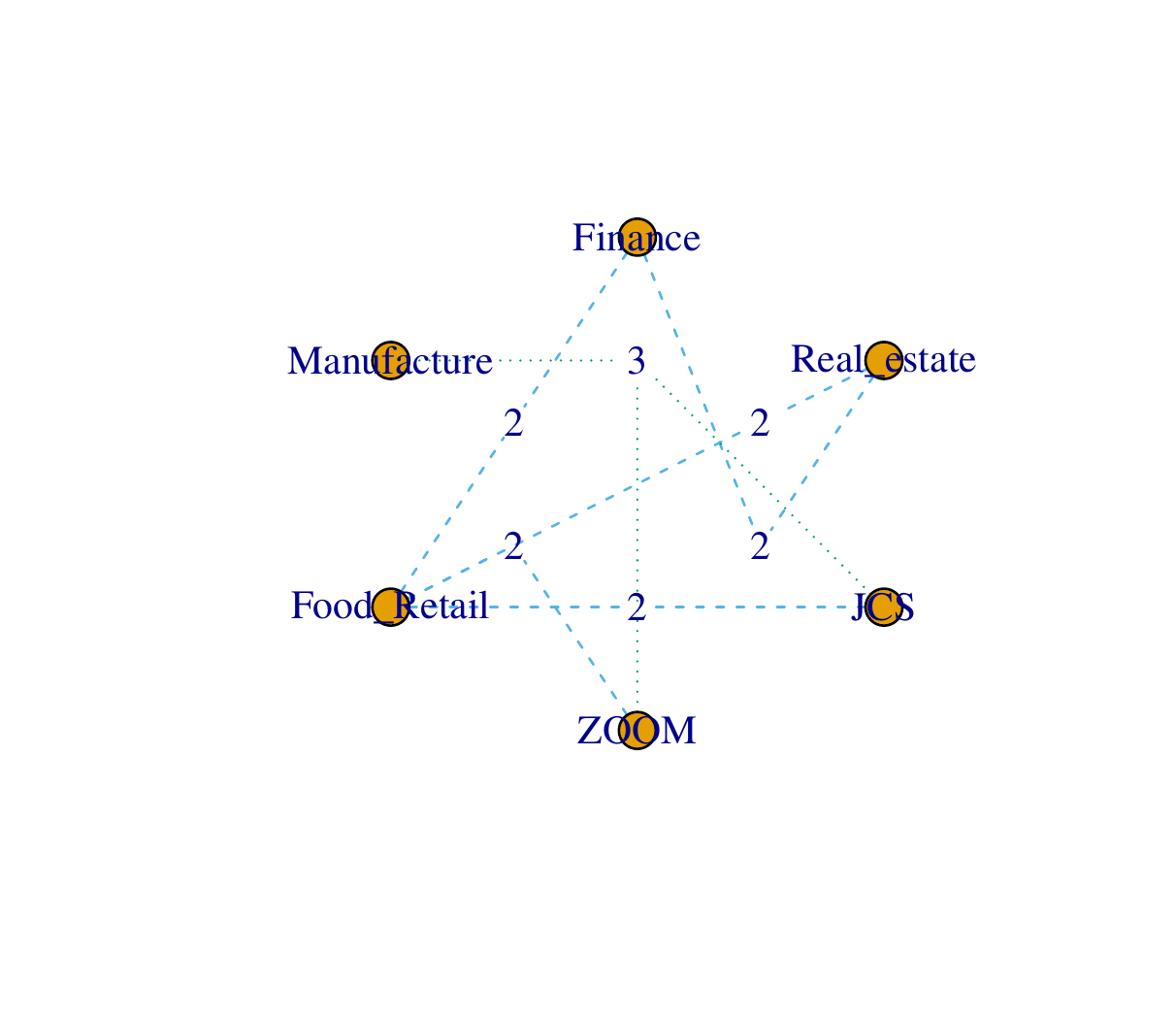} 
				(b) 
			\end{subfigure}
			\caption{\small{ Stock prices data: (a) Dependency graph obtained using higher-order $\dCov$, (b)  Dependency graph obtained using higher-order $\RdCov$. } } 
			\label{fig:finance}. 
		\end{figure*}

		\small

		\begin{table}[h]
			\centering
			\caption{\small{$p$-values for third order independence given pairwise independence. } }
			\label{table:higherorderstockprice}
			\begin{tabular}{l|r|r|r|r}
				& $\dCov$ & $\JdCov$ & $\RdCov$ & $\RJdCov$ \\
				\hline Real estate-JCS- ZOOM & 0.101 & 0.020 & 0.187 & 0.063 \\
				Finance-JCS-ZOOM & &  & 0.274 & 0.425\\
				Manufacture-JCS-ZOOM &  &  & 0.003 & 0.006\\
			\end{tabular}
		\end{table}
		
		\normalsize

\small{\subsection*{Acknowledgment} The authors are grateful to Bj{\"o}rn B{\"o}ttcher and David Matteson for sharing their codes and datasets. BBB was partly supported by NSF CAREER Grant DMS-2046393 and a Sloan research fellowship.

		\small

		\bibliographystyle{abbrvnat}
		\bibliography{ref}
		
		\normalsize

		\appendix

		\section{Combinatorial CLT with multiple permutations}
		\label{sec:clt}
		
		In this section we derive an analogue of Hoeffding's combinatorial CLT \cite{hoeffding} for tensors with multiple random permutations. We begin with the following assumption:

		\begin{assumption}\label{assumption:A} Fix $r \geq 2$ and suppose $\bm A_n =((a_{i_1, i_2, \ldots, i_r}))_{1 \leq i_1, i_2, \ldots, i_r \leq n}$ is a sequence of $r$-tensors satisfying the following conditions: 
			
			\begin{itemize} 
				
				\item[$(1)$] For all $1 \leq s \leq r$ and $1 \leq i_1, i_2, \ldots, i_{s-1}, i_{s+1}, \ldots, i_r \leq n$, 
				\begin{align}\label{eq:sumA}
					\sum_{i_s=1}^n a_{i_1, i_2, \cdots, i_r}=0 .
				\end{align}
				
				\item[$(2)$] There is a universal constants $K_1, K_2> 0$ such that $|a_{i_1, i_2, \ldots, i_r}| \leq K_1/\sqrt n$, for all $1 \leq i_1, i_2, \ldots, i_r \leq n$, and $\sum_{1 \leq i_1, i_2, \ldots, i_r \leq n} a_{i_1, i_2, \ldots, i_r}^2 \geq K_2 n^{r-1}$. 
			\end{itemize}
		\end{assumption}

		Under this assumption we have the following theorem: 
		
		\begin{theorem}\label{thm:clt} Fix $r \geq 2$ and suppose $\bm A_n =((a_{i_1, i_2, \ldots, i_r}))_{1 \leq i_1, i_2, \cdots, i_r \leq n}$ is a sequence of $r$-tensors satisfying Assumption \ref{assumption:A}. Consider $r-1$ independent random permutations $\pi_1, \pi_2, \ldots, \pi_{r-1}$ of the set $\{1,\ldots,n\}$ and define 
			$$C_n:=\sum_{i=1}^{n}a_{i, \pi_1(i), \ldots, \pi_{r-1}(i)}.$$ 
			Then, as $n \rightarrow \infty$, 
			\begin{align*}
				\frac{C_n}{\sqrt{\Var[C_n]}}  \dto N(0,1).
			\end{align*}
		\end{theorem}
		
		The proof of Theorem \ref{thm:clt} is given in Section \ref{thm:cltpf}. We first show that by centering $\bm A_n$, condition (1) in Assumption \ref{assumption:A} holds without loss of generality.

		\begin{remark}\label{remark:A} 
			Suppose $\bm A_n$ is an $r$-tensor not satisfying \eqref{eq:sumA}, then we can construct a tensor $\tilde{\bm A_n} = (( \tilde{a}_{i_1, i_2, \ldots, i_r}))_{1 \leq i_1, i_2, \ldots, i_r \leq n}$ which satisfies Assumption \ref{assumption:A}. We illustrate this for $r=3$. For $1 \leq i_1, i_2, i_3 \leq r$, define 
			\begin{align*}
				\tilde{a}_{i_1, i_2, i_3} = a_{i_1, i_2, i_3} - a_{\bullet, i_2, i_3} - a_{i_1, \bullet, i_3}
				-a_{i_1, i_2, \bullet} + a_{i_1, \bullet, \bullet} + a_{i_2, \bullet, \bullet} + a_{i_3, \bullet, \bullet} - a_{\bullet, \bullet, \bullet}
			\end{align*} 
			where 
			\begin{itemize} 
				
				\item $a_{\bullet, i_2, i_3} = \frac{1}{n} \sum_{i_1=1}^n a_{i_1, i_2, i_3}$, $a_{i_1, \bullet, i_3} = \frac{1}{n} \sum_{i_2=1}^n a_{i_1, i_2, i_3}$,  $a_{i_1, i_2, \bullet} = \frac{1}{n} \sum_{i_3=1}^n a_{i_1, i_2, i_3}$;
				
				\item $a_{i_1, \bullet, \bullet} = \frac{1}{n^2} \sum_{1 \leq i_2, i_3 \leq n} a_{i_1, i_2, i_3}$, $a_{\bullet, i_2, \bullet} = \frac{1}{n^2} \sum_{1 \leq i_1, i_3 \leq n} a_{i_1, i_2, i_3}$, $a_{\bullet, \bullet, i_3} = \frac{1}{n^2} \sum_{1 \leq i_1, i_2 \leq n} a_{i_1, i_2, i_3}$; 
				
				\item $a_{\bullet, \bullet, \bullet} = \frac{1}{n^3} \sum_{1 \leq i_1, i_2, i_3 \leq n} a_{i_1, i_2, i_3} $. 
				
			\end{itemize}
			Clearly, $\tilde{\bm A_n} = (( \tilde{a}_{i_1, i_2, i_3}))_{1 \leq i_1, i_2, i_3 \leq n}$ satisfies \eqref{eq:sumA}. 
		\end{remark}

		\subsection{Proof of Theorem \ref{thm:clt}}
		\label{thm:cltpf} 
		
		For notational convenience we will present the proof of Theorem \ref{thm:clt} for $r =3$. The proof can be easily extended to general $r$ by following the arguments. Hereafter, we will assume $\bm A_n =((a_{i_1, i_2, i_3}))_{1 \leq i_1, i_2, \cdots, i_r \leq n}$ is a sequence of $r$-tensors satisfying Assumption \ref{assumption:A}. We begin by computing $\Var[C_n]$. 
		
		\begin{lemma}\label{lm:variance} Suppose Assumption \ref{assumption:A} holds. Then 
			\begin{align*}
				\Var[C_n]=\frac{n-2}{n(n-1)^{2}} \sum_{1 \leq i_1, i_2, i_3 \leq n} a_{i_1, i_2, i_3}^{2} = \Theta(1).
			\end{align*} 
		\end{lemma} 
		
		\begin{proof}
			Note that 
			\begin{align}\label{eq:varianceC}
				\Var[C_n]=\sum_{i=1}^{n}\Var[a_{i, \pi_1(i), \pi_2(i)}] + \sum_{1 \leq i_1 \neq i_1' \leq n}\Cov [a_{i_1, \pi_1(i_1), \pi_2(i_1)}, a_{i_1', \pi_1(i_1'), \pi_2(i_1')} ] . 
			\end{align}
			Assumption \ref{assumption:A} implies that $\E[a_{i, \pi_1(i), \pi_2(i)}] = 0$, hence, 
			\begin{align}\label{eq:varC}
				\sum_{i=1}^{n}  \Var[a_{i, \pi_1(i), \pi_2(i)}]= \sum_{i=1}^{n} \E[a_{i, \pi_1(i), \pi_2(i)}^2]  = \frac{1}{n^{2}} \sum_{1\leq i_1, i_2, i_3 \leq n } a_{i_1, i_2, i_3}.
			\end{align}
			Moreover, 
			\begin{align*}
				\Cov[a_{i_1, \pi_1(i_1), \pi_2(i_1)}, a_{i_1', \pi_1(i_1'), \pi_2(i_1')}] 
				&
				=\E[a_{i_1, \pi_1(i_1), \pi_2(i_1)} a_{i_1' \pi_1(i_1') \pi_2(i'_1)}] \nonumber \\
				&
				=\frac{1}{n^{2}(n-1)^{2}}\sum_{1 \leq i_2 \neq i_2' \leq n} \sum_{1 \leq i_3 \ne i_3' \leq n} a_{i_1, i_2, i_3} a_{i_1', i_2', i_3'} \nonumber \\
				&
				=\frac{1}{n^{2}(n-1)^{2}}\sum_{1 \leq i_2, i_3 \leq n} \left\{a_{i_1, i_2, i_3} \sum_{\substack{1 \leq i_2' \leq n \\ i_2' \ne i_2}} \sum_{\substack{1 \leq i_3' \leq n \\ i_3' \ne i_3}} a_{i_1', i_2', i_3'}  \right\} \nonumber \\ 
				&
				=\frac{1}{n^{2}(n-1)^{2}}\sum_{1 \leq i_2, i_3 \leq n} \left\{a_{i_1, i_2, i_3} \sum_{\substack{1 \leq i_2' \leq n \\ i_2' \ne i_2}} (-a_{i_1', i_2', i_3})  \right\} \nonumber \\ 
				&
				=\frac{1}{n^{2}(n-1)^{2}} \sum_{1 \leq i_2, i_3 \leq n} \left\{a_{i_1, i_2, i_3} a_{i_1', i_2, i_3} \right \}. 
			\end{align*}
			Hence, 
			\begin{align}\label{eq:covarianceC}
				\sum_{1 \leq i_1 \neq i_1' \leq n}\Cov [a_{i_1, \pi_1(i_1), \pi_2(i_1)}, a_{i_1', \pi_1(i_1'), \pi_2(i_1')} ]    &
				=\frac{1}{n^{2}(n-1)^{2}}\sum_{1 \leq i_1 \neq i_1' \leq n} \sum_{1 \leq i_2, i_3 \leq n} \left\{a_{i_1, i_2, i_3} a_{i_1', i_2, i_3}  \right\} \nonumber \\
				&
				=\frac{1}{n^{2}(n-1)^{2}}\sum_{1 \leq i_1, i_2, i_3 \leq n} a_{i_1, i_2, i_3} \sum_{\substack{ 1\leq i_1' \leq n \\ i_1' \ne i_1}} a_{i_1', i_2, i_3} \nonumber \\
				&
				=-\frac{1}{n^{2}(n-1)^{2}}\sum_{1 \leq i_1, i_2, i_3 \leq n} a_{i_1, i_2, i_3}^{2} .
			\end{align}
			Therefore, combining \eqref{eq:varC} and \eqref{eq:covarianceC} with \eqref{eq:varianceC} the lemma follows. 
		\end{proof}

		Since $\Var[C_n]= \Theta(1)$ by Assumption \ref{assumption:A}, we can always normalize the elements of $\bm A_n$ such that $\Var[C_n]=1$. Hereafter, we will assume $\bm A_n$ is normalized such that $\Var[C_n]=1$. To prove the CLT in Theorem \ref{thm:clt} will use Stein's method based on exchangeable pairs \cite{chatterjee}. The first step towards this is to construct an exchangeable  pair for the random permutations $(\pi_1, \pi_2)$. This is done as follows: 
		
		\begin{itemize} 
			
			\item We define $\pi_1'=\pi_1\circ (I,J)$. Choose a transposition $(I,J)$ uniformly at random from the set of transpositions on $\{1,\ldots,n\}$ and define $\pi_1' :=\pi_1\circ (I,J)$, that is, $\pi_1'(I)=\pi_1(J)$, $\pi_1'(J)=\pi_1(I)$,  
			and $\pi_1'(\ell)=\pi_1(\ell)$ for $\ell \neq \{I,J\}$. 
			
			\item Similarly, choose an independent transposition $(K, L)$ uniformly at random from the set of transpositions on $\{1,\ldots,n\}$ and define $\pi_2' :=\pi_2\circ (K, L)$. 
			
		\end{itemize}
		
		
		\begin{lemma}\label{lm:pi12}
			$\{(\pi_1, \pi_2),(\pi_1',\pi_2')\}$ is an exchangeable pair.
		\end{lemma}
		
		\begin{proof} Denote by $S_n$ the collection of all permutations of $\{1, 2, \ldots, n\}$ and suppose $A_1, A_2, A_1' A_2' \subseteq S_n$. Then by the independence between $(\pi_1,\pi_1')$ and $(\pi_2,\pi_2')$, 
			\begin{align*}
				\P( (\pi_1, \pi_2) \in A_1 \times A_2, (\pi_1', \pi_2') \in A_1' \times A_2' ) & =\P(\pi_1\in A_1, \pi_1'\in A_1' )\P(\pi_2\in A_2, \pi_2'\in A_2' ) \nonumber \\ 
				& =\P(\pi_1' \in A_1, \pi_1 \in A_1' )\P(\pi_2' \in A_2, \pi_2 \in A_2' ) \nonumber \\     
				& = \P( (\pi_1', \pi_2') \in A_1 \times A_2, (\pi_1, \pi_2) \in A_1' \times A_2') ,
			\end{align*}
			where the second inequality uses the fact that $(\pi_1,\pi_1')$ and $(\pi_2,\pi_2')$ are marginally two exchangeable pairs. 
		\end{proof}
		
		Recall that $C_n=\sum_{i=1}^{n}a_{i\pi_1(i)\pi_2(i)}$ and define 
		\begin{align}\label{eq:Cnexchangeable}
			C_n'=\sum_{i=1}^{n}a_{i\pi_1'(i)\pi_2'(i)}. 
		\end{align}
		By Lemma \ref{lm:pi12}, $(C_n, C_n')$ is an exchangeable pair. Hence, by Stein's method based on exchangeable pairs we can obtain bounds on the Wasserstein distance between $C_n$ and $N(0, 1)$ in terms of the moments of the difference $C_n'-C_n$. To this end, recall that the Wasserstein distance between random variables $X \sim \mu$ and $Y \sim \nu$ on $\mathbb R$ is defined as 
		$$\mathrm{Wass}(X, Y) = \sup \left\{\left|\int f \mathrm d\mu - \int f \mathrm d \nu \right| : f \text{ is } 1-\text{Lipschitz}\right\},$$
		where a function $f: \mathbb R \rightarrow \mathbb R$ is 1-Lipschitz if $|f(x) - f(y)| \leq |x-y|$, for all $x, y \in \mathbb R$. We now invoke the following theorem: 
		
		\begin{theorem}\cite{chatterjee} \label{thm:exchangeablepair}
			Let $(C_n,C_n')$ be as defined above. If $\E[C_n'-C_n|C_n]=-\lambda C_n$, for some $0<\lambda<1$, and $\E[C_n^{2}]=1$, then 
			\begin{align*}
				\mathrm{Wass}(C_n, N(0, 1))\leq\sqrt{\frac{2}{\pi}\Var\left[\E\left[\frac{(C_n'-C_n)^2}{2\lambda}\Big|C_n \right] \right]}+\frac{1}{3\lambda}\E\left[|C_n'-C_n|^{3}\right]. 
			\end{align*}
		\end{theorem}

		The rest of the proof is organized as follows: 
		
		\begin{itemize}
			
			\item[(1)] In Lemma \ref{lm:expectationC} we evaluate $\E[C_n'-C_n|C_n]$ and show that  $\E[C_n'-C_n|C_n]=-\lambda C_n$ holds with $\lambda = \Theta(1/n)$.

			\item[(2)] Next, in Lemma \ref{lm:thirdmoment} we show $\frac{1}{\lambda}\E[|C_n'-C_n|^{3}] = O(1/\sqrt n )$.

			\item[(3)] Finally, in Lemma \ref{lm:Cexpvar} we show that $\Var[\E[\frac{(C_n'-C_n)^2}{2\lambda}|C_n ]] = O(1/n)$. 
			
		\end{itemize} 
		These three steps combined with Theorem \ref{thm:exchangeablepair} implies, $\mathrm{Wass}(C_n, N(0, 1)) = O(1/\sqrt n)$, thus, completing the proof of Theorem \ref{thm:clt}.

		\begin{lemma} \label{lm:expectationC}
			$\E[C_n'-C_n|C_n]=-\lambda C_n$, where 
			\begin{align*}
				\lambda= \frac{4(n^2-7n+11)}{n^2 (n-1)} = \Theta\left( \frac{1}{n} \right) . 
			\end{align*}
		\end{lemma}
		
		\begin{proof} 
			Note that 
			\begin{align}\label{eq:Cn}
				C_n'-C_n =
				\begin{cases}
					\Delta_1(I, J) - \Delta_0(I, J)& \text{if }K=I,L=J \text{ or }K=J, L=I , \\
					\Delta_2(I, J, K) - \Delta_0(I, J, K)& \text{if }K\neq J,L=I , \\
					\Delta_2(I, J, L) - \Delta_0(I, J, L)& \text{if }K=I, L\neq J , \\
					\Delta_3(I, J, K) -\Delta_0(I, J, K) & \text{if }K\neq I,L=J , \\
					\Delta_3(I, J, L) -\Delta_0(I, J, L) & \text{if }K=J, L\neq I , \\
					\Delta_4(I, J, K, L) - \Delta_0(I, J, K, L) & \text{otherwise;}
				\end{cases}
			\end{align}
			where 
			\begin{align*}
				\Delta_0(I, J) & := a_{I, \pi_1(I), \pi_2(I)} + a_{J, \pi_1(J), \pi_2(J)} , \\ 
				\Delta_0(I, J, K) & := a_{I, \pi_1(I), \pi_2(I)} + a_{J, \pi_1(J), \pi_2(J)} + a_{K, \pi_1(K), \pi_2(K)} , \\ 
				\Delta_0(I, J, K, L) & := a_{I, \pi_1(I), \pi_2(I)} + a_{J, \pi_1(J), \pi_2(J)} + a_{K, \pi_1(K), \pi_2(K)} + a_{L, \pi_1(L), \pi_2(L)} ;
			\end{align*} 
			and 
			\begin{align*}
				\Delta_1(I, J) & := a_{I, \pi_1(J), \pi_2(J)}+a_{J, \pi_1(I), \pi_2(I)}  \\
				\Delta_2(I, J, \cdot) & :=  a_{I, \pi_1(J), \pi_2(\cdot)}+a_{J, \pi_1(I),\pi_2(J)}+a_{\cdot, \pi_1(\cdot), \pi_2(I)} \\
				\Delta_3(I, J, \cdot) & := a_{I, \pi_1(J), \pi_2(I)}+a_{J, \pi_1(I), \pi_2(\cdot)}+a_{\cdot, \pi_1(\cdot), \pi_2(J)} \\ 
				\Delta_4(I, J, K, L)  & := a_{I, \pi_1(J), \pi_2(I)}+a_{J, \pi_1(I), \pi_2(J)}+a_{K, \pi_1(K), \pi_2(L)}+a_{L, \pi_1(L), \pi_2(K)} . 
			\end{align*}  
			Denote $\eta = \frac{1}{n(n-1)}$. Then \eqref{eq:Cn} implies, 
			\begin{align*}
				&
				\E[C_n'-C_n |\pi_1,\pi_2] = \eta^2 \left(  S_1 + S_2 + S_3 + S_4 \right) 
			\end{align*} 
			where    
			\begin{align*} 
				S_1 & := 2 \sum_{1 \leq i\neq j \leq n} \left\{ \Delta_1(i, j) - \Delta_0(i, j) \right\} \\ 
				S_2 & := \sum_{1 \leq i\neq j \ne k \leq n} \left\{ \Delta_2(i, j, k) + \Delta_3(i, j, k) - 2 \Delta_0(i, j, k) \right\} \\ 
				S_3 & := \sum_{1 \leq i\neq j \ne \ell \leq n} \left\{ \Delta_2(i, j, \ell) + \Delta_3(i, j, \ell) - 2 \Delta_0(i, j, \ell) \right\} \\ 
				S_4 & := \sum_{1 \leq i\neq j \ne k \ne \ell \leq n} \left\{ \Delta_4(i, j, k, \ell) - \Delta_0(i, j, k, \ell) \right\} .
			\end{align*}
			Note that, by Assumption \ref{assumption:A},  $\sum_{1 \leq i\neq j \leq n}  a_{i, \pi_1(j), \pi_2(j)} = - \sum_{j=1}^n a_{j, \pi_1(j), \pi_2(j)} = - C_n$. Hence, 
			\begin{align*} 
				S_1 = -4 C_n - 4(n-1) C_n . 
			\end{align*} 
			Computing the other terms similarly and simplifying gives 				\begin{align*}
					\E[C_n'-C_n|C_n]=\E[C_n'-C_n|\pi_1,\pi_2]= - \lambda C_n
				\end{align*}
				where 
				\begin{align*}
					\lambda=\frac{4(n-2)(n-3)}{n(n-1)^{2}}+\frac{4(n-2)(5-3n)}{n^{2}(n-1)^{2}}+\frac{4}{n^{2}(n-1)} = \frac{4(n^2-7n+11)}{n^2 (n-1)} = \Theta\left( \frac{1}{n} \right).
				\end{align*}
				This completes the proof of Lemma \ref{lm:expectationC}.
			\end{proof}

			\begin{lemma}\label{lm:thirdmoment} $\frac{1}{\lambda}\E[|C_n'-C_n|^{3}]=O(1/\sqrt n)$. 
			\end{lemma}
			
			\begin{proof} Denote $\eta = \frac{1}{n(n-1)}$. Then \eqref{eq:Cn} implies, 
				\begin{align*}
					\E[|C_n'-C_n|^{3}|\pi_1,\pi_2] \lesssim \eta^2 (W_1+W_2+W_3 ) , 
				\end{align*}
				where 
				\begin{align*} 
					W_1 & := \sum_{1 \leq i\neq j \leq n} \left( | \Delta_1(i, j)|^3 + |\Delta_0(i, j)|^3 \right) , \\ 
					W_2& := \sum_{1 \leq i\neq j \ne k \leq n} \left( | \Delta_2(i, j, k) |^3 + |\Delta_3(i, j, k)|^3 + |\Delta_0(i, j, k) |^3 \right) , \\ 
					W_3 & := \sum_{1 \leq i\neq j \ne k \ne \ell \leq n} \left( | \Delta_4(i, j, k, \ell) |^3 + | \Delta_0(i, j, k, \ell)|^3 \right) ,
				\end{align*}
				Using the fact $|a_{i_1, i_2, i_3}|\leq C/\sqrt{n}$, it follows that $W_1 + W_2 + W_3 = O(n^{5/2})$. Hence, using $\eta = \Theta(1/n^2)$ and $\lambda=\Theta(1/n)$ (by Lemma \ref{lm:expectationC}), 
				\begin{align*}
					\frac{1}{\lambda}\E[|C_n'-C_n|^{3}] = \frac{1}{\lambda}\E\left[\E\left[|C_n'-C_n|^{3}|\pi_1,\pi_2\right]\right]=O(1/\sqrt n) ,
				\end{align*}
				completing the proof of the lemma. 
			\end{proof}

			\begin{lemma} \label{lm:Cexpvar}
				$\Var[\E[\frac{(C_n'-C_n)^2}{2\lambda}|C_n ]] = O(1/n)$. 
			\end{lemma}

			\begin{proof}  
				Recall that for any sigma-fields $\mathcal{F}_1\subseteq \mathcal{F}_2$, 
				\begin{align*}
					\Var[\E[X|\mathcal{F}_1]] \leq \Var[\E[X|\mathcal{F}_2]] . 
				\end{align*} 
				Hence, to prove the lemma it suffices to bound $\Var[\E[|C_n'-C_n|^2|\pi_1,\pi_2]]$. We first compute 
				\begin{align*}
					\E[|C_n'-C_n|^{2}|\pi_1,\pi_2] = \eta^2 (T_1+T_2+T_3+T_4).
				\end{align*} 
				where $\eta = \frac{1}{n(n-1)}$ and recalling \eqref{eq:Cn}, 
				\begin{align*} 
					T_1 & := 2 \sum_{1 \leq i\neq j \leq n}  | \Delta_1(i, j) - \Delta_0(i, j)|^2  , \\ 
					T_2& := 2 \sum_{1 \leq i\neq j \ne k \leq n} | \Delta_2(i, j, k)  - \Delta_0(i, j, k) |^2 , \\ 
					T_3& := 2 \sum_{1 \leq i\neq j \ne k \leq n} | \Delta_3(i, j, k)  - \Delta_0(i, j, k) |^2 , \\ 
					T_4 & :=  \sum_{1 \leq i\neq j \ne k \ne \ell \leq n}  | \Delta_4(i, j, k, \ell) - \Delta_0(i, j, k, \ell)|^2 . 
				\end{align*}
							Denote $T_{\nu} = T_1+T_2+T_3$. This implies, 
				\begin{align}\label{eq:varT}
					\Var[\E[|C_n'-C_n|^{2}|\pi_1,\pi_2]] \lesssim \eta^4 \Var[T_\nu] + \eta^4 \Var[T_4]   + 2 \eta^4 \sqrt{\Var[T_\nu] \Var[T_4]}.
				\end{align} 
				By Assumption \ref{assumption:A}, $T_{\nu} = O(n^2)$. Hence, 
				\begin{align}\label{eq:Tnu}
					\eta^4 \Var[T_{\nu}] = O(1/n^4). 
				\end{align}
				
				We will now show that $\eta^4 \Var[T_4] = O(1/n^3)$. For this define 
				\begin{align*}
					b_{ijk\ell} :=  | \Delta_4(i, j, k, \ell) - \Delta_0(i, j, k, \ell)| . 
				\end{align*} 
				Then 
				\begin{align}\label{eq:T6}
					\Var[T_4]    = \sum_{1 \leq i\neq j \neq k \neq \ell \leq n}\sum_{1 \leq i'\neq j'\neq k'\neq \ell' \leq n} \left\{ \E [ b_{i j k \ell}^{2} b_{i'j'k'\ell'}^{2} ] -  \E[b_{ijk \ell}^{2}]\E[b_{i'j'k'\ell'}^{2}] \right\} = S_1+S_2 , 
				\end{align}  
				where 
				\begin{align}\label{eq:S2}
					S_1   & := \sum_{\substack{\text{not all different} \\ 
							1 \leq i\neq j \neq k \neq \ell \leq n, 1 \leq i'\neq j'\neq k'\neq \ell' \leq n } }  
					\left\{ \E [ b_{i j k \ell}^{2} b_{i'j'k'\ell'}^{2} ] -  \E[b_{ijk \ell}^{2}]\E[b_{i'j'k'\ell'}^{2}] \right\}  \nonumber \\ 
					S_2 & := \sum_{ 1 \leq i \ne i' \neq j  \ne j' \neq k \ne k' \neq \ell \ne \ell' \leq n } \left\{ \E [ b_{i j k \ell}^{2} b_{i'j'k'\ell'}^{2} ] -  \E[b_{ijk \ell}^{2}]\E[b_{i'j'k'\ell'}^{2}]  \right\}  ,  
				\end{align}
				
				For $S_1$, we bound
				\begin{align}\label{eq:S1_clt}
					\eta^4  S_1\leq \frac{1}{n^{4}(n-1)^{4}}  \sum_{\substack{\text{not all different} \\ 
							1 \leq i\neq j \neq k \neq \ell \leq n, 1 \leq i'\neq j'\neq k'\neq \ell' \leq n } }   \E \left[ b_{i j k \ell}^{2} b_{i'j'k'\ell'}^{2} \right] = O(1/n^3).
				\end{align} 
				since $b_{ijk\ell}=O(1/\sqrt{n})$ by Assumption \ref{assumption:A} and there are $O(n^7)$ terms in the sum in \eqref{eq:S1_clt}.  Next, we bound $S_2$. To this end, fix $r \geq 1$ and denote $(n)_r := n(n-1) \cdots (n-r+1)$ and $[n]_r$ the collection of ordered $r$ element subsets of $\{1, 2, \ldots, n\}$ with distinct elements. For $\bm  s = (s_1, s_2, s_3, s_4) \in [n]_4$  and $\bm s' = (s_1', s_2', s_3', s_4') \in [n]_4$, define 
				$$\Delta(i, j, k, \ell; \bm s, \bm s' ) := \left( a_{i, s_2, s'_1}+a_{j, s_1, s'_2}+a_{k, s_3, s'_4}+a_{\ell, s_4, s'_3}-a_{i, s_1, s'_1}-a_{j, s_2, s'_2}-a_{k, s_3, s'_3}-a_{\ell, s_4, s'_4} \right)^{2} $$
				For $\bm  t = (s_1, s_2, s_3, s_4) \in [n]_4$  and $\bm t' = (s_1', s_2', s_3', s_4')\in [n]_4$, $\Delta(i, j, k, \ell; \bm t, \bm t' )$ is defined similarly. Then 
				\begin{align}\label{eq:covbijkl}
					\E[ b_{i j k \ell}^{2} b_{i'j'k'\ell'}^{2} ] &  = \frac{1}{(n)_8^{2}} \sum_{(\bm s, \bm t) \in [n]_8 } \sum_{(\bm s', \bm t') \in [n]_8 }  \Delta(i, j, k, \ell; \bm s, \bm s' ) \Delta(i, j, k, \ell; \bm t, \bm t' ) 
				\end{align}
				and 
				\begin{align}\label{eq:expbijkl}
					\E[ b_{i j k \ell}^{2}] \E[ b_{i'j'k'\ell'}^{2} ] &  = \frac{1}{(n)_4^{4}} \sum_{\bm s, \bm s', \bm t, \bm t' \in [n]_4 } \Delta(i, j, k, \ell; \bm s, \bm s' ) \Delta(i, j, k, \ell; \bm t, \bm t' )  . 
				\end{align} 
				Then we further decompose
				\begin{align*}
					& \sum_{\bm s, \bm s', \bm t, \bm t' \in [n]_4 } \Delta(i, j, k, \ell; \bm s, \bm s' ) \Delta(i, j, k, \ell; \bm t, \bm t' ) \\ 
					&
					= \sum_{(\bm s, \bm t) \in [n]_8 } \sum_{(\bm s', \bm t') \in [n]_8 }  \Delta(i, j, k, \ell; \bm s, \bm s' ) \Delta(i, j, k, \ell; \bm t, \bm t' ) + \cR , 
				\end{align*}   
				where 
				\begin{align*}
					\cR :=   & \sum_{\substack{ \bm s, \bm s', \bm t, \bm t' \in [n]_4  \\ \bm s \cap \bm t \ne \emptyset \text{ or } \bm s' \cap \bm t' \ne \emptyset}} \Delta(i, j, k, \ell; \bm s, \bm s' ) \Delta(i, j, k, \ell; \bm t, \bm t' ) .
				\end{align*}
				Then, denoting $\sum_{\cD}$ as the sum over indices $1 \leq i \ne i' \neq j  \ne j' \neq k \ne k' \neq \ell \ne \ell' \leq n$ and noting from Assumption \ref{assumption:A} that 
				\begin{align}\label{eq:aijklst}
					\Delta(i, j, k, \ell; \bm s, \bm s' ) \Delta(i, j, k, \ell; \bm t, \bm t' ) = O\left( \frac{1}{n^2} \right),  
				\end{align}
				implies $\eta^4\sum_{\cD} \cR = O(1/n^3)$. Recalling \eqref{eq:S2}, \eqref{eq:covbijkl}, and \eqref{eq:expbijkl} then gives, 
				\begin{align*} 
					S_2  \leq  \sum_{ \cD } \left( \frac{1}{(n)_8^2}-\frac{1}{(n)_4^4} \right)  \sum_{(\bm s, \bm t) \in [n]_8 } \sum_{(\bm s', \bm t') \in [n]_8 }   
					\Delta(i, j, k, \ell; \bm s, \bm s' ) \Delta(i, j, k, \ell; \bm t, \bm t' ) + \sum_{\cD} \cR. 
				\end{align*}
				Note that  $\frac{1}{(n)_8^2}-\frac{1}{(n)_4^4} = O(1/n^{17})$. Hence, by \eqref{eq:aijklst}, $$\eta^4 S_2 = O(1/n^3).$$ Combining this with \eqref{eq:T6} and \eqref{eq:S1} it follows that $\eta^4 \Var[T_4] = O(1/n^3)$. This together with \eqref{eq:varT} and \eqref{eq:Tnu} implies, 
				\begin{align*}  
					\Var\left[\E\left[|C_n'-C_n|^{2}|\pi_1,\pi_2\right] \right]=O(1/n^{3}). 
				\end{align*}
				Since $\lambda= \Theta(1/n)$, by Lemma \ref{lm:expectationC}, the result in Lemma \ref{lm:Cexpvar} follows. 
			\end{proof}

			\section{Proof of Theorem \ref{thm:consistency} and Proposition \ref{ppn:phiconsistency}}
			\label{sec:consistencypf}
			
			This section is organized as follows: We prove the consistency of the estimate $\RJdCov_n^2(\bm X_S)$ (recall definition from \eqref{eq:RdCovSn}) in Section \ref{sec:consistencyestimatepf}. The consistency of the corresponding test \eqref{eq:testrank} is proved in Section \ref{sec:consistencytestpf}. 
			
			\subsection{Proof of Theorem \ref{thm:consistency}} 
			\label{sec:consistencyestimatepf}
			
			Recall from \eqref{eq:RdCovWXS} and \eqref{eq:estimateW} that 
			\begin{align}\label{eq:RdCovWXSpf}
				\RdCov^2_n(\bm X_S) &=\frac{1}{n^2}\sum_{1 \leq a, b \leq n} \prod_{i \in S} \hat{\cE}_i(a, b) ,
			\end{align} 
			where 
			\begin{align*}
				\hat{\cE}_i(a, b) & := \frac{1}{n}\sum_{v=1}^n\|\hat{R}_i(X_{i}^{(a)})-\hat{R}_i(X_{i}^{(v)})\|+\frac{1}{n}\sum_{u=1}^n\|\hat{R}_i(X_{i}^{(u)})-\hat{R}_i(X_{i}^{(b)})\| \nonumber \\ 
				& \hspace{1.25in} -\|\hat{R}_i(X_{i}^{(a)})-\hat{R}_i(X_{i}^{(b)})\| -\frac{1}{n^2}\sum_{1 \leq u,v \leq n}\|\hat{R}_i(X_{i}^{(u)})-\hat{R}_i(X_{i}^{(v)})\| . 
			\end{align*} 
			Now, expanding the product gives, 
			\begin{align}\label{eq:bivariateab}
				\frac{1}{n^{2}}\sum_{1 \leq a, b \leq n}\prod_{i \in S} \hat{\cE}_i(a, b)=\frac{1}{n^{2}}\sum_{(\ell_i)_{i \in S}}\sum_{1 \leq a, b \leq n} \prod_{i \in S} \Delta^{(\ell_i)}_i(a, b) , 
			\end{align}
			where $\ell_i \in\{1,2,3,4\}$ for $i\in\{1,\ldots,d\}$ and 
			\begin{align*}
				\Delta_i^{(1)}(a, b)& = \frac{1}{n}\sum_{v=1}^{n}\|\R_i(X_{i}^{(a)})-\R_{i}(X_{i}^{(v)})\|, \\
				\Delta_i^{(2)}(a, b)& = \frac{1}{n}\sum_{u=1}^{n}\|\R_i(X_{i}^{(b)})-\R_{i}(X_{i}^{(u)})\|,\\
				\Delta_i^{(3)} (a, b) & = - \|\R_i(X_{i}^{(a)})-\R_{i}(X_{i}^{(b)})\|, \\ 
				\Delta_i^{(4)}(a, b) & = - \frac{1}{n^2} \sum_{1 \leq u, v \leq n}\|\R_{i}(X_{i}^{(u)})-\R_{i}(X_{i}^{(v)})\|.
			\end{align*}
			(Note that $\Delta_i^{(1)}(a, b)$, $\Delta_i^{(2)}(a, b)$, and $\Delta_i^{(4)}(a, b)$ does not depend on $b$, $a$, and $a, b$, respectively. Nevertheless, for reasons of symmetry keep the dependence on $a, b$ in the notations.) Similarly, define 
			\begin{align*}
				\tilde \Delta_i^{(1)}(a, b)& = \frac{1}{n}\sum_{v=1}^{n}\| R_{\mu_i}(X_{i}^{(a)})-R_{\mu_i}(X_{i}^{(v)})\|, \\
				\tilde \Delta_i^{(2)}(a, b)& = \frac{1}{n}\sum_{u=1}^{n}\| R_{\mu_i}(X_{i}^{(b)})-R_{\mu_i}(X_{i}^{(u)})\|,\\
				\tilde \Delta_i^{(3)} (a, b) & = - \|R_{\mu_i}(X_{i}^{(a)})-R_{\mu_i}(X_{i}^{(b)})\|, \\ 
				\tilde \Delta_i^{(4)}(a, b) & = - \frac{1}{n^2} \sum_{1 \leq u, v \leq n}\|R_{\mu_i}(X_{i}^{(u)})-R_{\mu_i}(X_{i}^{(v)})\|. 
			\end{align*} 
			As for $\Delta_i^{(1)}(a,b),\tilde{\Delta}_i^{(1)}(a,b)$, we have
			\begin{align*}
				\left|\Delta_{i}^{(1)}(a, b)-\t\Delta_{i}^{(1)}(a, b) \right|
				&
				\leq\frac{1}{n}\sum_{v=1}^{n} \left| \|\R_i(X_{i}^{(a)})-\R_{i}(X_{i}^{(v)})\|-\|R_{\mu_i}(X_{i}^{(a)})-R_{\mu_i}(X_{i}^{(v)})\| \right|\\
				&
				\leq \frac{1}{n}\sum_{v=1}^{n}\|\R_i(X_{i}^{(v)})-R_{\mu_i}(X_{i}^{(v)})\|+\|\R_i(X_i^{(a)})-R_{\mu_i}(X_i^{(a)})\|, 
			\end{align*}
			where the second inequality holds by the reserve triangle inequality. Similarly, 
			\begin{align*}
				\left |\Delta_{i}^{(2)}(a, b)-\t\Delta_{i}^{(2)}(a, b) \right | & \leq  \frac{1}{n}\sum_{u=1}^{n}\|\R_i(X_{i}^{(u)})-R_{\mu_i}(X_{i}^{(u)})\|+\|\R_i(X_i^{(b)})-R_{\mu_i}(X_i^{(b)})\| 
			\end{align*}
			and 	
			\begin{align*}
				\big|\Delta_{i}^{(3)}(a, b)-\t\Delta_{i}^{(3)}(a, b)\big|	&
				\leq \|\R_i(X_{i}^{(b)})-R_{\mu_i}(X_{i}^{(b)})\|+\|\R_i(X_{i}^{(a)})-R_{\mu_i}(X_{i}^{(a)})\| .
			\end{align*}
			The last term we bound
			\begin{align*}
				\left |\Delta_{i}^{(4)}(a, b)-\t\Delta_{i}^{(4)}(a, b) \right |\leq \frac{1}{n}\sum_{u=1}^{n}\|\R_i(X_{i}^{(u)})-R_{\mu_i}(X_{i}^{(u)})\|+\frac{1}{n}\sum_{v=1}^{n}\|\R_i(X_{i}^{(v)})-R_{\mu_i}(X_{i}^{(v)})\|\\
			\end{align*}
			Hence, by \cite[Theorem 2.1]{deb2021multivariate}, for $1 \leq i \leq r$ and $1 \leq s \leq 4$, 
			\begin{align}\label{eq:Deltaab}
				\frac{1}{n^{2}}\sum_{1 \leq a, b \leq n}\big|\Delta_{i}^{(s)}(a, b)-\t\Delta_{i}^{(s)}(a, b)\big| \stackrel{a.s.}{\rightarrow} 0 . 
			\end{align}
			
			Note that $|\t\Delta_i^{(\ell_i)}(a, b)|\leq\sqrt{d_i}$ and $|\Delta_i^{(\ell_i)}(a, b)| \leq\sqrt{d_i}$, for all $1 \leq a \ne b \leq n$ and $1 \leq i \leq r$. Then for any $T \subseteq S$ and $i \notin T$, by triangle inequality, we have
			\begin{align*}
				&
				\frac{1}{n^{2}}\sum_{1 \leq a, b \leq n} \left| \prod_{j \in T} \t\Delta_j^{\ell_j}(a, b) \prod_{j \in S\backslash T} \Delta_j^{(\ell_j)}(a, b) 
				- \prod_{j \in T \bigcup \{i\}} \t\Delta_j^{(\ell_j)}(a, b) \prod_{j \in S\backslash (T \bigcup \{i\})} \Delta_j^{\ell_j}(a, b)  \right| \\
				&
				\leq\frac{1}{n^{2}}\sum_{1 \leq a, b \leq n} \left| \Delta_{i}^{(\ell_i)}(a, b)-\t\Delta_{i}^{(\ell_i)}(a, b) \right| \prod_{j \in S\backslash\{i\}} \sqrt{d_j}  \stackrel{a.s.} \rightarrow 0 , 
			\end{align*} 
			by \eqref{eq:Deltaab}. This implies, by a telescoping argument together with \eqref{eq:bivariateab}, 		
			\begin{align}\label{eq:RdCovab}
				\left|\RdCov^{2}_n(\bm X_S) - \frac{1}{n^{2}}\sum_{1 \leq a \ne b \leq n}\prod_{i \in S }{\cE}_i^\mathrm{oracle}(a, b) \right| \stackrel{a.s.}{\rightarrow} 0 ,
			\end{align}
			where 
			\begin{align*}
				{\cE}_i^\mathrm{oracle}(a, b) 
				&
				=\frac{1}{n}\sum_{v=1}^{n}\|R_{\mu_i}(X_{i}^{(a)})-R_{\mu_i}(X_{i}^{(v)})\|+\frac{1}{n}\sum_{u=1}^{n}\|R_{\mu_i}(X_{i}^{(b)})-R_{\mu_i}(X_{i}^{(u)})\|\\
				&
				\hspace{1.25in} -\|R_{\mu_i}(X_{i}^{(a)})-R_{\mu_i}(X_{i}^{(b)})\|- \frac{1}{n^2} \sum_{1 \leq u, v \leq n}\|R_{\mu_i}(X_{i}^{(u)})-R_{\mu_i}(X_{i}^{(v)})\|.
			\end{align*}
			Since $\{ (R_{\mu_1}(X_{1}^{(a)}), R_{\mu_2}(X_{2}^{(a)}), \ldots, R_{\mu_d}(X_{d}^{(a)}) \}_{1 \leq a \leq n}$ are i.i.d., applying \cite[Proposition 8]{chakraborty2019distance}  to the rank transformed data shows that 
			\begin{align*}
				\frac{1}{n^{2}}\sum_{1 \leq a, b \leq n}\prod_{i \in S} {\cE}_i^\mathrm{oracle}(a, b)  \stackrel{a.s.}{\rightarrow}\RdCov^{2}(\bm X_S) .
			\end{align*}
			This combined with \eqref{eq:RdCovab} completes the proof of Theorem  \ref{thm:consistency}. 
			
			\subsection{Proof of Proposition \ref{ppn:phiconsistency}} 
			\label{sec:consistencytestpf}

			Recall  that $c_{\alpha, n}$ is the upper $\alpha$ quantile of the universal distribution in Proposition \ref{ppn:H0distributionfree}.  Note from \eqref{eq:H0RJdCov} in Theorem \ref{thm:H0}, $c_{\alpha, n} = O(1/n)$. Moreover, $\RJdCov^{2}_n(\bm X; \bm C) \stackrel{a.s.} \rightarrow \RJdCov^{2}(\bm X; \bm C) > 0$ whenever $\mu \ne \mu_1\otimes\mu_2\otimes\cdots\otimes\mu_d$ by Theorem \ref{thm:consistency} and Proposition \ref{ppn:RJdCovindependence}. Hence, recalling \eqref{eq:testrank}, for $\mu \ne \mu_1\otimes\mu_2\otimes\cdots\otimes\mu_d$,  
			\begin{align*}
				\E_{\mu}\left[\phi_n(\bm C) \right] = \P_{\mu}(\RJdCov^{2}_n(\bm X; \bm C) > c_{\alpha, n}) \rightarrow 1.  
			\end{align*}

			\section{Proof of Theorem \ref{thm:H0}}
			\label{sec:H0pf}

				Recall from \eqref{eq:ZnS} the definition of the empirical process $Z_{S, n}(\bm t)$, for $S \in \cT$, $\bm t= (t_i)_{i \in S} \in \mathbb R^{d_S}$ and $t_i \in \mathbb R^{d_i}$, for $i \in S$: 
				\begin{align}\label{eq:ZnSpf}
					Z_{S, n}(\bm t)
					& =  \frac{1}{n}\sum_{b=1}^{n}\left\{\prod_{i \in S} \left (\frac{1}{n}\sum_{a=1}^{n}e^{\i\langle t_i,\R_i(X_{i}^{(a)})\rangle}-e^{\i\langle t_i,\R_i(X_{i}^{(b)})\rangle} \right) \right\} . 
				\end{align} 
    				    				
				Under the assumption of mutual independence, $\E_{H_0}[  Z_{S, n}(\bm t) ]=0$. Next, we compute the covariance between $Z_{S, n}(\bm t)$ and $Z_{S, n}(\bm t')$.

				\begin{lemma}\label{lm:Zcovariance} 
					Suppose $S \in \cT$ and $\bm t= (t_i)_{i \in S}, \bm t'= (t_i')_{i \in S} \in \mathbb R^{d_S}$. Then under $H_0$, 
					\begin{align}\label{eq:CSS}
						\lim_{n \rightarrow \infty} n \Cov[  Z_{S, n}(\bm t), Z_{S, n}(\bm t')] = C_S(\bm t, \bm t') , 
					\end{align}
					where $C_S(\bm t, \bm t')$ as defined in \eqref{eq:CS}. Moreover, for any $S_1 \ne S_2 \in \cT$ and $\bm t= (t_i)_{i \in S_1} \in \mathbb R^{d_{S_1}}$, $\bm t'= (t_i')_{i \in S_2} \in \mathbb R^{d_{S_2}}$, under $H_0$, 
					\begin{align}\label{eq:CS1S2}
						\Cov[  Z_{S_1, n}(\bm t), Z_{S_2, n}(\bm t')] = 0 . 
					\end{align}
				\end{lemma}

				\begin{proof}
					Denote, for $1 \leq b \leq n$ and $i \in S$,  
					\begin{align}\label{eq:ZSnti}
					Z_{S, n}(\bm t, b, i):=  \frac{1}{n}\sum_{a=1}^{n}e^{\i\langle t_i,\R_i(X_{i}^{(a)})\rangle}-e^{\i\langle t_i,\R_i(X_{i}^{(b)})\rangle} 
					\end{align} 
					and its complex conjugate 
					$$\bar Z_{S, n}(\bm t, b, i):=  \frac{1}{n}\sum_{a=1}^{n}e^{-\i\langle t_i,\R_i(X_{i}^{(a)})\rangle}-e^{-\i\langle t_i,\R_i(X_{i}^{(b)})\rangle}.$$
					Then for  $\bm t= (t_i)_{i \in S}, \bm t= (t_i')_{i \in S} \in \mathbb R^{d_S}$, 
					\begin{align}\label{eq:ZS1S2} 
						n  Z_{S, n}(\bm t) \bar Z_{S, n}(\bm t')  & = \frac{1}{n}\sum_{1 \leq b, b' \leq n} \prod_{i \in S} Z_{S, n}(\bm t, b, i) \bar Z_{S, n}(\bm t', b', i) \nonumber \\ 
				& = T_1 + T_2 , 
					\end{align} 
					where 
						\begin{align*}
							T_1 & :=  \frac{1}{n}\sum_{b=1}^n   \prod_{i \in S}  Z_{S, n}(\bm t, b, i) \bar Z_{S, n}(\bm t', b, i) \\
							T_2 & :=  \frac{1}{n}\sum_{1 \leq b \ne b' \leq n}   \prod_{i \in S}  Z_{S, n}(\bm t, b, i) \bar Z_{S, n}(\bm t', b', i) .
						\end{align*} 
						Note that under $H_0$, $\{\R_i(X_{i}^{(1)}),\ldots, \R_i(X_{i}^{(n)}) \}$ are uniformly distributed over the fixed grid $\cH_n^{d_i}=\{h_1^{d_i}, h_2^{d_i}, \ldots, h_n^{d_i} \}$ and independent over $1 \leq i \leq r$. Therefore, 
						\begin{align*}
							& \E [ Z_{S, n}(\bm t, b, i) \bar Z_{S, n}(\bm t', b, i) ] \\ 
							& = \E \left[
							\left(  \frac{1}{n}\sum_{a=1}^{n}e^{\i\langle t_i,\R_i(X_{i}^{(a)})\rangle}-e^{\i\langle t_i,\R_i(X_{i}^{(b)})\rangle} \right) \left(  \frac{1}{n}\sum_{a=1}^{n}e^{\i\langle t_i',\R_i(X_{i}^{(a)})\rangle}-e^{\i\langle t_i',\R_i(X_{i}^{(b)})\rangle} \right) \right] \\ 
							&= \E \left[
							\left(  \frac{1}{n}\sum_{a=1}^{n}e^{\i\langle t_i, h_a^{d_i} \rangle}-e^{\i\langle t_i,\R_i(X_{i}^{(b)})\rangle} \right) \left(  \frac{1}{n}\sum_{a=1}^{n}e^{-\i\langle t_i', h_a^{d_i} \rangle}-e^{-\i\langle t_i',\R_i(X_{i}^{(b)})\rangle} \right) \right] \\ 
							& = \frac{1}{n} \sum_{u=1}^n \left(  \frac{1}{n}\sum_{a=1}^{n}e^{\i\langle t_i, h_a^{d_i} \rangle}-e^{\i\langle t_i, h_{u}^{d_i} \rangle} \right) \left(  \frac{1}{n}\sum_{a=1}^{n}e^{-\i\langle t_i', h_a^{d_i} \rangle}-e^{-\i\langle t_i', h_{u}^{d_i} \rangle} \right) \\ 
							& = \frac{1}{n}  \sum_{u=1}^n e^{\i\langle t_i - t_i', h_u^{d_i} \rangle} -  \left(  \frac{1}{n}\sum_{a=1}^{n}e^{\i\langle t_i, h_a^{d_i} \rangle} \right) \left(  \frac{1}{n}\sum_{a=1}^{n}e^{-\i\langle t_i', h_a^{d_i} \rangle} \right) . 
						\end{align*} 
						This implies, since the collection $\{ Z_{S, n}(\bm t, b, i) Z_{S, n}(\bm t', b', i)\}_{i \in S}$ is independent, 
							\begin{align}\label{eq:S1}
								\E[T_1] & :=  \frac{1}{n}\sum_{b=1}^n   \prod_{i \in S} \E [ Z_{S, n}(\bm t, b, i) Z_{S, n}(\bm t', b', i) ] \nonumber \\ 
								& = \prod_{i \in S} \left(  \frac{1}{n} \sum_{u=1}^n  e^{\i\langle t_i - t_i', h_u^{d_i} \rangle} -  \left(  \frac{1}{n}\sum_{a=1}^{n}e^{\i\langle t_i, h_a^{d_i} \rangle} \right) \left(  \frac{1}{n}\sum_{a=1}^{n}e^{-\i\langle t_i', h_a^{d_i} \rangle} \right)  \right) \nonumber \\
								& \rightarrow C_S(\bm t, \bm t') , 
							\end{align}
							by Assumption \ref{assumption:U}. Similarly, 
							\begin{align*}
								& \E [ Z_{S, n}(\bm t, b, i) \bar Z_{S, n}(\bm t', b', i) ] \\ 
								&= \E \left[
								\left(  \frac{1}{n}\sum_{a=1}^{n}e^{\i\langle t_i, h_a^{d_i} \rangle}-e^{\i\langle t_i,\R_i(X_{i}^{(b)})\rangle} \right) \left(  \frac{1}{n}\sum_{a=1}^{n}e^{-\i\langle t_i', h_a^{d_i} \rangle}-e^{-\i\langle t_i',\R_i(X_{i}^{(b')})\rangle} \right) \right] \\ 
								& = \frac{1}{n(n-1)} \sum_{1 \leq u \ne v \leq n} \left(  \frac{1}{n}\sum_{a=1}^{n}e^{\i\langle t_i, h_a^{d_i} \rangle}-e^{\i\langle t_i, h_{u}^{d_i} \rangle} \right) \left(  \frac{1}{n}\sum_{a=1}^{n}e^{-\i\langle t_i', h_a^{d_i} \rangle}-e^{-\i\langle t_i', h_{v}^{d_i} \rangle} \right) \\ 
								& = \frac{1}{n(n-1)}  \sum_{1 \leq u \ne v \leq n} e^{\i\langle t_i, h_u^{d_i} \rangle - \i\langle t_i', h_v^{d_i} \rangle} -  \left(  \frac{1}{n}\sum_{a=1}^{n}e^{\i\langle t_i, h_a^{d_i} \rangle} \right) \left(  \frac{1}{n}\sum_{a=1}^{n}e^{-\i\langle t_i', h_a^{d_i} \rangle} \right) \\ 
								& = - \frac{1}{(n-1)} \left[ \frac{1}{n} \sum_{u=1}^n e^{\i\langle t_i-t_i', h_u^{d_i} \rangle } -  \left(  \frac{1}{n}\sum_{a=1}^{n}e^{\i\langle t_i, h_a^{d_i} \rangle} \right) \left(  \frac{1}{n}\sum_{a=1}^{n}e^{-\i\langle t_i', h_a^{d_i} \rangle} \right) \right] .
							\end{align*} 
							This implies, 
							\begin{align}\label{eq:T2}
								\E[T_2] & :=  \frac{(-1)^{|S|}}{(n-1)^{|S|-1}} \prod_{i \in S} \left[ \frac{1}{n} \sum_{u=1}^n e^{\i\langle t_i-t_i', h_u^{d_i} \rangle } -  \left(  \frac{1}{n}\sum_{a=1}^{n}e^{\i\langle t_i, h_a^{d_i} \rangle} \right) \left(  \frac{1}{n}\sum_{a=1}^{n}e^{-\i\langle t_i', h_a^{d_i} \rangle} \right) \right] \nonumber \\ 
								& \rightarrow 0, 
							\end{align}
							by Assumption \ref{assumption:U}. Combining \eqref{eq:S1} and \eqref{eq:T2} with \eqref{eq:ZS1S2} the result in \eqref{eq:CSS} follows. Next, suppose $S_1 \ne S_2 \in \cT$ and $\bm t= (t_i)_{i \in S_1}, \bm t'= (t_i')_{i \in S_2}$. Then 
							\begin{align*} 
								& Z_{S_1, n}(\bm t) \bar Z_{S_2, n}(\bm t')  \\ 
								& = \frac{1}{n^2}\sum_{1 \leq b, b' \leq n} \prod_{i \in S_1 \bigcap S_2} Z_{S_1, n}(\bm t, b, i) \bar Z_{S_2, n}(\bm t', b', i) \prod_{i \in S_1 \backslash S_2} Z_{S_1, n}(\bm t, b, i) \prod_{i \in S_2 \backslash S_1}  \bar Z_{S_2, n}(\bm t', b', i) . 
							\end{align*} 
							Note that the collections $\{ Z_{S_1, n}(\bm t, b, i) \}_{i \in S_1 \backslash S_2}$ and $\{Z_{S_2, n}(\bm t', b', i)\}_{i \in S_2 \backslash S_1}$ are independent. Hence,  
							\begin{align} 
								& \E[ Z_{S_1, n}(\bm t) \bar Z_{S_2, n}(\bm t') ] \nonumber \\ 
								& = \frac{1}{n^2}\sum_{1 \leq b, b' \leq n} \prod_{i \in S_1 \bigcap S_2} \E[Z_{S_1, n}(\bm t, b, i) \bar Z_{S_2, n}(\bm t', b', i)] \prod_{i \in S_1 \backslash S_2} \E [Z_{S_1, n}(\bm t, b, i)] \prod_{i \in S_2 \backslash S_1}  \E [\bar Z_{S_2, n}(\bm t', b', i)] \nonumber \\ 
								& = 0 , \nonumber 
							\end{align} 
							since $\E [Z_{S, n}(\bm t, b, i)] = \E [ \bar Z_{S, n}(\bm t, b, i)] = 0$ for $i \in S$ under $H_0$ and at least one of the sets $S_1 \backslash S_2$ and $S_2\backslash S_1$ is non-empty. This proves \eqref{eq:CS1S2}.
						\end{proof} 
						
						We use the following result to prove for Theorem \ref{thm:H0} which is restatement of \cite[Theorem 8.6.2]{resnick2013probability}. 
						To this end, denote $D_{S}(\delta)=\prod_{i\in S}T_i(\delta)$ with $T_i(\delta)=\{t_i\in\mathbb{R}^{d_i}:\delta\leq \|t_i\|\leq 1/\delta\}$. 
						
						\begin{lemma}\label{lem:convergence_cond}
							Suppose the following conditions hold: 
							\begin{enumerate}
								\item[$(1)$] As $n \rightarrow \infty$, $$\left\{n\int_{D_{S}(\delta)}|Z_{S,n}(\bm t)|^2\prod_{i\in S}\mathrm{d}w_i\right\}_{S\in\mathcal{T}}\stackrel{D}{\rightarrow}\left\{\int_{D_S(\delta)}|Z_{S}(\bm t)|^2\prod_{i\in S}\mathrm{d}w_i\right\}_{S\in\mathcal{T}};$$ 
								
								\item[$(2)$] As $\delta\rightarrow 0$, 
								$$\left\{\int_{D_S(\delta)}|Z_{S}(\bm t)|^2\prod_{i\in S}\mathrm{d}w_i\right\}_{S\in\mathcal{T}}\stackrel{D}{\rightarrow}\left\{\int_{\mathbb{R}^{\sum_{i\in S}d_i}}|Z_{S}(\bm t)|^2\prod_{i\in S}\mathrm{d}w_i\right\}_{S\in\mathcal{T}} ;$$
								\item[$(3)$] As $n \rightarrow \infty$ followed by $\delta \rightarrow 0$,  
								$$\E\left|\left\{n\int_{D_{S}(\delta)}|Z_{S,n}(\bm t)|^2\prod_{i\in S}\mathrm{d}w_i\right\}_{S\in\mathcal{T}}-\left\{n\int_{\mathbb{R}^{\sum_{i\in S}d_i}}|Z_{S,n}(\bm t)|^2\prod_{i\in S}\mathrm{d}w_i\right\}_{S\in\mathcal{T}}\right|^2=0.$$ 
							\end{enumerate}
							Then the distributional convergence in Theorem \ref{thm:H0} holds.
						\end{lemma}

						\noindent {\it Establishing Condition (1) in Lemma \ref{lem:convergence_cond}}: Fix a positive integer $K \geq 1$, denote $\kappa := K |\cT|$ dimensional random vector 
						\begin{align}\label{eq:KS}
							{\bm \cZ}_n:=   \left( \left( Z_{S, n}(\bm t_S^{(1)}),   Z_{S, n}(\bm t_S^{(2)}), \ldots,  Z_{S, n}(\bm t_S^{(K)} ) \right) \right)_{S \in \cT}  , 
						\end{align} 
						where $\bm t_S^{(\ell)} \in \mathbb R^{d_S}$, for $1 \leq \ell \leq K$. We will use the Cramer-Wold device to establish the joint asymptotic normality of $\cZ_n$. For this, suppose 
						$$\bm \eta = \left ( \bm \eta_S \right)_{S \in \cT} \in \mathbb R^{\kappa}, \text{ where } \bm \eta_S = ( \eta_S^{(1)}, \eta_S^{(2)}, \ldots, \eta_S^{(K)})$$
						and consider 
						\begin{align*}
							\sqrt n \langle \bm \eta, {\bm \cZ}_n \rangle = \sum_{S \in \cT} \sum_{\ell=1}^{K} \eta_S^{(\ell)} \left (\sqrt n Z_{S, n}(\bm t_S^{(\ell)}) \right) = \sum_{i=1}^{n}\big\{A_{i, \pi_1(i), \ldots, \pi_{d-1}(i)}+ \i B_{i, \pi_1(i), \ldots, \pi_{d-1}(i)}\big\} , 
						\end{align*}     
						where $A_{i_1, \ldots, i_r}$ and $B_{i_1, \ldots, i_r}$, for $1 \leq i_1, i_2, \ldots, i_r \leq n$ are the real and imaginary parts of $\sqrt n \langle \bm \eta, {\bm \cZ}_n \rangle$, respectively. For $a, b\in\mathbb{R}$ consider, 
						\begin{align}\label{eq:C-W_1}
							a\sum_{i=1}^{n}A_{i, \pi_1(i), \ldots, \pi_{d-1}(i)} + b\sum_{i=1}^{n}B_{i, \pi_1(i), \ldots, \pi_{d-1}(i)}.
						\end{align}
						Note from \eqref{eq:ZnSpf} the modulus of each summand in \eqref{eq:C-W_1} is bounded by $C/n^{\frac{1}{2}}$, for some constant $C$ depending only on $r, \bm \eta, a, b$. 
						Moreover, 
						\begin{align*}
							\E\left[\prod_{i\in S}\left (\frac{1}{n}\sum_{a=1}^{n}e^{\i\langle t_i^{(\ell)},\R_i(X_{i}^{(a)})\rangle}-e^{\i\langle t_i^{(\ell)},\R_i(X_{i}^{(b)})\rangle} \right)\right]=0 , 
						\end{align*} 
						for $\bm t_S^{(\ell)} \in \mathbb R^{d_S}$. 
						Hence, $A_{i_1,\ldots,i_r},B_{i_1,\ldots,i_r}$ are centered around $0$. Moreover, 
						\begin{align} 
							&
							\Cov \left[ a \sum_{i=1}^{n}A_{i, \pi_1(i), \ldots, \pi_{d-1}(i)},  b \sum_{i=1}^{n} B_{i, \pi_1(i), \ldots \pi_{d-1}(i)} \right] \nonumber \\
							&
							= ab \Cov \left[ \sum_{i=1}^{n}A_{i, \pi_1(i), \ldots, \pi_{d-1}(i)}, \sum_{i=1}^{n}B_{i, \pi_1(i), \ldots, \pi_{d-1}(i)} \right] \nonumber \\
							&
							=\frac{ab}{2}\mathrm{Im} \left( \Var \left[ \sum_{S \in \cT} \sum_{\ell=1}^{K} \eta_S^{(\ell)} \left (\sqrt n Z_{S, n}(\bm t_S^{(\ell)}) \right)  \right]  \right) \nonumber \\
							&
							=\frac{ab}{2}\sum_{S \in \cT} \mathrm{Im} \left( \Var \left[ \sum_{\ell=1}^{K} \eta_S^{(\ell)} \left (\sqrt n Z_{S, n}(\bm t_S^{(\ell)}) \right) \right] \right) \tag*{ (by \eqref{eq:CS1S2}) } \nonumber \\
							& \rightarrow \frac{ab}{2}\sum_{S \in \cT}  \mathrm{Im} \left( \sum_{\ell=1}^K (\eta_S^{(\ell)})^2 C_S(\bm t_S^{(\ell)}, \bm t_S^{(\ell)}) + \sum_{1 \leq \ell, \ell' \leq K} \eta_S^{(\ell)}  \eta_S^{(\ell')} C_S(\bm t_S^{(\ell)}, \bm t_S^{(\ell')})
							\right) \tag*{ (by \eqref{eq:CSS}) } \nonumber \\ 
							& = \frac{ab}{2}\sum_{S \in \cT} \mathrm{Im} \left( \Var \left[ \sum_{\ell=1}^{K} \eta_S^{(\ell)}  Z_{S}(\bm t_S^{(\ell)}) \right] \right) .  \label{eq:covarianceZnS} 
						\end{align}
						from Definition \ref{defn:ZSprocess}. 
						
						To compute variance, we have the following lemma which follows directly from definitions. 
						\begin{lemma}\label{lem:complex_var}
							Suppose $Z=X+iY$, where $X,Y$ are two real-valued variables. Define $\E[Z]=\E[X]+i\E[Y]$. Then $\mathrm{Var}[X]=\frac{1}{2}\left(\mathrm{Re}[\E[Z^2]]+\E[|Z|^2]\right)$ whenever $\E[X]=0$.
						\end{lemma}
						
						By Lemma \ref{lem:complex_var}, 
						\begin{align}\label{eq:ZSnvariance}
							\mathrm{Var}\left[\sum_{i=1}^{n}A_{i, \pi_1(i), \ldots, \pi_{d-1}(i)}\right]=\frac{1}{2} \left( \mathrm{Re}\left(\E\left[n \langle \bm \eta, {\bm \cZ}_n \rangle^2\right]\right) +\E\left[|\sqrt n \langle \bm \eta, {\bm \cZ}_n \rangle|^2\right] \right) . 
						\end{align}
						For the second term in the RHS above Lemma \ref{lm:Zcovariance} implies that 
						\begin{align}\label{eq:complex_variance}
							\E\left[|\sqrt n \langle \bm \eta, {\bm \cZ}_n \rangle|^2\right] \nonumber \\ 
							&
							=\sum_{S\in\mathcal{T}}\E\left[\left|\sum_{\ell=1}^K\eta_{S}^{(\ell)}\left(\sqrt{n}Z_{S,n}(\bm t_{S}^{(\ell)})\right)\right|^2\right] \nonumber \\
							&
							= \sum_{S\in\mathcal{T}}\sum_{\ell,\ell'=1}^K\eta_{S}^{(\ell)}\eta_{S}^{(\ell')}\E\left[nZ_{S,n}( \bm t_{S}^{(\ell)})\bar{Z}_{S,n}( \bm t_{S}^{(\ell')})\right] \nonumber \\ 
							& \rightarrow \sum_{S\in\mathcal{T}}\sum_{1 \leq \ell,\ell' \leq K} \eta_{S}^{(\ell)}\eta_{S}^{(\ell')}  C_S(\bm t_{S}^{(\ell)}, \bm t_{S}^{(\ell')} ) \nonumber \\ 
							& = 						
							\E\left[ \left| \sum_{\ell=1}^{K} \eta_S^{(\ell)}  Z_{S}(\bm t_S^{(\ell)}) \right|^2 \right] . 
						\end{align}
					Similarly, we can show that  
						\begin{align}\label{eq:real_variance}
							\E\left[n \langle \bm \eta, {\bm \cZ}_n \rangle^2\right] 
							&
							=\sum_{S\in\mathcal{T}}\sum_{\ell,\ell'=1}^K\eta_{S}^{(\ell)}\eta_{S}^{(\ell')}\E\left[nZ_{S,n}(t_{S}^{(\ell)})Z_{S,n}(t_{S}^{(\ell')})\right] \nonumber \\ 
							& \rightarrow \E\left[ \left( \sum_{\ell=1}^{K} \eta_S^{(\ell)}  Z_{S}(\bm t_S^{(\ell)}) \right)^2 \right] . 
						\end{align}
						Therefore, combining \eqref{eq:ZSnvariance} with \eqref{eq:complex_variance} and \eqref{eq:real_variance} gives, 						\begin{align*} 	
							\lim_{n\rightarrow\infty}\mathrm{Var}\left[\sum_{i=1}^{n}A_{i, \pi_1(i), \ldots, \pi_{d-1}(i)}\right] 
							&
							=\lim_{n\rightarrow\infty} \frac{1}{2} \left(\mathrm{Re}\left(\E\left[n \langle \bm \eta, {\bm \cZ}_n \rangle^2\right]\right)+\E\left[|\sqrt n \langle \bm \eta, {\bm \cZ}_n \rangle|^2\right] \right) \\ 
							& = \frac{1}{2} \left(\mathrm{Re}\left(\E\left[ \langle \bm \eta, {\bm \cZ} \rangle^2\right]\right)+\E\left[| \langle \bm \eta, {\bm \cZ} \rangle|^2\right] \right) , 
						\end{align*} 
						where 
\begin{align}\label{eq:KSZ}
{\bm \cZ}:=   \left( \left( Z_{S}(\bm t_S^{(1)}),   Z_{S}(\bm t_S^{(2)}), \ldots,  Z_{S}(\bm t_S^{(K)} ) \right) \right)_{S \in \cT} . 
\end{align}					
						 Hence, by  Theorem \ref{thm:clt}, under $H_0$ (recall \eqref{eq:KS}),
						\begin{align*}
							{\bm \cZ}_n \stackrel{D} \rightarrow {\bm \cZ}, 
		\end{align*} 
	This establishes the finite dimensional convergence of the process $\sqrt{n}Z_{S,n}(\bm t)$. To establish the weak convergence of the process $\sqrt{n}Z_{S,n}(\bm t)\stackrel{D}{\longrightarrow}Z_{S}(\bm t)$ in $L^{\infty}(D_{S}(\delta))$, we need to prove the following stochastic equicontinuity statement for all $ S\in\mathcal{T}$.  
						
						\begin{lemma}\label{lem:equicontinuity}
							For any given $\eta>0$,
							\begin{align}\label{eq:equicontinuity}
				\lim_{\varepsilon\rightarrow0}\limsup_{n\rightarrow\infty}\P\left\{\sup_{\mathop{}_{\|\bm t-\bm t'\|\leq\varepsilon}^{\bm t,\bm t'\in D_{S}(\delta) } }\big|\sqrt{n}Z_{S,n}(\bm t)-\sqrt{n}Z_{S,n}(\bm t')\big|>\eta\right\}=0.
							\end{align} 
						\end{lemma}

						\begin{proof}[Proof of Lemma \ref{lem:equicontinuity}]
							By  \cite[Theorem 2.11.1]{van1996weak} to prove the the Lemma \eqref{lem:equicontinuity} it is enough to show the following: 
							
							\begin{enumerate}
								\item There exists a sequence of $\eta_n$ such that 
								\begin{align}\label{eq:equicontinuity1}
									\sup_{1\leq b\leq n}\sup_{\bm t\in D_{S}(\delta)}\left|\Gamma(\bm t,b)\right|\leq\eta_n,\ \Gamma(\bm t,b):=\frac{1}{\sqrt{n}}\prod_{i \in S} \left (\frac{1}{n}\sum_{a=1}^{n}e^{\i\langle t_i,\R_i(X_{i}^{(a)})\rangle}-e^{\i\langle t_i,\R_i(X_{i}^{(b)})\rangle} \right) . 
								\end{align}
								\item For any sequence $\varepsilon_n\downarrow0$, 
								\begin{align}\label{eq:equicontinuity2}
									\sup_{\bm t,\bm t'\in D_{S}(\delta),\|\bm t-\bm t'\|\leq \varepsilon_n}\E\bigg|\sqrt{n}Z_{S,n}(\bm t)-\sqrt{n}Z_{S,n}(\bm t')\bigg|^2\rightarrow 0.
								\end{align}
								\item For any sequence $\varepsilon_n\downarrow0$
								\begin{align}\label{eq:equicontinuity3}
									\int_{0}^{\varepsilon_n}\left[\log\left(N\left(\eta,D_{S}(\delta),d_n(\cdot,\cdot)\right)\right)\right]^{\frac{1}{2}}\mathrm{d}\eta\stackrel{P}{\rightarrow} 0 , 
								\end{align}
								where $N(\eta,D_{S}(\delta),d_n(\cdot,\cdot))$ denotes the $\eta$ covering number of the set $D_{S}(\delta)$ based on the random metric $d_n(\cdot,\cdot)$ satisfying 
								\begin{align}\label{eq:dn}
									d_n^2(\bm t,\bm t')=\sum_{b=1}^{n}\left|\Gamma(\bm t, b)-\Gamma(\bm t',b)\right|^2.
								\end{align}
							\end{enumerate}

							Note that $|\Gamma(\bm t,b)|\leq 2^{|S|}/\sqrt{n}$ so that \eqref{eq:equicontinuity1} is immediate. The next lemma proves \eqref{eq:equicontinuity2}. 		
							
							\begin{lemma}\label{lem:lipschitz_Z_n}
								There exists a constant $C > 0$ is depending only on $d_i$, for $1 \leq i \leq r$, such that 
								\begin{align*}
									\E\big|\sqrt{n}Z_{S,n}(\bm t)-\sqrt{n}Z_{S,n}(\bm t')\big|^{2}\leq C\|\bm t-\bm t'\|.
								\end{align*} 
								Consequently, the condition in \eqref{eq:equicontinuity2} holds. 
							\end{lemma}
							
							\begin{proof}  For ease of notation define, 
								\begin{align*}
									F^{(1)}(\bm t,\bm t', i):=
									&
									\frac{1}{n}\sum_{u=1}^{n}e^{\i\langle t_{i}-t_i',h_u^{d_i})\rangle}-1\\
									F^{(2)}(\bm t,i):=
									&
									1-\left(\frac{1}{n}\sum_{u=1}^{n}e^{\i\langle t_{i},h_{u}^{d_i}\rangle}\right)\left(\frac{1}{n}\sum_{a=1}^{n}e^{\i\langle -t_{i},h_{a}^{d_i}\rangle}\right)\\
									F^{(3)}(\bm t,\bm t', i):=
									&
									\bigg(\frac{1}{n}\sum_{u=1}^{n}e^{\i\langle t_i',h_{u}^{d_i}\rangle}\bigg)\bigg(\frac{1}{n}\sum_{u=1}^{n}e^{\i\langle -t_i',h_u^{d_i}\rangle}\bigg)-\bigg(\frac{1}{n}\sum_{u=1}^{n}e^{\i\langle t_{i},h_u^{d_i}\rangle}\bigg)\bigg(\frac{1}{n}\sum_{u=1}^{n}e^{\i\langle -t_i',h_u^{d_i}\rangle}\bigg).
								\end{align*}
				Then as in the proof of Lemma \ref{lm:Zcovariance} we can  compute
								\begin{align} \label{eq:ZSnT14}
									&
									\E\left|\sqrt{n}Z_{S,n}(\bm t)-\sqrt{n}Z_{S,n}(\bm t')\right|^2 \nonumber \\
									&
									=\E\left|\sqrt{n}Z_{S,n}(\bm t)\right|^2+\E\left|\sqrt{n}Z_{S,n}(\bm t')\right|^2-\E\left[nZ_{S,n}(\bm t)\bar{Z}_{S,n}(\bm t')\right]-\E\left[n\bar{Z}_{S,n}(\bm t)Z_{S,n}(\bm t')\right] \nonumber \\ 
									&
									=c_{S,n}(T_{1}+T_2+T_3+T_4) .
								\end{align}
								where $c_{S,n}:=1+\frac{(-1)^{|S|}}{(n-1)^{|S|-1}}$ and 
								\begin{align*}
									T_1
									&
									=\prod_{i\in S}F^{(2)}(\bm t,i),\\ 
									T_2
									&
									=\prod_{i\in S}F^{(2)}(\bm t', i),\\ 
									T_3
									&
									=-\prod_{i\in S}\left\{F^{(1)}(\bm t, \bm t',i)+F^{(2)}(\bm t',i)+F^{(3)}(\bm t,\bm t',i)\right\},\\
									T_4
									&
									=-\prod_{i\in S}\left\{\bar{F}^{(1)}(\bm t, \bm t',i)+F^{(2)}(\bm t,i)+F^{(3)}(\bm t',\bm t,i)\right\} , 
								\end{align*} 
								where $\bar{F}^{(1)}$ denotes the complex conjugate of $F^{(1)}$. 
								We will now show that there exists universal constants $C_1,C_2$ such that for all $i \in S$, 
								\begin{align*}
									|F^{(1)}(\bm t,\bm t',i)|\leq C_{1}\|\bm t-\bm t'\| \text{ and } \ |F^{(3)}(\bm t,\bm t', i)|\leq C_{2}\|\bm t-\bm t'\| . 
								\end{align*}
								Towards this, using $|e^{\i t}-1|\leq |t|$ and the Cauchy-Schwarz inequality, we have
								\begin{align*}
									|F^{(1)}(\bm t,\bm t',i)|\leq \frac{1}{n}\sum_{u=1}^n|\langle t_i-t_i',h_u^{d_i}\rangle|\leq \|\bm t-\bm t'\|\sqrt{d_0}.
								\end{align*}
								Similarly, we bound
								\begin{align*}
									|F^{(3)}(\bm t,\bm t',i)|
									&
									\leq\left|\frac{1}{n}\sum_{u=1}^ne^{\i\langle t_i', h_{u}^{d_i}\rangle}-\frac{1}{n}\sum_{u=1}^ne^{\i \langle t_{i},h_{u}^{d_i}\rangle}\right|\\
									&
									\leq \left|\frac{1}{n}\sum_{u=1}^ne^{\i\langle t_i'-t_i,h_u^{d_i}\rangle}-1\right|\\
									&
									\leq \|\bm t-\bm t'\|\sqrt{d_0}.
								\end{align*}
								Also, we notice that $|F^{(1)}(\bm t,\bm t',i)|,|F^{(2)}(\bm t,i)|,|F^{(3)}(\bm t,\bm t',i)|\leq2$ for all $\bm t,\bm t'\in\mathbb{R}^{d_0}$ and $i\in S$. This implies, 
								\begin{align*}
									|T_1+T_2+T_3+T_4|\leq C\|\bm t-\bm t'\|. 
								\end{align*} 
									Applying this in \eqref{eq:ZSnT14} completes the proof of  Lemma \ref{lem:lipschitz_Z_n}. 
							\end{proof}
							
							Now we prove \eqref{eq:equicontinuity3}. For this we need the following lemma:							
							\begin{lemma}\label{lem:lip_dist}
								There exists a universal constant $C_0 > 0$								\begin{align*}
									d_n(\bm t,\bm t')\leq C_0\|\bm t-\bm t'\| , 
								\end{align*} 
								for $d_n(\cdot, \cdot)$ as defined in \eqref{eq:dn}. 
							\end{lemma} 
							
							\begin{proof}
								We begin by computing 
									$d^2_n(\bm t,\bm t')
									=\sum_{b=1}^n|\Gamma(\bm t,b)-\Gamma(\bm t',b)|^2$.				Note that 
								\begin{align*}
									|\Gamma(\bm t,b)-\Gamma(\bm t',b)|=\frac{1}{\sqrt{n}}\left|\prod_{i\in S}\left(\frac{1}{n}\sum_{a=1}^ne^{\i \langle t_i,h_{a}^{d_i}\rangle}-e^{\i\langle t_i,h_b^{d_i}\rangle}\right)-\prod_{i\in S}\left(\frac{1}{n}\sum_{a=1}^ne^{\i \langle t_i',h_{a}^{d_i}\rangle}-e^{\i\langle t_i',h_b^{d_i}\rangle}\right)\right|.
								\end{align*}
								Now, recall the following inequality
								\begin{align}\label{eq:prod_sum}
									\left|\prod_{a=1}^nw_a-\prod_{a=1}^nv_a\right|\leq\sum_{a=1}^n|w_a-v_a|, \quad |w_a|,|v_a|\leq 1, \text{ for all } 1 \leq a \leq n .
								\end{align}
								Then 
								\begin{align*}
									|\Gamma(\bm t,b)-\Gamma(\bm t',b)|\leq \frac{2^{|S|-1}}{\sqrt{n}}\sum_{i\in S}\left|\frac{1}{n}\sum_{a=1}^ne^{\i \langle t_i,h_{a}^{d_i}\rangle}-\frac{1}{n}\sum_{a=1}^ne^{\i \langle t_i',h_{a}^{d_i}\rangle}+e^{\i \langle t_i',h_{b}^{d_i}\rangle}-e^{\i \langle t_i,h_{b}^{d_i}\rangle}\right|.
								\end{align*}
								Note that 
								\begin{align*}
									\left|\frac{1}{n}\sum_{a=1}^ne^{\i \langle t_i,h_{a}^{d_i}\rangle}-\frac{1}{n}\sum_{a=1}^ne^{\i \langle t_i',h_{a}^{d_i}\rangle}\right|\leq \frac{1}{n}\sum_{a=1}^n\left|e^{\i \langle t_i,h_{a}^{d_i}\rangle}-e^{\i \langle t_i',h_{a}^{d_i}\rangle}\right| & \leq \frac{1}{n}\|\bm t-\bm t'\|\|h_a^{d_1}\| \nonumber \\ 
									& \leq\sqrt{d_0}\|\bm t-\bm t'\|.
								\end{align*}
								Similarly, 
								\begin{align*}
									|e^{\i \langle t_i, h_b^{d_i}}-e^{\i \langle t_i',h_b^{d_i}\rangle}|\leq \|\bm t-\bm t'\|\|h_b^{d_i}\|\leq \sqrt{d_0}\|\bm t-\bm t'\|.
								\end{align*}
								Therefore, 
								\begin{align*}
									|\Gamma(\bm t,b)-\Gamma(\bm t',b)|\leq \frac{2^{|S|}}{\sqrt{n}}\sum_{i\in S}\sqrt{d_0}\|\bm t-\bm t'\|=\frac{|S|\sqrt{d_0}}{\sqrt{n}2^{|S|-1}}\|\bm t-\bm t'\| , 
								\end{align*}
								which completes the proof of Lemma \ref{lem:lip_dist}. 
							\end{proof} 
						
						To complete the proof of \eqref{eq:equicontinuity3} we compute the following covering number: 
							\begin{align*}
								N(\eta,D_{S}(\delta),\|\cdot\|)\leq \left(\frac{3}{\eta}\right)^{d_0}\frac{\mathrm{vol}(D_{S}(\delta))}{\mathrm{vol}(B_1(0))}\leq \left(\frac{3}{\eta}+1\right)^{d_0}\frac{\mathrm{vol}(D_{S}(\delta))}{\mathrm{vol}(B_1(0))}.
							\end{align*}
							Then by Lemma \ref{lem:lip_dist}, 							\begin{align*}
								& \int_{0}^{\varepsilon_n}\left[\log\left(N(\eta,D_{S}(\delta),d_n(\cdot,\cdot))\right)\right]^{\frac{1}{2}}\mathrm{d}\eta \nonumber \\ 
								&
								\leq \int_{0}^{\varepsilon_n}\left[\log\left(N(\eta/C_0,D_{S}(\delta),\|\cdot\|_2)\right)\right]^{\frac{1}{2}}\mathrm{d}\eta\\
								&
								\leq \int_{0}^{\varepsilon_n}\left(\log \left(\frac{\mathrm{vol}(D_{S}(\delta))}{\mathrm{vol}(B_1(0))} \right)\right)^{\frac{1}{2}}\mathrm{d}\eta+\int_{0}^{\varepsilon_n} \left(\frac{3C_0d_0}{\eta}\right)^{\frac{1}{2}}\mathrm{d}\eta\\
								&
								\rightarrow0
							\end{align*}
							as $\varepsilon_n\rightarrow0.$ This completes the proof of \eqref{eq:equicontinuity3} and, hence, the proof of Lemma \ref{lem:equicontinuity}.  
						\end{proof}
				
				Lemma \ref{lem:equicontinuity} together with the finite dimensional convergence established earlier implies that the process $\sqrt{n}Z_{S,n}(\bm t)\stackrel{D}{\longrightarrow}Z_{S}(\bm t)$ in $L^{\infty}(D_{S}(\delta))$. This completes the proof of condition (1) in Lemma \ref{lem:convergence_cond}. \\ 
						
						\noindent {\it Establishing Condition (3) in Lemma \ref{lem:convergence_cond}}:
					Note that for all $S\in \mathcal{T}$, 
						\begin{align*}
							\E\left|n\int_{D_{S}(\delta)}|Z_{S,n}(\bm t)|^2\prod_{i\in S}\mathrm{d}w_i-n\int_{\mathbb{R}^{d_0}}|Z_{S,n}(\bm t)|^2\prod_{i\in S}\mathrm{d}w_i\right|=\E\left|n\int_{D_{S}^c(\delta)}|Z_{S,n}(\bm t)|^2\prod_{i\in S}\mathrm{d}w_i\right|.
						\end{align*}
						Then by triangle inequality and Fubini's theorem, to establish Condition (3) in Lemma \ref{lem:convergence_cond} it is enough to show that 
						\begin{align}\label{eq:complement_converg}
						\lim_{\delta\rightarrow0}\limsup_{n\rightarrow\infty}\int_{D_S^c(\delta)}n\E\left|Z_{S,n}(\bm t)\right|^2\prod_{i\in S}\mathrm{d}w_i=0.
						\end{align}
						Recalling the definition of $D_S(\delta)=\prod_{i\in S}T_i(\delta)$, we have the following inequality, 
						\begin{align}\label{eq:ZSndelta}
							& \int_{D_S^c(\delta)}n\E\left|Z_{S,n}(\bm t)\right|^2\prod_{i\in S}\mathrm{d}w_i \nonumber \\ 
							& \leq \sum_{i\in S}\left\{\int_{\|t_i\|< \delta}n\E\left|Z_{S,n}(\bm t)\right|^2\prod_{i\in S}\mathrm{d}w_i+\int_{\|t_i\|>1/\delta}n\E\left|Z_{S,n}(\bm t)\right|^2\prod_{i\in S}\mathrm{d}w_i\right\}.
						\end{align} 
						Now, define \begin{align*}
							G_i(\delta,t)=\int_{\|w\|<\delta}\frac{1-\cos(\langle t,w\rangle)}{c_{d_i}\|w\|^{d_i+1}}\mathrm{d}w . 
						\end{align*}
						Then by Lemma \ref{lm:Zcovariance}, we obtain
						\begin{align}
							&
							\int_{\|t_i\|< \delta}n\E\left|Z_{S,n}(\bm t)\right|^2\prod_{i\in S}\mathrm{d}w_i \nonumber \\
							&
							= \int_{\|t_i\|< \delta}\left(1+\frac{(-1)^{|S|}}{(n-1)^{|S|-1}}\right)\prod_{i\in S}\left[1-\left\{\frac{1}{n}\sum_{u=1}^{n}e^{\i\langle t_{i},h_u^{d_i})\rangle}\right\}\left\{\frac{1}{n}\sum_{u=1}^{n}e^{\i\langle -t_{i},h_u^{d_i})\rangle}\right\}\right]\prod_{i\in S}\mathrm{d}w_i \nonumber \\
							&
							\leq 2\int_{\|t_i\|<\delta}\prod_{i\in S}\left[1-\left\{\frac{1}{n}\sum_{u=1}^{n}e^{\i\langle t_{i},h_u^{d_i})\rangle}\right\}\left\{\frac{1}{n}\sum_{u=1}^{n}e^{\i\langle -t_{i},h_u^{d_i})\rangle}\right\}\right]\prod_{i\in S}\mathrm{d}w_i \tag*{(by \eqref{eq:ZSndelta})} \nonumber \\
							&
							=\frac{2}{n^{2|S|}} \sum_{\substack{u_{1},\ldots,u_{|S|} \\ v_1,\ldots,v_{|S|}}}\int_{\|t_i\|<\delta}\frac{1-\cos\left(\langle t_i,h_{u_i}^{d_i}-h_{v_i}^{d_i}\rangle\right)}{c_{d_i}\|t_i\|^{1+d_i}}\mathrm{d}t_i\prod_{i'\in S\backslash\{i\}}\left(1-\cos\left(\langle t_{i'},h_{u_{i'}}^{d_{i'}}-h_{v_{i'}}^{d_{i'}}\rangle\right)\right)\prod_{i'\in S\backslash\{i\}}\mathrm{d}w_{i'} \nonumber \\
							&
							=\frac{2}{n^{2|S|}} \sum_{\substack{u_{1},\ldots,u_{|S|} \\ v_1,\ldots,v_{|S|} } } G_i(\delta,h_{u_i}^{d_i}-h_{v_i}^{d_i})\prod_{i'\in S\backslash\{i\}}\|h_{u_{i'}}^{d_{i'}}-h_{v_{i'}}^{d_{i'}}\|. \label{eq:GZSn}
						\end{align}
						Note that the second equality above holds due to the fact that $\sin$ is an odd function and hence integrates to $0$ over symmetric sets and the last equality holds due to \cite[Lemma 1]{szekely2007measuring}. Now, suppose $t\in\{t:\|t\|\leq \sqrt{d_0}\}$. Then by \cite[Lemma 1]{szekely2007measuring},  
						\begin{align*}
							G_i(y,t)\leq \|t\|\leq \sqrt{d_0}.
						\end{align*}
						Also, by dominated convergence theorem, $\lim_{\delta\rightarrow0}G_i(\delta,t)=0$ for any $t\in \{t:\|t\|\leq \sqrt{d_0}\}$. Then by Assumption \ref{assumption:U} and \eqref{eq:GZSn}, 
						\begin{align*}
							\lim_{n\rightarrow\infty}\int_{\|t_i\|< \delta}n\E\left|Z_{S,n}(\bm t)\right|^2\prod_{i\in S}\mathrm{d}w_i\leq 2\E\left[G_{i}(\delta,U_i^{a}-U_i^{b})\prod_{i'\in S\backslash\{i\}}\|U_{i'}^{a}-U_{i'}^{b}\|\right] , 
						\end{align*}
						where $U_i^{a},U_i^{b}$ are two independent samples from $\mathrm{Unif}([0,1]^{d_i})$. Applying the dominated convergence theorem again gives, 
						\begin{align*}
							\lim_{\delta\rightarrow0}\E\left[G_{i}(\delta,U_i^{a}-U_i^{b})\prod_{i'\in S\backslash\{i\}}\|U_{i'}^{a}-U_{i'}^{b}\|\right]=0 , 
						\end{align*} 
						which implies, 
						\begin{align}\label{eq:ZSndelta1}
							\lim_{\delta \rightarrow 0} \lim_{n\rightarrow\infty}\int_{\|t_i\|< \delta}n\E\left|Z_{S,n}(\bm t)\right|^2\prod_{i\in S}\mathrm{d}w_i\leq 2\E\left[G_{i}(\delta,U_i^{a}-U_i^{b})\prod_{i'\in S\backslash\{i\}}\|U_{i'}^{a}-U_{i'}^{b}\|\right]  = 0 . 
						\end{align}
						
						 On the other hand, 
						\begin{align*}
							&
							\int_{\|t_i\|>1/\delta}n\E\left|Z_{S,n}(\bm t)\right|^2\prod_{i\in S}\mathrm{d}w_i\\
							&
							\leq \frac{2}{n^{2|S|}} \sum_{\substack{u_{1},\ldots,u_{|S|} \\ v_1,\ldots,v_{|S|} }} \int_{\|t_i\|>1/\delta}\frac{1-\cos\left(\langle t_i,h_{u_i}^{d_i}-h_{v_i}^{d_i}\rangle\right)}{c_{d_i}\|t_i\|^{1+d_i}}\mathrm{d}t_i\prod_{i'\in S\backslash\{i\}}\left(1-\cos\left(\langle t_{i'},h_{u_{i'}}^{d_{i'}}-h_{v_{i'}}^{d_{i'}}\rangle\right)\right)\prod_{i'\in S\backslash\{i\}}\mathrm{d}w_{i'}\\
							&
							=\frac{2}{n^{2|S|}} \sum_{\substack{u_{1},\ldots,u_{|S|} \\ v_1,\ldots,v_{|S|}} }\int_{\|t_i\|>1/\delta}\frac{1-\cos\left(\langle t_i,h_{u_i}^{d_i}-h_{v_i}^{d_i}\rangle\right)}{c_{d_i}\|t_i\|^{1+d_i}}\mathrm{d}t_i\prod_{i'\in S\backslash\{i\}}\|h_{u_{i'}}^{d_{i'}}-h_{v_{i'}}^{d_{i'}}\|\\
							&
							\leq \frac{2}{n^{2|S|}} \sum_{\substack{u_{1},\ldots,u_{|S|} \\ v_1,\ldots,v_{|S|}} }\prod_{i'\in S\backslash\{i\}}\|h_{u_{i'}}^{d_{i'}}-h_{v_{i'}}^{d_{i'}}\|\int_{\|t_i\|>1/\delta}\frac{1}{c_{d_i}\|t_i\|^{1+d_i}}\mathrm{d}t_i.
						\end{align*}
						Now, note that 
						\begin{align*}
							\int_{\|t_i\|>1/\delta}\frac{1}{c_{d_i}\|t_i\|^{1+d_i}}\mathrm{d}t_i\leq K_{d_i}\delta , 
						\end{align*}
						where $K_{d_i}$ is a constant that only depends on $d_i$. This implies, 
						\begin{align*}
							\lim_{n\rightarrow\infty}\int_{\|t_i\|>1/\delta}n\E\left|Z_{S,n}(\bm t)\right|^2\prod_{i\in S}\mathrm{d}w_i
							&
							\leq \lim_{n\rightarrow\infty}\frac{2}{n^{2|S|}} \sum_{\substack{u_{1},\ldots,u_{|S|} \\ v_1,\ldots,v_{|S|}}} \prod_{i'\in S\backslash\{i\}}\|h_{u_{i'}}^{d_{i'}}-h_{v_{i'}}^{d_{i'}}\|K_{d_i}\delta\\
							&
							=2\E\left[\prod_{i'\in S\backslash\{i\}}\|U_{i'}^a-U_{i'}^b\|\right] K_{d_i}\delta . 
						\end{align*}
						Taking $\delta\rightarrow0$ on both sides above we conclude
						\begin{align}\label{eq:ZSndelta2}
						\lim_{\delta\rightarrow0}\lim_{n\rightarrow\infty}\int_{\|t_i\|>1/\delta} n\E\left|Z_{S,n}(\bm t)\right|^2\prod_{i\in S}\mathrm{d}w_i=0.
						\end{align} 
						
						Combining \eqref{eq:ZSndelta1} and \eqref{eq:ZSndelta2} establishes \eqref{eq:complement_converg}, which completes the proof of statement (3) in Lemma \ref{lem:convergence_cond}. \\ 
						
						\noindent {\it Establishing Condition (2) in Lemma \ref{lem:convergence_cond}}: First note that 
						\begin{align*}
							Z_{S}(\bm t)\bm{1}\big\{\bm t\in D_{S}(\delta)\big\}\stackrel{a.s.}{\rightarrow}Z_{S}(\bm t)\bm{1}\big\{\bm t\in \mathbb{R}^{d_1}\times\cdots\times \mathbb{R}^{d_{|S|}}\big\}.
						\end{align*}
						Also, 
						\begin{align*}
							\int_{\mathbb{R}^{\sum_{i\in S}d_i}}\E\left[|Z_{S}(\bm t)|^2\right]\prod_{i\in S}\mathrm{d}w_i
							&
							=\int_{\mathbb{R}^{\sum_{i\in S}d_i}} \prod_{i\in S}\left[1-\E\left[e^{\i\langle t_{i},U_{i}\rangle}\right]\E\left[e^{\i\langle -t_{i},U_{i}\rangle}\right]\right]\prod_{i\in S}\d w_{i}\\
							&
							=\prod_{i\in S}\E\left\|U_{i}-U'_{i}\right\|<\infty.
						\end{align*}
						where $U_{i},U'_{i}\sim \mathrm{Unif}([0,1]^{d_i})$. 
					Then by dominated convergence theorem, 
					$$\int_{D_{S}(\delta)}|Z_{S}(\bm t)|^2\prod_{i\in S}\mathrm{d}w_i\stackrel{a.s.}{\rightarrow}\int_{\mathbb{R}^{\sum_{i\in S}d_i}}|Z_{S}(\bm t)|^2\prod_{i\in S}\mathrm{d}w_i,$$ which establishes  Condition (2) in Lemma \ref{lem:convergence_cond}. \hfill $\Box$

					\section{Consistency with Finite Resamples}
					\label{sec:consistencypermutationpf} 
					
					In this section we prove the technical result required for establishing the consistency of the resampling based implementation of the test as shown in Theorem \ref{thm:consistency_permutation}. 					
					
					\begin{proposition}\label{ppn:consistencypermutation}
						Let $\bm \pi=\{\pi_1,\ldots,\pi_{r}\}$ be a collection of independent uniform random permutations in $S_n$. Then for any $S \subseteq \{1, 2, \ldots, r\}$ with $|S| \geq 2$,  
						\begin{align}\label{eq:permutationdCov}
						\theta_{S}(\bm h_{S, \bm \pi }^{(1)}, \ldots, \bm h_{S, \bm \pi}^{(n)} ) \stackrel{P}{\rightarrow} 0 , 
						\end{align} 
						where $\theta_S$ is as defined in \eqref{eq:thetaS} and $\bm h_{S, \bm \pi}^{(a)} := (h_{\pi_i(a)}^{d_i})_{i \in S} \in \mathbb R^{d_S}$, for $1 \leq a \leq n$. Consequently, recalling \eqref{eq:pi}, 
$$\RdCov_n^{(1)}(\bm X; \bm C) \stackrel{P}{\rightarrow} 0.$$
					\end{proposition}

					\begin{proof}[Proof of Proposition \ref{ppn:consistencypermutation}] 
						 Note that since $\theta_S$ is always non-negative, to prove \eqref{eq:permutationdCov} it suffices to show that 
						\begin{align*}
							\lim_{n\rightarrow\infty}\E[ \theta_{S}(\bm h_{S, \bm \pi }^{(1)}, \ldots, \bm h_{S, \bm \pi}^{(n)} ) ] \rightarrow 0.
						\end{align*}
						 Recalling \eqref{eq:RdCovWXS}, \eqref{eq:estimateW} and \eqref{eq:RdCovHalton}, $\theta_{S}(\bm h_{S, \bm \pi }^{(1)}, \ldots, \bm h_{S, \bm \pi}^{(n)} )$ can be written as 
						\begin{align*}
							\theta_{S}(\bm h_{S, \bm \pi }^{(1)}, \ldots, \bm h_{S, \bm \pi}^{(n)} )=\frac{1}{n^{2}}\sum_{1 \leq a, b \leq n}\prod_{i=1}^{r}\hat{w}_i(h_{\pi_i(a)}^{d_i},h_{\pi_i(b)}^{d_i}), 
						\end{align*} 
						where, for $x, y \in \mathbb R^{d_i}$,  
					\begin{align*}
							\hat{w}_i(x, y) 
							& = \frac{1}{n}\sum_{v=1}^n\| x - h_v^{d_i} \|+\frac{1}{n}\sum_{u=1}^n\| h_u^{d_i} - y \| \nonumber \\ 
							& \hspace{1.25in} -\| x - y \| -\frac{1}{n^2}\sum_{1 \leq u,v \leq n} \| h_u^{d_i} - h_v^{d_i} \| . 
						\end{align*} 
				Note that, for $1 \leq i \leq r$, 
								\begin{align*} 
		\E\left[ \hat{w}_i(h_{\pi_i(a)}^{d_i},h_{\pi_i(b)}^{d_i}) \right] = \frac{1}{n(n-1)} \sum_{1 \leq s \ne t \leq n} \hat{w}_i(h_{s}^{d_i},h_{t}^{d_i}) \rightarrow 0 , 
		\end{align*} 
		by Assumption \ref{assumption:U}. Hence, by independence of the permutations, 
						\begin{align*} 
							\E[ \theta_{S}(\bm h_{S, \bm \pi }^{(1)}, \ldots, \bm h_{S, \bm \pi}^{(n)} ) ] & =\frac{1}{n^{2}}\sum_{1 \leq a, b \leq n}\prod_{i=1}^{r} \E [\hat{w}_i(h_{\pi_i(a)}^{d_i},h_{\pi_i(b)}^{d_i}) ] \nonumber \\ 
							& = \prod_{i=1}^{r} \left( \frac{1}{n(n-1)} \sum_{1 \leq s \ne t \leq n} \hat{w}_i(h_{s}^{d_i},h_{t}^{d_i})  \right) \rightarrow 0. 
						\end{align*}
						This completes the proof of Proposition \ref{ppn:consistencypermutation}. 
					\end{proof}

					\section{Proof of Results in section \ref{sec:local_power}}\label{sec:local_power_proof} 
					
					This section is organized as follows: We begin with the proof of the H\'ajek projection result in Section \ref{sec:ZnSpf}. The proof of Theorem \ref{thm:mixture} is given Section \ref{sec:mixturepf}. Theorem \ref{thm:powerK} is proved in Section \ref{sec:powerKpf}.

					\subsection{Proof of Proposition \ref{ppn:ZnS}} 
					\label{sec:ZnSpf}
					
					To prove Proposition \ref{ppn:ZnS} we will show 
					\begin{align}\label{eq:ZSnt2} 
					\E_{H_0}|\sqrt{n}Z_{S,n}^{\o}(\bm t)-\sqrt{n}Z_{S,n}(\bm t)|^{2}=o(1) . 
					\end{align} 
					Hereafter, all moments will be computed under $H_0$, so we will omit the subscript $H_0$ for notational convenience. To begin with, note that 					\begin{align}\label{eq:ZSndifference}
						&
						\E|\sqrt{n}Z_{S,n}^{\o}(\bm t)-\sqrt{n}Z_{S,n}(\bm t)|^{2} \nonumber \\
						&
						=\E[nZ_{S,n}^{\o}(\bm t)\bar{Z}_{S,n}^{\o}(\bm t)] +\E[nZ_{S,n}(\bm t)\bar{Z}_{S,n}(\bm t)]-\E[nZ_{S,n}(\bm t)\bar{Z}_{S,n}^{\o}(\bm t)]-\E[n\bar{Z}_{S,n}(\bm t)Z_{S,n}^{\o}(\bm t)] \nonumber \\ 
						&
						:=A_n+B_n+C_n+D_n . 
					\end{align} 
					By Lemma \ref{lm:Zcovariance}, 
					$$\lim_{n \rightarrow \infty} B_n = \lim_{n \rightarrow \infty} \E[nZ_{S,n}(\bm t)\bar{Z}_{S,n}(\bm t)] = \lim_{n \rightarrow \infty} \mathrm{Cov}[Z_{S,n}(\bm t), \bar{Z}_{S,n}(\bm t)] = C_S(\bm t, \bm t).$$ 
					Similarly, as in the proof of Lemma \ref{lm:Zcovariance} we can show that 
					$$\lim_{n \rightarrow \infty} A_n = C_S(\bm t, \bm t).$$ 
				
					Now, consider $C_n$. For $1 \leq b \leq n$ and $i \in S$,  recall the definition of $Z_{S, n}(\bm t, b, i)$ from \eqref{eq:ZSnti} and define 					$$\bar Z_{S, n}^{\o}(\bm t, b, i):=  \frac{1}{n}\sum_{a=1}^{n}e^{-\i\langle t_i, R_{\mu_i}(X_{i}^{(a)})\rangle}-e^{-\i\langle t_i, R_{\mu_i} (X_{i}^{(b)})\rangle}.$$ 
					Then 
					\begin{align}\label{eq:ZSnt1} 
						C_n & =  \E[nZ_{S,n}(\bm t)\bar{Z}_{S,n}^{\o}(\bm t)] 
						\nonumber \\ 
						& = \E\left[
						 \frac{1}{n}\sum_{1 \leq b, b' \leq n} \prod_{i \in S} Z_{S, n}(\bm t, b, i) \bar Z_{S, n}^{\o}(\bm t, b', i) \right] \nonumber \\ 
						 & = \prod_{i\in S}\E[ Z_{S, n}( \bm t, 1, i)  \bar Z_{S, n}^{\o}( \bm t, 1, i) ] + (n-1)  \prod_{i\in S}\E[ Z_{S, n}( \bm t, 1, i)  \bar Z_{S, n}^{\o}( \bm t, 2, i) ] , 
					\end{align} 
using the distributional symmetry of the rank maps. Note that for any $i' \in S$, $$\sum_{b=1}^n \bar Z_{S, n}^{\o}(\bm t, b, i')  = 0,$$ hence, $\bar Z_{S, n}^{\o}(\bm t, 1, i')  = - \sum_{b=2}^n \bar Z_{S, n}^{\o}(\bm t, b, i') $.  This implies, 
					\begin{align*}
					\E[ Z_{S, n}(\bm t, 1, i')  \bar Z_{S, n}^{\o}(\bm t, 1, i') ] & =  - \sum_{b=2}^n \E[ Z_{S, n}(\bm t, 1, i') \bar Z_{S, n}^{\o}(\bm t, b, i')	] \nonumber \\ 
					& = - (n-1) \E[ Z_{S, n}(\bm t, 1, i') \bar Z_{S, n}^{\o}(\bm t, 2, i')]  .				\end{align*} 				
Hence, 
			\begin{align}\label{eq:ZSnt2} 
		 &	(n-1) \prod_{i \in S}\E[ Z_{S, n}(\bm t, 1, i)  \bar Z_{S, n}^{\o}(\bm t, 2, i) ] \nonumber \\ 
		 & = (n-1) \E[ Z_{S, n}(\bm t, 1, i') \bar Z_{S, n}^{\o}(\bm t, 2, i')] \prod_{i \in S\backslash\{i'\}}\E[ Z_{S, n}(\bm t, 1, i)  \bar Z_{S, n}^{\o}(\bm t, 2, i) ]  \nonumber \\ 
			& = - \E[ Z_{S, n}(\bm t, 1, i')  \bar Z_{S, n}^{\o}(\bm t, 1, i') ]  \prod_{i \in S\backslash\{i'\}}\E[ Z_{S, n}(\bm t, 1, i)  \bar Z_{S, n}^{\o}(\bm t, 2, i) ] \nonumber \\ 
				& =	 - \E[ Z_{S, n}(\bm t, 1, i')  \bar Z_{S, n}^{\o}(\bm t, 1, i') ]  \prod_{i \in S\backslash\{i'\}}\E[ Z_{S, n}(\bm t, 1, i) ] \E[ \bar Z_{S, n}^{\o}(\bm t, 2, i) ] = 0. 
					\end{align} 
					Combining \eqref{eq:ZSnt1} and \eqref{eq:ZSnt2} gives, 
					\begin{align}\label{eq:ZSnt4} 
						 C_n = \E[nZ_{S,n}(\bm t)\bar{Z}_{S,n}^{\o}(\bm t)] 
						& = \prod_{i\in S} \E[ Z_{S, n}( \bm t, 1, i)  \bar Z_{S, n}^{\o}( \bm t, 1, i) ] 
						\end{align} 
To compute the limit of the RHS of \eqref{eq:ZSnt4} note that 
					\begin{align*}
						& \left| Z_{S, n}( \bm t, 1, i)  - \bar Z_{S, n}^{\o}( \bm t, 1, i) \right| \\ 
						&
						= \left|\frac{1}{n}\sum_{u=1}^{n}e^{\i \langle t_{i},R_{\mu_i}(X_{i}^{(u)})\rangle}-e^{\i\langle t_{i},R_{\mu_i}(X_{i}^{(1)})\rangle}-\frac{1}{n}\sum_{u=1}^{n}e^{\i \langle t_{i},\R_{i}(X_{i}^{(u)})\rangle}+e^{\i\langle t_{i},\R_{i}(X_{i}^{(1)})\rangle} \right|\\
						&
						\leq \frac{1}{n}\sum_{u=1}^{n}\bigg|e^{\i \langle t_{i},R_{\mu_i}(X_{i}^{(u)})\rangle}-e^{\i \langle t_{i},\R_{i}(X_{i}^{(u)})\rangle}\bigg|+\bigg|e^{\i\langle t_{i},R_{\mu_i}(X_{i}^{(1)})\rangle}-e^{\i\langle t_{i},\R_{i}(X_{i}^{(1)})\rangle}\bigg|\\
						&
						\leq \frac{1}{n}\sum_{u=1}^{n} \left|\langle t_{i},R_{\mu_i}(X_{i}^{(u)})-\R_{i}(X_{i}^{(u)})\rangle \right| + \left|\langle t_{i},R_{\mu_i}(X_{i}^{(1)})-\R_{i}(X_{i}^{(1)})\rangle\right|\\
						&
						\leq \frac{1}{n}\sum_{u=1}^{n}\|t_{i}\|\|R_{\mu_i}(X_{i}^{(u)})-\R_{i}(X_{i}^{(u)})\|+\|t_{i}\|\|R_{\mu_i}(X_{i}^{(1)})-\R_{i}(X_{i}^{(1)})\| \\ 
						& \rightarrow 0 , 
					\end{align*} 
					by \cite[Theorem 2.1]{deb2021multivariate}. 
					Moreover, since $Z_{S, n}( \bm t, 1, i)$ and $\bar Z_{S, n}^{\o}( \bm t, 1, i)$ are uniformly bounded in $n$, by the dominated convergence theorem, 
$$ \lim_{n \rightarrow \infty} \E \left| Z_{S, n}( \bm t, 1, i)  - \bar Z_{S, n}^{\o}( \bm t, 1, i) \right| = 0. $$ This implies 
	\begin{align*}
        & \left|\E[ Z_{S, n}( \bm t, 1, i)  \bar Z_{S, n}^{\o}( \bm t, 1, i) ] - \E[ Z_{S, n}( \bm t, 1, i)  \bar Z_{S, n}( \bm t, 1, i) ] \right| \nonumber \\ 
	& \leq \E\left[ \left|  \bar Z_{S, n}^{\o}( \bm t, 1, i)  \right| \left|  \bar Z_{S, n}^{\o}( \bm t, 1, i) ] - \bar Z_{S, n}( \bm t, 1, i)  \right| \right] \rightarrow 0 , 
	\end{align*}
as $n \rightarrow \infty$. Hence, from \eqref{eq:ZSnt4} and as in the proof of  \eqref{eq:S1}, 
					\begin{align*}
						\lim_{n\rightarrow\infty} C_n = \lim_{n \rightarrow \infty }\prod_{i\in S} \E[  Z_{S, n}( \bm t, 1, i)  \bar Z_{S, n}^{\o}( \bm t, 1, i) ]  = \lim_{n \rightarrow \infty } \prod_{i\in S} \E[ Z_{S, n}( \bm t, 1, i)  \bar Z_{S, n} ( \bm t, 1, i) ]
						= C_S(\bm t, \bm t) . 
					\end{align*} 
					Similarly, we can show that $\lim_{n\rightarrow\infty} D_n= C_S(\bm t, \bm t)$. Recalling \eqref{eq:ZSndifference} the proof of Proposition \ref{ppn:ZnS} follows.

					\subsection{Proof of Theorem \ref{thm:mixture}}
					\label{sec:mixturepf}

					As the in the proof of Theorem \ref{thm:H0} denote $\kappa:=K|\mathcal{T}|$ and (recall \eqref{eq:KS}) 
						\begin{align*}
							{\bm \cZ}_n:=   \left( \left( Z_{S, n}(\bm t_S^{(1)}),   Z_{S, n}(\bm t_S^{(2)}), \ldots,  Z_{S, n}(\bm t_S^{(K)} ) \right) \right)_{S \in \cT} .
						\end{align*}
						Define $\delta_n := h/\sqrt n$ and 
						\begin{align}\label{eq:loglikelihood}
							V_n:=\sum_{a=1}^n\log\left(\frac{f(\bm X_a)}{\prod_{i=1}^rf_{X_i}(X_i^{(a)})}\right) = \sum_{a=1}^n\log\left(\frac{(1-\delta_n) \prod_{i=1}^rf_{X_i}(X_i^{(a)}) + \delta_n g(\bm X_a)}{\prod_{i=1}^rf_{X_i}(X_i^{(a)})}\right), 
							\end{align} 
							the log-likelihood ratio of the data under for the hypothesis \eqref{eq:H0mixture}.  To obtain the limiting distribution of $\sqrt{n} {\bm \cZ}_n$ under $H_1$ as in \eqref{eq:H0mixture}, we invoke Le Cam's third lemma (see \cite[Example 6.7]{vdv2000asymptotic}), which entails deriving joint distribution of $(\sqrt{n}\bm \cZ_n, V_n)$ under $H_0$. Note that by Proposition \ref{ppn:ZnS}, it is enough to derive the joint distribution of $(\sqrt{n}\bm \cZ_n^{\o},V_n)$ under $H_0$. Also, it can be shown that 
						\begin{align}\label{eq:ZSnoraclet}
							\sqrt{n}Z_{S,n}^{\o}(\bm t)-\sqrt{n}\tilde{Z}_{S,n}(\bm t)\stackrel{P}{\longrightarrow}0 , 
						\end{align} 
						where 
						\begin{align}\label{eq:ZSnt}
						\sqrt{n}\tilde{Z}_{S,n}(\bm t):=\frac{1}{\sqrt{n}}\sum_{u=1}^n\prod_{i\in S}\left\{\E\left[e^{\i \langle t_i, R_{\mu_i}(X_i)\rangle}\right]-e^{\i \langle t_i, R_{\mu_i}(X_i^{(u)})\rangle}\right\} . 
						\end{align} 
Therefore, we only need to obtain the limiting distribution $(\sqrt{n}\tilde{\bm{\cZ}}_n,V_n)$ under $H_0$, where 
\begin{align}\label{eq:KSZt}
\tilde{\bm{\cZ}}_n:=((\tilde{Z}_{S,n}(\bm t_{S}^{(1)}),\ldots,\tilde{Z}_{S,n}(\bm t_{S}^{(K)})))_{S\in\mathcal{T}}. 
\end{align} To this end, by a Taylor's expansion around $\delta=0$ in \eqref{eq:loglikelihood} gives,  
						\begin{align*}
							V_n=\dot{V}_n+o_P(1), 
						\end{align*} 
						where 
		$$\dot{V}_n:=\frac{h}{\sqrt{n}}\sum_{a=1}^{n} \left(\frac{g(\bm X_a)}{\prod_{i=1}^{r}f_{X_i}(X_{i}^{(a)})}-1 \right)-\frac{h^{2}}{2}\E\left[\frac{g(\bm X)}{\prod_{i=1}^{r}f_{X_i}(X_{i})}-1\right]^{2}.$$ 
						Therefore, under $H_0$, by the multivariate central limit theorem, 						\begin{align}\label{eq:ZVn}
							(\sqrt{n}\tilde{\bm{\cZ}}_n, \dot{V}_n)\stackrel{D}{\rightarrow}N_\kappa \left(\begin{pmatrix}
							\bm 0 \\ 
							-\tfrac{1}{2} \gamma  
							\end{pmatrix} , \begin{pmatrix}
								\bm \cD & \bm \mu \\
								\bm{\mu}^{\top} & \gamma \\
							\end{pmatrix} \right), 
				\end{align}
				where 
				$\gamma := h^{2}\E\left[\frac{g(\bm X)}{\prod_{i=1}^{r}f_{X_i}(X_{i})}-1\right]^{2}$, 
%
%
						$\bm \cD$ is $\kappa \times \kappa$ a block-diagonal matrix with elements $C_S(\bm t^{(\ell)}, \bm t^{(\ell')})$ for $\ell, \ell' \in \{1, 2, \ldots, K\}$ and $S \in \cT$, and 						
						\begin{align*}
							\bm \mu=((\mu_{S}(\bm t^{(1)}), \mu_{S}(\bm t^{(2)}), \ldots, \mu_{S}(\bm t^{(K)})))_{S\in\mathcal{T}} ,  
							\end{align*} 
							with $\mu_S(\cdot)$ is as defined in \eqref{eq:meanH1}. 

							 Then by Proposition \ref{ppn:ZnS} and \eqref{eq:ZSnoraclet}, under $H_0$ $(\sqrt{n}\tilde{\bm{\cZ}}_n, \dot{V}_n)$ converges to the same limit as in \eqref{eq:ZVn}. Hence, by Le Cam's third lemma (\cite[Example 6.7]{vdv2000asymptotic}), under $H_1$					
						\begin{align*}
							\sqrt{n}\bm \cZ_n \stackrel{D}{\rightarrow}N_\kappa(\bm \tau, \bm{\cD}) \stackrel{D} = \bm{\cZ} + \bm \mu ,
						\end{align*} 
						for $\bm{\cZ}$ as defined in \eqref{eq:KSZ}. 
						This establishes the finite dimensional convergence of the process $\sqrt{n}Z_{S,n}(\bm t)$ under $H_1$. Finally, since $H_0$ and $H_1$ are mutually contiguous, the condition in \eqref{eq:equicontinuity} holds under $H_1$ as well. Hence, the process $\{ \sqrt{n}Z_{S,n}(\bm t) : S \in \cT \} $ converges weakly to $\{ Z_{S}(\bm t) + \mu_S(\bm t) : S \in \cT\}$ and the result in \eqref{eq:H0NRdCovS} follows by the continuous mapping theorem. This completes the proof of Theorem \ref{thm:mixture}. \hfill $\Box$

					\subsection{Proof of Theorem \ref{thm:powerK}}
					\label{sec:powerKpf}
					
					To show Theorem \ref{thm:powerK}, we will first show the contiguity between $H_0$ and $H_1$ under the given assumptions. For this, denote by $\log(L_n)$ the  likelihood ratio for the hypothesis \eqref{eq:KH0}, 
		$$ W_n=2\sum_{i=1}^{n}\bigg\{\frac{f_{\bm X}^{\frac{1}{2}}(\bm X_i|\delta_n)}{f_{\bm X}^{\frac{1}{2}}(\bm X_i|0)}-1\bigg\} \text{ and } T_n=\delta_n\sum_{i=1}^{n}L'(\bm X_i|0).$$
where $\delta_n=h/\sqrt{n}$. To show the contiguity between $H_0$ and $H_1$, we will use Le Cam's second lemma \cite[Corollary 12.3.1]{lr}. The following lemma verifies the conditions required for applying Le Cam's second lemma.

					
					\begin{lemma}\label{lm:KH0}
						Denote $\delta_n=h/\sqrt{n}$ and suppose Assumption \ref{assumption:K}. Then under $H_0$ the following conditions hold: 
						
						\begin{itemize} 
						
						\item[$(a)$] $W_n=T_n-\tfrac{1}{4} \gamma+o_P(1)$, 						where $\Var_{H_0}[T_n]=\gamma<\infty$. 
						
						\item[$(b)$] The summands of likelihood ratio are uniformly asymptotic negligible, that is, 
						\begin{align*}
							\P_n\bigg(\bigg|\frac{f_{\bm X}(\bm X_i|\delta_n)}{f_{\bm X}(\bm X_i|0)}-1\bigg|>\varepsilon\bigg)\rightarrow0.
						\end{align*}
						
						
						\end{itemize} 
						
					\end{lemma}

					The proof of Lemma \ref{lm:KH0} is given in Section \ref{sec:KH0pf}. First we apply it to complete the proof of Theorem \ref{thm:powerK}. To this end, note that since $T_n \stackrel{D} \rightarrow N(0, \gamma)$, by Lemma \ref{lm:KH0} (a) 
					$W_n \stackrel{D} \rightarrow N(-\frac{1}{4} \gamma, \gamma)$. Hence, by Lemma \ref{lm:KH0} and Le Cam's second lemma,  the Konjin local alternative $H_1$ as in \eqref{eq:KH0} are contiguous to $H_0$. Also, by Le Cam's second lemma we know that $\log(L_n) - (T_n- \frac{1}{2} \gamma) = o_P(1)$. Therefore, as in \eqref{eq:ZVn} under $H_0$, by the multivariate central limit theorem, 						\begin{align}\label{eq:ZVn2}
							( \sqrt{n}\tilde{\bm{\cZ}}_n, \log(L_n) )\stackrel{D}{\rightarrow}N_\kappa \left(\begin{pmatrix}
							\bm 0 \\ 
							-\tfrac{1}{2} \gamma  
							\end{pmatrix} , \begin{pmatrix}
								\bm \cD & \bm \nu \\
								\bm{\nu}^{\top} & \gamma \\
							\end{pmatrix} \right), 
				\end{align}
				where $\tilde{\bm{\cZ}}_n$ is as in \eqref{eq:KSZt},  $\bm {\cD}$ as in \eqref{eq:ZVn} and 
								\begin{align*}
							\bm \nu=((\nu_{S}(\bm t^{(1)}), \nu_{S}(\bm t^{(2)}), \ldots, \nu_{S}(\bm t^{(K)})))_{S\in\mathcal{T}} ,  
							\end{align*} 
							with $\nu_S(\cdot)$ is as defined in \eqref{eq:meanKH1}. The rest of proof can now be completed by arguments similar to the proof of Theorem \ref{thm:mixture}.

					\subsubsection{Proof of Lemma \ref{lm:KH0}}
					\label{sec:KH0pf} 
					
					We begin by computing the derivative of $f_{\bm X}(\bm x|\delta)$ (as defined in \eqref{eq:densitydelta}) with respect to $\delta$. This is stated in the following lemma which follows by a direct computation. 
					
					\begin{lemma}\label{lem: derivative_L}
						Consider the family of distributions $\{f_{\bm X}(\bm x|\delta)\}_{\delta \in \Theta}$ as in \eqref{eq:densitydelta}. Then 
						\begin{align*}
							\frac{\partial}{\partial \delta} f_{\bm X}(\bm x|\delta) =-\frac{1}{|\det(\bm A_{\delta})|}\mathrm{tr}(-\bm A_{\delta}^{-1}P)f_{\bm X}(\bm A_{\delta}^{-1}\bm x|0)+\frac{1}{|\det(\bm A_{\delta})|}(\bm A_{\delta}^{-1}P\bm A_{\delta}^{-1}\bm x)^{\top} \nabla f_{\bm X}(\bm x|0).
						\end{align*} 
						where 
						$$P=- \frac{\partial}{\partial \delta} \bm A_{\delta}\Big|_{\delta=0} = \begin{pmatrix}
					\bm I_{d_1} & - \bm M_{1, 2} & \cdots & - \bm M_{1, r} \\
					- \bm M_{2, 1} & \bm I_{d_2} & \cdots & - \bm M_{2, r} \\ 
					\vdots & \vdots & \ddots & \vdots \\ 
					- \bm M_{r, 1}^\top & -  \bm M_{r, 2} & \cdots &  \bm I_{d_r} \\ 
				\end{pmatrix} ,					
						$$ 
						and $\nabla f_{\bm X}(\bm x|0) = \frac{\partial }{\partial \delta} f_{\bm X}(\bm x|\delta)|_{\delta=0}$. 
						\end{lemma}

We begin with the proof of Lemma \ref{lm:KH0} (a). From Assumption \ref{assumption:K} it follows that $\Var_{H_0}[T_n]=h^{2} \E_{H_0}[(L'(\bm X|0))^2]<\infty$. To prove $W_n=T_n-\tfrac{1}{4} \gamma+o_P(1)$, it suffices to show that
\begin{align}\label{eq:Wn}
\E_{H_0}[W_n]\rightarrow- \tfrac{1}{4} \gamma \quad \text{and} \quad \Var_{H_0}[W_n-T_n]\rightarrow0, 
\end{align} 
as $n\rightarrow\infty$. To this end, denote $s(\bm x)=f^{\frac{1}{2}}_{\bm X}(\bm x|0)$. Then $\nabla s(\bm x)=(\nabla f_{\bm X}(\bm x|0))/(2f_{\bm X}^{\frac{1}{2}}(\bm x|0))$ and 	\begin{align*}
	\E_{H_0}[W_n] =2n\E_{H_0}[L^{\frac{1}{2}}(\bm X|\delta_n)]-2 
						&
						=2n\E_{H_0}\left[\frac{|\bm A_{\delta_n}|^{-\frac{1}{2}}s(\bm A_{\delta_n}^{-1}\bm X)}{s(\bm X)}\right]-2\\
						&
						=-h^{2}\int \bigg\{\frac{|\bm A_{\delta_n}|^{-\frac{1}{2}}s(\bm A_{\delta_n}^{-1}\bm x)-s(\bm x)}{\delta_n}\bigg\}^{2}\d x.
					\end{align*}
					To compute the limit of the integral in the RHS above we first show that given the assumptions, the quadratic mean differentiability (QMD) condition is satisfied for $f_{\bm X}(\bm x|\delta)$ around $\delta=0$. It is easy to verify that $f_{\bm X}(\bm x|\delta)$ is continuously differentiable with respect to $\delta$ around the neighborhood of $\delta=0$ for any $\bm x\in\supp\{f_{\bm X}(\bm x|0)\}$. This can be easily seen from the expression of $\partial f/\partial\delta$ in Theorem \ref{lem: derivative_L} since $f_{\bm X}(\bm x|0)$ is continuously differentiable by the assumption. Combined with the Assumption \ref{assumption:K} of the well-definedness and continuity of $\E_{H_0}[(L'(\bm x|\delta))^2]$ around $\delta=0$ we know the QMD assumption is satisfied around $\delta=0$. Then by Taylor's expansion, 
					\begin{align*}
						\frac{|\bm A_{\delta_n}|^{-\frac{1}{2}}s(\bm A_{\delta_n}^{-1}\bm x)-s(\bm x)}{\delta_n}=\frac{1}{2}L'(\bm x|0)s(\bm x) + \frac{r_n}{\delta_n}
					\end{align*}
					where $r_n=o(\delta_n^{2})$. By QMD, we know $\int r_n^{2}\d x=o(\delta_n^{2})$. Therefore we have
					\begin{align*}
						&
						\bigg|\int \bigg\{\frac{|\bm A_{\delta_n}|^{-\frac{1}{2}}s(\bm A_{\delta_n}^{-1}\bm x)-s(\bm x)}{\delta_n}\bigg\}^{2}\d x - \int \bigg\{\frac{1}{2}L'(\bm x|0)s(\bm x)\bigg\}^{2}\d x\bigg|\\
						&
						=\bigg|2\int\bigg\{\frac{1}{2}L'(\bm x|0)s(\bm x)\bigg\}\frac{r_n}{\delta_n}\d x+\int\frac{r_n^{2}}{\delta_n^{2}}\d x\bigg|\\
						&
						\leq 2\sqrt{\int \bigg\{\frac{1}{2}L'(\bm x|0)s(\bm x)\bigg\}^{2}\d x\int\frac{r_n^{2}}{\delta_n^{2}}\d x}+\int\frac{r_n^{2}}{\delta_n^{2}}\d x\\
						&
						\rightarrow 0 , 
					\end{align*}
					using the third part of Assumption \ref{assumption:K}. This implies, 					\begin{align}\label{eq:change_limit}
						\lim_{n\rightarrow\infty}\E_{H_0}[W_n]=-h^{2}\int \bigg\{\frac{1}{2}L'(\bm x|0)s(\bm x)\bigg\}^{2}\d x & =-\frac{h^{2}}{4}\E_{H_0}\left[(L'(\bm X|0))^{2}\right] \nonumber \\ 
						& = -\frac{h^{2}}{4}\mathrm{Var}_{H_0}\left[(L'(\bm X|0))^{2}\right] = \frac{1}{4} \gamma , 
					\end{align} 
					where the second last inequality uses $E_{H_0}\left[L'(\bm X|0)\right]=0$, by the QMD condition. This completes the proof of the first condition in \eqref{eq:Wn}. To show the second condition in \eqref{eq:Wn} note that 
					\begin{align*}
						\Var_{H_0}[W_n-T_n]
						&
						=4h^2\Var_{H_0}\left[\frac{|\bm A_{\delta_n}^{-1}|s(\bm A_{\delta_n}^{-1}\bm X)}{\delta_n s(\bm X)}-\frac{1}{\delta_n}-\frac{1}{2}L'(\bm X|0)\right]\\
						&
						\leq 4h^2\E_{H_0}\left[\left(\frac{|\bm A_{\delta_n}^{-1}|s(\bm A_{\delta_n}^{-1}\bm X)}{\delta_n s(\bm X)}-\frac{1}{\delta_n}-\frac{1}{2}L'(\bm X|0)\right)^{2}\right]\\
						&
						=4h^2\int \bigg\{\frac{|\bm A_{\delta_n}^{-1}|s(\bm A_{\delta_n}^{-1}\bm x)-s(\bm x)}{\delta_n}-\frac{1}{2}L'(\bm x|0)s(\bm x)\bigg\}^{2}\d x.
					\end{align*}
					Following the same idea as in proving \eqref{eq:change_limit}, we have by QMD condition					\begin{align*} 
						\lim_{n\rightarrow\infty}\Var_{H_0}[W_n-T_n] 
						& \leq 4h^{2}\lim_{n\rightarrow\infty}\int \bigg\{\frac{|\bm A_{\delta_n}^{-1}|s(\bm A_{\delta_n}^{-1}\bm x)-s(\bm x)}{\delta_n}-\frac{1}{2}L'(\bm x|0)s(\bm x)\bigg\}^{2}\d x\\
						&
						= \lim_{n\rightarrow\infty}4h^{2}\int\frac{r_n^{2}}{\delta_n^{2}}\d x\\
						&
						=0.
					\end{align*} 
					This establishes the second condition in \eqref{eq:Wn}, and hence, the completes proof of Lemma \ref{lm:KH0} (a). 
%
					
					Now, we prove Lemma \ref{lm:KH0} (b). Define $Z:=L'(\bm X|0)$ and
					\begin{align*}
						Z_n:=\frac{f_{\bm X}(\bm X|\delta_n)-f_{\bm X}(\bm X|0)}{\delta_n f_{\bm X}(\bm X|0)}.
					\end{align*} 
%
%
					 By a Taylor's expansion, 
					\begin{align*}
						f^{\frac{1}{2}}_{\bm X}(\bm x|\delta_n)=f^{\frac{1}{2}}_{\bm X}(\bm x|0)+\delta_n \frac{1}{2}L'(\bm x|0) f^{\frac{1}{2}}_{\bm X}(\bm x|0)+r_n , 
					\end{align*}
					where $r_n=o(\delta_n^{2}),\int r_n^2/\delta_n^2\mathrm{d}x=o(1)$. This implies, 
					\begin{align*}
						\E_{H_0}[|Z_n|] & \leq \int \bigg|\frac{f_{\bm X}(\bm x|\delta_n)-f_{\bm X}(\bm x|0)}{\delta_n}\bigg|\d x \nonumber \\ 
						&
						\leq\int \bigg|\delta_n\bigg\{\frac{1}{2}L'(\bm x|0) f^{\frac{1}{2}}_{\bm X}(\bm x|0)\bigg\}^{2}\bigg|\d x +2\int \bigg|r_n\frac{1}{2}L'(\bm x|0)f^{\frac{1}{2}}_{\bm X}(\bm x|0)\bigg|\d x\\
						&
						\qquad+2\int\left|\frac{1}{2}f_{\bm X}(\bm x|0)L'(\bm x|0)\right|\mathrm{d}x+2\int\bigg|\frac{r_n}{\delta_n} f^{\frac{1}{2}}_{\bm X}(\bm x|0)\bigg|d x+\int \bigg|\frac{r_n^{2}}{\delta_n}\bigg|\d x.
					\end{align*} 
					Then for $n$ enough and the Cauchy-Schwarz inequality gives, 					\begin{align*}
						\E_{H_0}[|Z_n|]
						&
						\leq \E_{H_0}\left[|L'(\bm x|0)|\right]+o(1)<\infty.
					\end{align*} 
					Therefore, 
					\begin{align*}
						\P\bigg(\bigg|\frac{f_{\bm X}(\bm X_i|\delta_n)}{f_{\bm X}(\bm X_i|0)}-1\bigg|>\varepsilon\bigg)=\P\bigg(\bigg| \delta_n Z_n\bigg|>\varepsilon\bigg)\leq\frac{h\E(|Z_n|)}{\varepsilon\sqrt{n}}\rightarrow0.
					\end{align*}
					This completes the proof of Lemma \ref{lm:KH0} (b). \hfill $\Box$ 
%
%
%

					\section{Proof of results in Section \ref{sec:ICA}}
					
					This section is organized as follows: In Section \ref{sec:gradICA} we compute the gradient for the objective in \eqref{eq:kernelICA}. The proof of  Theorem \ref{thm:independentcomponent} is given in Section \ref{sec:icapf}.

					\subsection{Gradient Computation of $\RJdCov_n^2$}\label{sec:gradICA}

					In practice, we use a gradient descent algorithm to solve the optimization problem in \eqref{eq:kernelICA}. To compute the gradient of the objective function we use the Chain rule, 
		            \begin{align*}
			        \frac{\partial \JdCov^2_n(\bm{\tilde{F}}(\bm Z_n \bm W(\theta)^{\top});c)}{\partial \tilde{F}_i(W_i(\theta)^{\top}Z_a)}=\frac{1}{n^2}\sum_{b=1}^n\prod_{i'\neq i}^r\big\{\tilde{W}_{i'}(a,b) + c\big\}\frac{\partial \tilde{\cE}_i(a,b)}{\partial \tilde{F}_i(W_i(\theta)^{\top}Z_a)} , 
		            \end{align*}
		            where 
		            $$\bm{\tilde{F}}(\bm Z_n \bm W(\theta)^{\top})=(\tilde{F}_1(\bm ZW_1(\theta)),\ldots,\tilde{F}_d(\bm ZW_d(\theta))
$$ and 
		            $ \tilde{\cE}_i(a,b)$ is  obtained after replacing $\R_i(X_{i}^{(a)})$ with $\tilde{F}_i(W_i(\theta)^{\top}Z_a)$ and $\R_i(X_{i}^{(b)})$ with $\tilde{F}_i(W_i(\theta)^{\top}Z_b)$ in the expression of $\hat{\cE}_{i}(a,b)$ (recall \eqref{eq:estimateW}). Moreover, $\partial \tilde{\cE}_i(a,b)/\partial \tilde{F}_i(W_i(\theta)^{\top}Z_a)$ can be approximated by omitting $O(1/n)$ term as follows: 
		            \begin{align*}
			        \frac{1}{n}\sum_{v=1}^n\mathrm{sign}(\tilde{F}_i(W_i(\theta)^{\top}Z_a)-\tilde{F}_i(W_i(\theta)^{\top}Z_v))-\mathrm{sign}(\tilde{F}_i(W_i(\theta)^{\top}Z_a)-\tilde{F}_i(W_i(\theta)^{\top}Z_b)), 
		            \end{align*} 
		            for all $\theta\in\bar{\Theta}$, where $\mathrm{sign}(x)$ equals 1 if $x>0$, -1 if $x<0$, 0 otherwise. In the optimization procedure, suppose the last updated value is $\theta^{(k-1)}$, then we are able to compute
		            \begin{align*}
			        \frac{\partial \tilde{F}_i(W_i(\theta)^{\top}Z_a)}{\partial W_i(\theta)}\bigg|_{\theta=\theta^{(k-1)}}=\frac{1}{n}\sum_{v=1}^nG'\bigg(\frac{W_i^{\top}(\theta^{(k-1)})Z_a-W_i^\top(\theta^{(k-1)})Z_v}{h_{n}(i)}\bigg)\frac{Z_a}{h_{n}(i)} , 
		            \end{align*}
where $G'$ is the derivate of the function $G$ in \eqref{eq:cdfG}.

					\subsection{Proof of Theorem \ref{thm:independentcomponent}}\label{sec:icapf}
					
					 Throughout this proof we will drop the dependence on $c$ from the notations of $\RJdCov^2/\RJdCov^2_n$. First we will show that the distribution function of $W_i(\theta)^{\top}Z$ is Lipschitz by showing the density is uniformly bounded. 
					This is because of the convolution formula for the sum of independent random variables. For example, if $X$ and $Y$ are 2 independent random variables then the density of $X+Y$ is given by, 
						\begin{align*}
							f_{X+Y}(z)=\int_{-\infty}^{\infty}f_{Y}(z-x)f(x)\d x\leq \sup_{y}f_{Y}(y).
						\end{align*}
						Therefore, we know the density of $W_i(\theta)^{\top}Z$ is uniformly bounded, hence the distribution function of $W_i(\theta)^{\top}Z$ is Lipschitz. 
						
						
						Now suppose $\cD$ is a  metric satisfying the conditions in Theorem \ref{sec:icapf}. Partition $\SO(r)$ into equivalence classes by identifying elements with $\mathcal{D}$-distance zero with the same equivalence class.  
						Let $\SO(r)_{\mathcal{D}}$ be the quotient space $\SO(r)/\mathcal{D}$ of these equivalence classes. Then $\bm W=\bm A$ on $\SO(r)_{\mathcal{D}}$ if and only if $\mathcal{D}(\bm W, \bm A)=0$. To establish the consistency of $\hat \theta_n$ we will use the following result about the consistency of $M$-estimators. 
												
						\begin{theorem}[Theorem 2.12 in \cite{kosorok2008introduction}]\label{thm:consitency_estimator}
							Assume for some $\theta_0\in \Theta$ that $\liminf_{n\rightarrow\infty}M(\theta_n)\geq M(\theta_0)$ implies $d(\theta_n,\theta_0)\rightarrow0$ for any sequence $\{\theta_n\}\in\Theta$. Then for a sequence of estimator $\hat{\theta}_n\in\Theta$, if $M_n(\hat{\theta}_n)=\sup_{\theta\in\Theta}M_n(\theta)-o(1)$ and $\sup_{\theta\in\Theta}|M_n(\theta)-M(\theta)|\stackrel{a.s.}{\rightarrow}0$, then $d(\hat{\theta}_n,\theta_0)\stackrel{a.s.}{\rightarrow}0$.
						\end{theorem} 
						
						We will apply the above result with $M_n= \RJdCov^2_n$ and $M= \RJdCov^2$. We begin by proving the uniform convergence of the objective functions in the following lemma. 
						
						\begin{lemma}\label{lem:RWthetaZ}
							Under Assumptions in Theorem \ref{thm:independentcomponent},  as $n\rightarrow\infty$, 
							\begin{align*}
								\sup_{\theta:W(\theta)\in \SO(r)_{\mathcal{D}}}\big|\RJdCov^{2}_n(\bm Z_n \bm W(\theta)^{\top})-\RJdCov^{2}(\bm W(\theta)Z)\big|\stackrel{a.s.}{\rightarrow}0 . 
							\end{align*}
						\end{lemma} 
						
						\begin{proof}[Proof of Lemma \ref{lem:RWthetaZ}] 
						The proof of this lemma relies on the stochastic Arzel\'a-Ascoli lemma.  
%
%
%
To this end, note that since $\SO(r)_{\mathcal{D}}$ is a quotient space so it is compact. We have also shown in Theorem \ref{thm:consistency} that $$\RJdCov^{2}_n(\bm Z_n \bm W(\theta)^{\top})\stackrel{a.s.}{\rightarrow}\RJdCov^{2}(\bm W(\theta)Z),$$
for all $\theta$. Therefore, to apply stochastic Arzel\'a-Ascoli lemma we need check the condition for stochastic equicontinuity, that is, we need to prove 							\begin{align*}
								\lim_{c\rightarrow\infty}\limsup_{n\rightarrow\infty}m_{1/c}(\RJdCov^{2}_n)\stackrel{a.s.}{\rightarrow}0 , 
							\end{align*}
							where 
							\begin{align*}
								m_{1/c}(\RJdCov^{2}_n)=\sup\{|\RJdCov^{2}_n & (\bm Z_n \bm W(\theta)^{\top})-\RJdCov^{2}_n(\bm Z_n \bm W(\psi)^{\top})|: \\ 
								&
								\bm W(\theta), \bm W(\psi)\in\SO(r)_{\mathcal{D}}, \|\bm W(\theta) - \bm W(\psi)\|_{F}<1/c\}.
							\end{align*}
							 Consider
							\begin{align*}
								&|\RJdCov^{2}_n(\bm Z_n \bm W(\theta)^{\top})-\RJdCov^{2}_n(\bm Z_n \bm W(\psi)^{\top})|\\
								&
								\leq c^{d-2}\sum_{1 \leq i_1 < i_2 \leq r}\bigg|\RdCov^{2}_n(Z_{i_1}^\top W_{i_1}(\theta), Z_{i_2}^\top W_{i_2}(\theta))-\RdCov^{2}_n( Z_{i_1}^\top W_{i_1}(\psi), Z_{i_2}^\top W_{i_2}(\psi))\bigg|\\
								&\qquad\qquad\qquad+\cdots+\big|\RdCov^{2}_n(\bm Z_n \bm W(\theta)^{\top})-\RdCov^{2}_n(\bm Z_n \bm W(\psi)^{\top})\big|.
							\end{align*} 
							Therefore, we only need to prove that 
							\begin{align*}
								\lim_{c\rightarrow\infty}\limsup_{n\rightarrow\infty}m_{1/c}(\RdCov^{2}_n(\bm Z_n \bm W^\top_{\mathcal{S}}))=0 , 							\end{align*} 
								for all $\mathcal{S}\subset\{1, \ldots, r\}$, 
							where $\bm W_\mathcal{S}$ is the row submatrix of $\bm W$ with rows containing in $\mathcal{S}$ and
							\begin{align*}
								m_{1/c}(\RdCov^{2}_n(\bm Z_n \bm W^\top_{\mathcal{S}})) & =\sup\{|\RdCov^{2}_n(\bm Z_n \bm W_{\mathcal{S}}(\theta)^{\top})-\RdCov^{2}_n(\bm Z_n \bm W_{\mathcal{S}}(\psi)^{\top} )|: \\ 
								& \hspace{0.85in} \bm W(\theta), \bm W(\psi)\in\SO(r)_{\mathcal{D}},\| \bm W(\theta)-\bm W(\psi)\|_{F}<1/c\}.
							\end{align*} 
							Now, consider the telescoping sum 
							\begin{align*}
								&|\RdCov^{2}_n(\bm Z_n \bm W_{\mathcal{S}}(\theta)^{\top})-\RdCov^{2}_n(\bm Z_n \bm W_{\mathcal{S}}(\psi)^{\top})|\\ 
								&
								=\bigg|\frac{1}{n^{2}}\sum_{1 \leq a, b \leq n}\prod_{i\in \mathcal{S}}\hat{\cE}_i(a,b;\theta)-\frac{1}{n^{2}}\sum_{1 \leq a, b \leq n}\prod_{i\in\mathcal{S}}\hat{\cE}_i(a,b;\psi)\bigg|\\
								&
								= \bigg|\frac{1}{n^{2}}\sum_{1 \leq a, b \leq n}\prod_{i\in \mathcal{S}}\hat{\cE}_i(a,b;\theta)-\frac{1}{n^{2}}\sum_{1 \leq a, b \leq n}\hat{\cE}_{i_1}(a, b;\psi)\prod_{i\in \mathcal{S}\backslash \{i_1\}}\hat{\cE}_i(a,b;\theta)\\
								&
								+\frac{1}{n^{2}}\sum_{1 \leq a, b \leq n}\hat{\cE}_{i_1}(a, b; \psi)\prod_{i\in \mathcal{S} \backslash \{i_1\}}\hat{\cE}_i(a,b;\theta)-\frac{1}{n^{2}}\sum_{1 \leq a, b \leq n}\hat{\cE}_{i_1}(a,b;\psi)\hat{\cE}_{i_2}(a,b;\psi)\prod_{i\in \mathcal{S}\backslash\{i_1,i_2\}}\hat{\cE}_i(a,b;\theta)\\
								&
								+\cdots+\\
								&
								+\frac{1}{n^{2}}\sum_{1 \leq a, b \leq n}\hat{\cE}_{i_{|\mathcal{S}|}}(a,b;\theta)\prod_{i\in \mathcal{S}\backslash\{i_{|\mathcal{S}|}\}}\hat{\cE}_i(a,b;\psi)-
								\frac{1}{n^{2}}\sum_{1 \leq a, b \leq n}\prod_{i\in \mathcal{S}}\hat{\cE}_i(a,b;\psi)\bigg|\\
								&
								\leq C\sum_{i=1}^{r}\bigg|\frac{1}{n^{2}}\sum_{1 \leq a, b \leq n}\big\{\hat{\cE}_i(a,b;\theta)-\hat{\cE}_i(a,b;\psi)\big\}\bigg|
							\end{align*}
							where $C$ is a universal constant and the last inequality holds due to the fact that $\hat{\cE}_i(a, b)$ is uniformly bounded. Then we consider
							\begin{align*}
								&
								\frac{1}{n^{2}}\sum_{1 \leq a, b \leq n}\big\{\hat{\cE}_i(a,b;\theta)-\hat{\cE}_i(a,b;\psi)\big\}\\
								&
								\leq \frac{1}{n^{2}}\sum_{1 \leq a, b \leq n}\frac{1}{n}\sum_{v=1}^{n}\big|\hat{R}_i(W_{i}^{\top}(\theta)Z_a)-\hat{R}_i(W_{i}^{\top}(\psi)Z_a)+\hat{R}_i(W_i^{\top}(\psi)Z_v)-\hat{R}_i(W_i(\theta)^{\top}Z_v)\big|\\
								&
								+\frac{1}{n^{2}}\sum_{1 \leq a, b \leq n}\frac{1}{n}\big|\hat{R}_i(W_{i}^{\top}(\theta)Z_a)-\hat{R}_i(W_{i}^{\top}(\psi)Z_a)+\hat{R}_i(W_{i}^{\top}(\psi)Z_b)-\hat{R}_i(W_{i}^{\top}(\theta)Z_b)\big|\\
								&
								+\frac{1}{n^{2}}\sum_{1 \leq a, b \leq n}\frac{1}{n}\sum_{u=1}^{n}\big|\hat{R}_i(W_{i}^{\top}(\theta)Z_b)-\hat{R}_i(W_{i}^{\top}(\psi)Z_b)+\hat{R}_i(W_i^{\top}(\psi)Z_u)-\hat{R}_i(W_i(\theta)^{\top}Z_u)\big|\\
								&
								+\frac{1}{n^2}\sum_{u,v=1}^{n}\big|\hat{R}_i(W_{i}^{\top}(\theta)Z_u)-\hat{R}_i(W_{i}^{\top}(\psi)Z_u)+\hat{R}_i(W_{i}^{\top}(\psi)Z_v)-\hat{R}_i(W_{i}^{\top}(\theta)Z_v)\big|\\
								&
								\leq \frac{8}{n}\sum_{a=1}^{n}\big|\hat{R}_i(W_i(\theta)^{\top}Z_a)-R_{\mu_i}(W_i(\theta)^{\top}Z_a)\big|+\frac{8}{n}\sum_{a=1}^{n}\big|\hat{R}_i(W_i^{\top}(\psi)Z_a)-R_{\mu_i}(W_i^{\top}(\psi)Z_a)\big|\\
								&
								+ \frac{8}{n}\sum_{a=1}^{n}\big|R_{\mu_i}(W_i(\theta)^{\top}Z_a)-R_{\mu_i}(W_i^{\top}(\psi)Z_a)\big|.
							\end{align*}
							As $n\rightarrow\infty$, we know the first and the second term goes to zero by the consistency of empirical rank. Then we consider the last term by utilizing the Lipschitzness of $R_{\mu_i}(\cdot)$, since $R_{\mu_i}(\cdot)$ should be the CDF of $W_i^{\top}(\cdot) Z$. Therefore, if $0 < K < \infty$ is the bound on the density $Z$, 
							\begin{align*}
								\frac{1}{n^{2}}\sum_{1 \leq a, b \leq n}\big\{\hat{\cE}_i(a,b;\theta)-\hat{\cE}_i(a,b;\psi)\big\}
								&
								\leq \frac{8K}{n}\sum_{k=1}^{n}\big|W_i(\theta)^{\top}Z_k-W_i^{\top}(\psi)Z_k\big|\\
								&
								\leq 8K \|W_i(\theta)-W_i(\psi)\|\frac{1}{n}\sum_{k=1}^{n}\|Z_k\|.
							\end{align*}
							Then by strong law of large number we know $\frac{1}{n}\sum_{k=1}^{n}\|Z_k\|\stackrel{a.s.}{\rightarrow}\E(\|Z\|)<\infty$. Therefore we can conclude that
							\begin{align*}
								\lim_{c\rightarrow\infty}\limsup_{n\rightarrow\infty}m_{1/c}(\RdCov^{2}_n(\bm Z_n\bm W^\top_{\mathcal{S}}))=0 , 
							\end{align*} 
							for all $\mathcal{S}\subset\{1,\ldots,d\}$. This complete the proof of Lemma \ref{lem:RWthetaZ}.  
						\end{proof}

				To complete the proof of Theorem \ref{thm:independentcomponent}, recall the definition of $\theta_0$ (from \eqref{eq:minimizerICA}) and $\hat{\theta}_n$ (from \eqref{eq:thetaestimation}) and note that
						\begin{align*}
							\RJdCov^{2}_n(\bm Z_n \bm W(\hat{\theta}_n)^{\top})\leq\RJdCov_n^{2}(\bm Z_n \bm W(\theta_0)^{\top}) \text{ and } \RJdCov^{2}(\bm W(\theta_0)Z)\leq \RJdCov^{2}(\bm W(\hat{\theta}_n) Z).
						\end{align*}
						Therefore, we have
						\begin{align*}
							\RJdCov^{2}_n(\bm Z_n \bm W(\theta_0)^{\top})-\RJdCov^{2}(\bm W(\theta_0)Z)
							&
							\geq \RJdCov^{2}_n(\bm Z_n \bm W(\hat{\theta}_n)^{\top})-\RJdCov^{2}(\bm W(\theta_0)Z)\\
							&
							\geq \RJdCov^{2}_n(\bm Z_n \bm W(\hat{\theta}_n)^{\top})-\RJdCov^{2}(\bm W(\hat{\theta}_n)Z) . 
						\end{align*}
						This implies, 
						\begin{align*}
							&
							|\RJdCov_n^{2}(\bm Z_n \bm W(\hat{\theta}_n)^{\top})-\RJdCov^{2}(\bm W(\theta_0)Z)|\\
							&
							\leq \max\{|\RJdCov^{2}_n(\bm Z_n \bm W(\theta_0)^{\top})-\RJdCov^{2}(\bm W(\theta_0)Z)|,|\RJdCov^{2}_n(\bm Z_n \bm W(\hat{\theta}_n)^{\top})-\RJdCov^{2}(\bm W(\hat{\theta}_n)Z)|\}\\
							&
							\leq \sup_{\theta:W(\theta)\in \SO(r)_{\mathcal{D}}}\big|\RJdCov_n^{2}(\bm Z_n \bm W(\theta)^{\top})-\RJdCov^{2}(\bm W(\theta)Z)\big|\stackrel{a.s.}{\rightarrow}0.
						\end{align*}
						Then, since $\RJdCov^{2}(\bm W(\theta)Z)$ is a Lipschitz continuous with respect to $\theta$, Theorem \ref{thm:consitency_estimator} that $\mathcal{D}(\bm W(\hat{\theta}_n), \bm W(\theta_0))\stackrel{a.s.}{\rightarrow}0$ for $\bm W(\theta_0)\in\SO(r)_{\mathcal{D}}$. Since the map $\bm W(\theta)\mapsto \theta$ is continuous by the assumption, we conclude the result in Theorem \ref{thm:independentcomponent}. \hfill $\Box$

				\end{document}